\documentclass[12pt]{amsart}

\usepackage{times}
\usepackage{latexsym,  amsmath,amscd, amssymb, amsthm,graphicx}
\usepackage[left=1in, right=1in, top=1in, bottom=1in, bindingoffset=0cm, a4paper]{geometry}
\usepackage[T1]{fontenc}
\usepackage{lmodern}
\usepackage[applemac]{inputenc}
\usepackage{calrsfs}
\usepackage{enumitem}
\usepackage[mathscr]{euscript}

\makeatletter
\def\thm@space@setup{%
  \thm@preskip=\parskip \thm@postskip=0pt 
}
\renewenvironment{proof}[1][\proofname]{\par
  \vspace{-\topsep}
  \pushQED{\qed}%
  \normalfont
  \topsep0pt \partopsep0pt 
  \trivlist
  \item[\hskip\labelsep
        \itshape
    #1\@addpunct{.}]\ignorespaces
}{%
  \popQED\endtrivlist\@endpefalse
}
\makeatother

\makeatletter
\def\subsection{\@startsection{subsection}{3}%
  \z@{.4\linespacing\@plus.0\linespacing}{.1\linespacing}%
  {\normalfont\bfseries}}
\def\subsubsection{\@startsection{subsubsection}{3}%
  \z@{.4\linespacing\@plus.0\linespacing}{.1\linespacing}%
  {\normalfont\itshape}}
\makeatother

\begin{document}
\bibliographystyle{plainnat}
\linespread{1.6}

\abovedisplayskip=0pt
\abovedisplayshortskip=0pt
\belowdisplayskip=0pt
\belowdisplayshortskip=0pt
\abovecaptionskip=0pt
\belowcaptionskip=0pt

\newcounter{Lcount}
\newcommand{\squishlisttwo}{
\begin{list}{\alph{Lcount}. }
 { \usecounter{Lcount}
 \setlength{\itemsep}{0pt}
 \setlength{\parsep}{0pt}
 \setlength{\topsep}{0pt}
 \setlength{\partopsep}{0pt}
 \setlength{\leftmargin}{2em}
 \setlength{\labelwidth}{1.5em}
 \setlength{\labelsep}{0.5em} } }

\newcommand{\squishend}{
 \end{list} }

\newtheorem{lemma}{Lemma}[subsection]
\newtheorem{definition}{Definition}[section]
\newtheorem{remarks}{Remark}[section]
\newtheorem{remark}{Remark}[subsection]
\newtheorem{theorem}{Theorem}[subsection]
\newtheorem{corollary}[lemma]{Corollary}
\newtheorem{proposition}[lemma]{Proposition}
\numberwithin{equation}{subsection}
\newtheorem{axiom}[lemma]{Axiom}
\newtheorem*{acknowledgements}{Acknowledgments}

\newcommand{\R}{\mbox{$\Bbb R$}} \newcommand{\C}{\mbox{$\Bbb C$}}\newcommand{\F}{\mbox{$\Bbb F$}}
\newcommand{\N}{\mbox{$\Bbb N$}}
\newcommand{\Z}{\mbox{$\Bbb Z$}} \def\g{\mathfrak{g}} \def\q{\mathfrak{q}}
\def\h{\mathfrak{h}} \def\c{\mathfrak{c}} \def\d{\mathfrak{d}}
\def\m{\mathfrak{m}} \def\n{\mathfrak{n}} \def\i{\mathfrak{i}}
\def\l{\mathfrak{l}} \def\s{\mathfrak{s}}
\def\L{\mathscr{L}}
\def\r{\mathscr{R}}

\mbox{}

\vskip 1cm

\title[Solvable extensions of quasi-filiform algebras]{Solvable extensions of the naturally graded quasi-filiform Leibniz algebra
of second type $\mathcal{L}^2$}\maketitle
\begin{center}
{A. Shabanskaya}

Department of
Mathematics and Statistics, The University of Toledo, 2801 W
Bancroft St. Toledo, OH 43606, USA

ashaban@rockets.utoledo.edu
\end{center}
\begin{abstract} For a sequence of the naturally graded quasi-filiform Leibniz algebra
of second type $\mathcal{L}^2$ introduced by Camacho, G\'{o}mez, Gonz\'{a}lez and Omirov, all possible right and left solvable indecomposable
extensions over the field $\R$ are constructed so that the algebra serves
as the nilradical of the corresponding solvable Leibniz algebras we find in the paper. 
 \end{abstract}



\vskip 6cm

\newpage

\section{Introduction} 

Leibniz algebras were discovered by Bloch in 1965 \cite{B} who called them $D-$ algebras. Later on they were considered by Loday and Cuvier \cite{ C, Lo, L, LP} as
a non-antisymmetric analogue of Lie algebras. It makes every Lie algebra be a Leibniz algebra, but the converse is not true.
 Exactly Loday named them by Leibniz algebras after Gottfried Wilhelm Leibniz.
 
Since then many analogs of important theorems in Lie theory were found to be true for Leibniz algebras, such as
the analogue of Levi's theorem which was proved by Barnes \cite{Ba}. He showed that any finite-dimensional complex Leibniz algebra is decomposed into a semidirect sum of the solvable radical and a semisimple Lie algebra. 
Therefore the biggest challenge in the classification problem of finite-dimensional complex Leibniz algebras is to study the solvable part. And to classify solvable Leibniz algebras, we need nilpotent Leibniz algebras as their nilradicals, same as in the case of Lie algebras \cite{Mub3}.

There are also other analogous theorems such as Engel's theorem that has been generalized by several researchers like
 Ayupov and Omirov \cite{AO}, in a stronger form by Patsourakos \cite{P} and by Barnes who gave a simple proof of this theorem \cite{BD}.
Some other work on analogous properties could be found in the citations \cite{A, AR,  BDW, Barn, BC, BR, BH, Cas, D, FK, GV, MY, Om, OB}.

Every Leibniz algebra satisfies a generalized version of the Jacobi identity called the Leibniz identity.
 There are two Leibniz identities: the left and the right.
 We call Leibniz algebras right Leibniz algebras if they satisfy the right Leibniz identity
 and left, if they satisfy the left. A left Leibniz algebra is not necessarily
 a right Leibniz algebra \cite{D}. In this paper we construct both right and left solvable Leibniz algebras. \footnote{In the overview we only mention if the algebra is left, otherwise it is right.}
 
Leibniz algebras inherit an important property of Lie algebras which is that the right (left)
multiplication operator of a right (left) Leibniz algebra is a derivation \cite{CL}. Besides the algebra of right (left) multiplication operators
is endowed with a structure of a Lie algebra by means of the commutator \cite{CL}.  
 Also the quotient
algebra by the two-sided ideal generated by the square elements of a Leibniz algebra is a Lie algebra \cite{O},
where such ideal is the minimal, abelian and in the case of
non-Lie Leibniz algebras it is always non trivial. 

We will continue with an overview of the nilpotent and solvable Leibniz algebras over the field of characteristic zero as we only work with them.
In 1993 Loday found Leibniz
algebras in dimensions one and two over the field that does not have any zero divisors \cite{L}, such that 
the Leibniz algebra in dimension one is abelian and coincides with a Lie algebra
and in dimension two there are two non Lie Leibniz algebras: one is nilpotent (null-filiform) and one is solvable.

Leibniz algebras in dimension three over the complex numbers were classified by Omirov and Ayupov in 1998 \cite{AO,AOB} and were reviewed in 2012 
by Casas, Insua, Ladra, M. and Ladra, S. \cite{CI}. They noticed that one isomorphism class does not have a Leibniz algebra structure, so altogether there are four nilpotent (one is null-filiform, three are filiform) 
and seven solvable non Lie Leibniz algebras, such that one nilpotent and two solvable algebras contain a parameter.

Two dimensional and three dimensional left Leibniz algebras were classified by Demir, Misra and Stitzinger \cite{D} over the field of characteristic 0. In dimension two they found two non Lie Leibniz algebras: one is nilpotent, same as the right Leibniz algebra \cite{L}, and one is solvable. 
In dimension three there are five nilpotent Leibniz algebras with one of them depending on a parameter all isomorphic to the right Leibniz algebras \cite{CI} and seven solvable non Lie algebras such that two of them define one parametric families.

Four dimensional nilpotent complex Leibniz algebras
were classified by Albeverio, Omirov and Rakhimov \cite{AOR}. Altogether they found 21 nilpotent non Lie algebras,
where the first ten of them are either nulfiliform or filiform \cite{ASO} and remaining eleven are associative algebras,
such that only three of them, which are associative algebras, depend on a parameter.

Four dimensional solvable Leibniz algebras in dimension four over the field
of complex numbers were classified by Ca\~{n}ete and Khudoyberdiyev \cite{CK}.
They found 38 non Lie algebras, where 16 of them have up to two parameters.

In dimension five Khudoyberdiyev, Rakhimov and Said Husain studied solvable Leibniz algebras with 3-dimensional nilradicals \cite{KR}. 
They found 22 such algebras including one with the Heisenberg nilradical, which is
a Lie algebra, where twelve of them depend on up to four
parameters.

The classification of nilpotent Leibniz algebras has attempts to move to any finite dimension over the field of characteristic zero: Ayupov and Omirov proved that
in every finite dimension $n$ there is the only 
null-filiform non Lie Leibniz algebra \cite{ASO}.
Two kinds of filiform Leibniz algebras were classified in any finite dimension,
which are
naturally graded filiform Leibniz algebras \cite{ASO, Ve} and
non-characteristically nilpotent filiform Leibniz algebras \cite{ASO, KLO}.
 
There is a humongous work on filiform Leibniz algebras mostly based on their class:
$TLeib_{n+1},$ $FLeib_{n+1}$ or $SLeib_{n+1},$ which will not be overviewed, but
could be found in the references \cite{Ab, AA, ASO,CCR,GJK,GO,Omi,OR,Ra,RaS,RS,RB,RH,rh,rah,RIS,Rak,Ri}.

Naturally graded quasi-filiform Leibniz algebras in any finite dimension over $\C$ were studied by Camacho, G\'{o}mez, Gonz\'{a}lez,
Omirov \cite{CGGO}. They found six such algebras of the first type, where two of them depend on a parameter and eight algebras of the second type with one of them
depending on a parameter. This paper continues with finding all solvable extensions in any finite dimension of non characteristically
nilpotent quasi-filiform Leibniz algebras over the field of real numbers.
Solvable extensions of $\mathcal{L}^{1}$ and $\mathcal{L}^{3},(n\geq4)$ of the second type were found in \cite{Sha}.
In this paper we do the same work with $\mathcal{L}^{2}$ of the second type as well.

Some other work on nilpotent Leibniz algebras 
could be found in the following references \cite{CCR, CCGO, CO, Cam, Ca, CG, CR,CC}.

It is possible to find solvable Leibniz algebras in any finite dimension working with the sequence of nilpotent Leibniz algebras in any finite dimension and their ``nil-independent''
derivations. This method for Lie algebras is based on what was shown by
Mubarakzyanov in \cite{Mub3}: the dimension of the complimentary vector space to the nilradical does not exceed the number of 
nil-independent derivations of the nilradical. This result 
was extended for Leibniz algebras by Casas, Ladra, Omirov and Karimjanov \cite{CLOK} with the help
of \cite{AO}. Besides, similarly to the case of Lie algebras, for a solvable Leibniz
algebra $L$ we also have the inequality $\dim\n\i\l(L)\geq\frac{1}{2}\dim L$ \cite{Mub3}. 
Further, there is the following work performed over the field of characteristic zero using 
this method: Casas, Ladra, Omirov and Karimjanov classified solvable Leibniz algebras with null-filiform nilradical \cite{CLOK};
Omirov and his colleagues Casas, Khudoyberdiyev, Ladra, Karimjanov, Camacho and Masutova classified solvable Leibniz algebras whose nilradicals are a direct sum of null-filiform algebras \cite{KL}, naturally graded filiform \cite{CL, LadraMO}, triangular \cite{KK} and finally filiform \cite{COM}.
Bosko-Dunbar, Dunbar, Hird and Stagg attempted to classify left solvable Leibniz algebras with Heisenberg nilradical \cite{BDHS}.
Left and right solvable extensions of $\mathcal{R}_{18}$ \cite{AOR}, $\mathcal{L}^1$ and $\mathcal{L}^3$ \cite{CGGO} over the field of real numbers were 
found by Shabanskaya in \cite{ShA, Sha}.

Some other work and kinds of Leibniz algebras not discussed are shown in
the citations \cite{Abd, AB, ao, ACK, BaD, Bar, BS, ca, CCG, cgv, CLP, D, DA, F, G, GV, J, O, RR, Z}.

The starting point of the present article as was already mentioned is a naturally graded quasi-filiform non Lie Leibniz algebra of the second type
$\mathcal{L}^2,(n\geq4)$ in the notation given in \cite{CGGO}.
This algebra is left and right at the same time and an associative when $n=4$.

For a sequence $\mathcal{L}^2,$ we
find all possible solvable indecomposable extensions as left and right Leibniz algebras in every finite dimension over the field of real numbers. Such extensions of only dimensions one and two are
possible.
 
Right solvable extensions with a codimension
one nilradical $\mathcal{L}^2$ are found following the steps in
Theorems \ref{TheoremRL2}, \ref{TheoremRL2Absorption} and \ref{RL2(Change of Basis)} with the main result summarized in Theorem \ref{RL2(Change of Basis)}, where
it is shown there are four such algebras and neither of them is left at the same time. 
There is the only solvable indecomposable right Leibniz algebra with a codimension
two nilradical stated in Theorem \ref{RCodim2L2}, which is not left either.

We follow the steps in
Theorems \ref{TheoremLL2}, \ref{TheoremL(L2)Absorption} and \ref{TheoremL(L2)Basis} to find one dimensional left solvable extensions. 
We notice in Theorem \ref{TheoremL(L2)Basis}, that we have four of them as well and neither of them is right.
We find the only two dimensional left (not right) solvable extension stated in Theorem \ref{(L)Codim2L2}. 

As regards notation, we use $\langle e_1,e_2,...,e_r \rangle$ to
denote the $r$-dimensional subspace generated by
$e_1,e_2,...,e_r,$ where $r\in\N$, $LS$ is the lower central series and $DS$ is the derived series and refer to them collectively as the characteristic series. Besides $\g$
and $l$
are used to denote solvable right and left Leibniz algebras, respectively.

Throughout the paper all the algebras are finite dimensional over the field of real numbers and
if the bracket is not given, then it is assumed to be zero, except the brackets for the
nilradical, which are not given (see Remark \ref{remark5.1}) to save space. Besides when
we say that the algebra depends on a parameter, we do not mean a discrete value. 

In the tables throughout the paper an ordered triple is a shorthand notation
for a derivation property of the multiplication operators, which is either $\r_z\left([x,y]\right)=
[\r_z(x),y]+[x,\r_z(y)]$ or $\L_z\left([x,y]\right)=
[\L_z(x),y]+[x,\L_z(y)]$. We also assign $\r_{e_{n+1}}:=\r$ and $\L_{e_{n+1}}:=\L.$

We use Maple software to compute the Leibniz identity, the ``absorption'' (see \cite{Shab, ST} and
Section \ref{Solvable left Leibniz algebras}), the change
of basis for solvable Leibniz algebras in some particular dimensions, which are generalized and proved in an arbitrary finite dimension.

\section{Preliminaries}\label{Pr}
We give some Basic definitions encountered working with Leibniz algebras.
\begin{definition}
\begin{enumerate}[noitemsep, topsep=0pt]
\item[1.] A vector space L over a field F with a bilinear operation\\
$[-,-]:\,L\rightarrow L$
is called a Leibniz algebra if for any $x,\,y,\,z\in\,L$ the so called Leibniz identity
\begin{equation}\nonumber [[x,y],z]=[[x,z],y]+[x,[y,z]]\end{equation}
holds. This Leibniz identity is known as the right Leibniz identity and we call $L$
in this case a right Leibniz algebra.
\item[2.] There exists the version corresponding to the left Leibniz identity
\begin{equation}\nonumber
[[x,y],z]=[x,[y,z]]-[y,[x,z]],
\end{equation}
and a Leibniz algebra L is called a left Leibniz algebra.
\end{enumerate}
\end{definition}
\begin{remarks}
In addition, if $L$ satisfies $[x,x]=0$ for every $x\in L$, then it is a Lie algebra.
Therefore every Lie algebra is a Leibniz algebra, but the converse is not true.
\end{remarks}
\begin{definition}
The two-sided ideal $C(L)=\{x\in L:\,[x,y]=[y,x]=0\}$ is said to be the center of $L.$
\end{definition}
\begin{definition}
A linear map $d:\,L\rightarrow L$ of a Leibniz algebra $L$ is a derivation
if for all $x,\,y\in L$
\begin{equation}
\nonumber d([x,y])=[d(x),y]+[x,d(y)].
\end{equation}
\end{definition}
If $L$ is a right Leibniz algebra and $x\in L,$ then the right multiplication operator $\r_x:\,L\rightarrow L$ defined as $\r_x(y)=[y,x],\,y\in L$
is a derivation (for a left Leibniz algebra $L$ with $x\in L,$ the left multiplication operator $\L_x:\,L\rightarrow L,\,\L_x(y)=[x,y],\,y\in L$ is a derivation).

Any right Leibniz algebra $L$ is associated with the algebra of right multiplications $\r(L)=\{\r_x\,|\,x\in L\}$
endowed with the structure of a Lie algebra by means of the commutator $[\r_x,\r_y]=\r_x\r_y-\r_y\r_x=\r_{[y,x]},$
which defines an antihomomorphism between $L$ and $\r(L).$

For a left Leibniz algebra $L,$ the corresponding algebra of left multiplications $\L(L)=\{\L_x\,|\,x\in L\}$
is endowed with the structure of a Lie algebra by means of the commutator as well $[\L_x,\L_y]=\L_x\L_y-\L_y\L_x=\L_{[x,y]}.$
In this case we have a homomorphism between $L$ and $\L(L)$.
\begin{definition}
Let $d_1,d_2,...,d_n$ be derivations of a Leibniz algebra $L.$ The derivations $d_1,d_2,...,d_n$ are said to be ``nil-independent''
if $\alpha_1d_1+\alpha_2d_2+\alpha_3d_3+\ldots+\alpha_nd_n$ is not nilpotent for any scalars $\alpha_1,\alpha_2,...,\alpha_n\in F,$
otherwise they are said to be ``nil-dependent''.
\end{definition}
\begin{definition}
For a given Leibniz algebra $L,$ we define the sequence of two-sided ideals as follows:
\begin{equation}
\nonumber L^0=L,\,L^{k+1}=[L^k,L],\,(k\geq0)\qquad\qquad L^{(0)}=L,\,L^{(k+1)}=[L^{(k)},L^{(k)}],\,(k\geq0),
\end{equation}
which are the lower central series and the derived series of $L,$ respectively.

A Leibniz algebra $L$ is said to be nilpotent (solvable) if there exists $m\in\N$ such that $L^m=0$ ($L^{(m)}=0$).
The minimal such number $m$ is said to be the index of nilpotency (solvability). 
\end{definition}
\begin{definition}
A Leibniz algebra $L$ is called a quasi-filiform if $L^{n-3}\neq0$ and $L^{n-2}=0,$ where $\dim(L)=n.$ 
\end{definition}

\section{Constructing solvable Leibniz algebras with a given nilradical}

Every solvable Leibniz algebra $L$ contains a unique maximal nilpotent
ideal called the nilradical and denoted $\n\i\l(L)$ such that
$\dim\n\i\l(L)\geq\frac{1}{2}\dim(L)$ \cite{Mub3}. Let us consider
the problem of constructing solvable Leibniz algebras
$L$ with a given nilradical $N=\n\i\l(L)$. Suppose
$\{e_1,e_2,e_3,e_4,...,e_n\}$ is a basis for the nilradical and
$\{e_{n+1},...,e_{p}\}$ is a basis for a subspace complementary to
the nilradical.

If $L$ is a solvable Leibniz algebra \cite{AO}, then
\begin{equation}\label{Nil}[L,L]\subseteq N \end{equation} and we have
the following structure equations
\begin{equation}\label{Algebra}[e_i, e_j ] =C_{ij}^ke_k, [e_a, e_i]
=A^k_{ai}e_k, [e_i,e_a]=A^k_{ia}e_k, [e_a, e_b] =B^k_{ab} e_k, \end{equation} where $1\leq i,
j, k, m\leq n$ and $n+1\leq a, b\leq p$.

\subsection{Solvable right Leibniz algebras}  Calculations show that satisfying the right Leibniz identity is equivalent to the
following conditions: 
\begin{equation}\label{RLeibniz}
A_{ai}^kC_{kj}^m=A_{aj}^kC_{ki}^m+C_{ij}^kA_{ak}^m,\,A_{ia}^kC_{kj}^m=C_{ij}^kA_{ka}^m+A_{aj}^kC_{ik}^m,\,
C_{ij}^kA_{ka}^m=A_{ia}^kC_{kj}^m+A_{ja}^kC_{ik}^m,
\end{equation}
 \begin{equation}\label{RRLeibniz}
 B_{ab}^kC_{ki}^m=A_{ai}^kA_{kb}^m+A_{bi}^kA_{ak}^m,\,A_{ai}^kA_{kb}^m=B_{ab}^kC_{ki}^m+A_{ib}^kA_{ak}^m,\,
 B_{ab}^kC_{ik}^m=A_{kb}^mA_{ia}^k-A_{ka}^mA_{ib}^k.
\end{equation} 
 Then the entries of the matrices
$A_a,$ which are $(A_i^k)_a$, must satisfy the equations $(\ref{RLeibniz})$ which come from
all possible Leibniz identities between the triples $\{e_a,e_i,e_j\}.$ In $(\ref{RRLeibniz})$ we have
structure constants equations obtained from all Leibniz identities between the triples
$\{e_a,e_b,e_i\}$.

Since $N$ is the nilradical of $L$, no nontrivial linear
combination of the matrices $A_a,\,(n+1\leq a\leq p)$ is nilpotent,
which means that the matrices $A_a$ must be ``nil-independent''
(see Section \ref{Pr}., \cite{CLOK}, \cite{Mub3}).

Let us now consider the right multiplication operator $\r_{e_a}$ and restrict it
to $N$, $(n+1\leq a\leq p).$ We shall
get outer derivations of the nilradical $N=\n\i\l(L)$
\cite{CLOK}. Then finding the matrices $A_a$ is the same as
finding outer derivations $\r_{e_a}$ of $N.$ Further the commutators $[\r_{e_b},\r_{e_a}]=\r_{[e_a,e_b]},\,(n+1\leq a,b\leq
p)$ due to $(\ref{Nil})$ consist of inner derivations of
$N$. So those commutators give the structure constants $B_{ab}^k$ as shown in the last equation of $(\ref{RRLeibniz}),$ but only up to the elements in the center of the nilradical
 $N$, because if $e_i,\,(1\leq i\leq n)$ is in the center of $N,$ then $\left(\r_{e_i}\right)_{|_N}=0,$ where $\left(\r_{e_i}\right)_{|_{N}}$
 is an inner derivation of the nilradical. 
\subsection{Solvable left Leibniz algebras}\label{Solvable left Leibniz algebras} Satisfying the left Leibniz identity, we have
 \begin{equation}\label{LLeibniz}
A_{ai}^kC_{jk}^m=A_{ja}^kC_{ki}^m+C_{ji}^kA_{ak}^m,\,A_{ia}^kC_{jk}^m=C_{ji}^kA_{ka}^m+A_{ja}^kC_{ik}^m,\,
C_{ij}^kA_{ak}^m=A_{ai}^kC_{kj}^m+A_{aj}^kC_{ik}^m,
\end{equation}
 \begin{equation}\label{LLLeibniz}
 B_{ab}^kC_{ik}^m=A_{ia}^kA_{kb}^m+A_{ib}^kA_{ak}^m,\,A_{ib}^kA_{ak}^m=B_{ab}^kC_{ik}^m+A_{ai}^kA_{kb}^m,\,
 B_{ab}^kC_{ki}^m=A_{ak}^mA_{bi}^k-A_{bk}^mA_{ai}^k,
\end{equation} 
and the entries of the matrices
$A_a,$ which are $(A_a)_i^k$, must satisfy the equations $(\ref{LLeibniz})$.
Similarly $A_a=\left(\L_{e_a}\right)_{|_{N}},\,(n+1\leq a\leq p)$ are outer derivations and
the commutators $[\L_{e_a},\L_{e_b}]=\L_{[e_a,e_b]}$ give the structure constants $B_{ab}^k,$ but only up to the elements in the center of the nilradical
 $N$.

Once the left or right Leibniz identities are satisfied in the most general possible way and the outer derivations are found: 

\begin{enumerate}[noitemsep, topsep=0pt]
 \item[(i)]
We can carry out the technique of ``absorption'' \cite{Shab, ST}, which means we can
simplify a solvable Leibniz algebra without affecting the nilradical in $(\ref{Algebra})$ applying the transformation
\begin{equation}
\nonumber e^{\prime}_i=e_i,\,(1\leq i\leq n),\,e^{\prime}_a=e_a+\sum_{k=1}^nd^ke_k,\,(n+1\leq a\leq p).
\end{equation}
\item[(ii)] We change basis without affecting the
nilradical in $(\ref{Algebra})$ to remove all the possible parameters and simplify the algebra. \end{enumerate}

\section{The nilpotent sequence $\mathcal{L}^2$}
In $\mathcal{L}^2,(n\geq4)$ the positive
integer $n$ denotes the dimension of the algebra. The center
of this algebra is $C(\mathcal{L}^2)=\langle e_{2},e_n \rangle$. $\mathcal{L}^2$ can be described explicitly as follows: in the
basis $\{e_1,e_2,e_3,e_4,\ldots,e_n\}$ it has only the following
non-zero brackets:  
\begin{equation}
\begin{array}{l}
\displaystyle  [e_1,e_1]=e_{2},[e_i,e_1]=e_{i+1},(3\leq i\leq n-1),[e_1,e_3]=e_2-e_4,[e_1,e_j]=-e_{j+1},\\
\displaystyle (4\leq j\leq n-1,n\geq4).
\end{array} 
\label{L2}
\end{equation}

 The
dimensions of the ideals in the characteristic
series are
\begin{equation}\nonumber DS=[n,n-2,0],
LS=[n,n-2,n-4,n-5,n-6,...,0],(n\geq4).\end{equation}

A quasi-filiform Leibniz algebra  $\mathcal{L}^2$ was introduced by Camacho, G\'{o}mez, Gonz\'{a}lez and Omirov
in \cite{CGGO}.
This algebra is served as the nilradical for the
left and right solvable indecomposable extensions we construct in this
paper.

 It is shown below that solvable indecomposable right (left) Leibniz algebras
 with the nilradical $\mathcal{L}^2$ only exist for $\dim\g=n+1$ and
$\dim\g=n+2$ ($\dim l=n+1$ and
$\dim l=n+2$).

For the solvable indecomposable right Leibniz algebras with a codimension one
nilradical, we use the notation $\g_{n+1,i}$ where $n+1$ indicates the
dimension of the algebra $\g$ and $i$ its numbering within the list
of algebras. There are four types of such algebras up to isomorphism
so $1\leq i\leq 4,$ where neither of them are left at the same time. There is the only solvable indecomposable right Leibniz algebra up to
isomorphism with a codimension two nilradical denoted $\g_{n+2,1}$.

We also have four solvable indecomposable left Leibniz algebras with a codimension one
nilradical $\mathcal{L}^2$ denoted $l_{n+1,i},(1\leq i\leq 4)$.
There is the only solvable indecomposable left (not right) Leibniz algebra with a codimension two nilradical denoted $l_{n+2,1}$.

There are only two right (left) algebras that contain parameters, which are $\g_{n+1,1}$ and $\g_{n+1,4}$
(or $l_{n+1,1},l_{n+1,4}$).
Altogether right (left) algebras depend on at most $n-4,(n\geq6)$ parameters and only one parameter when $n=4,5$:
$\g_{n+1,1}$ (or $l_{n+1,1}$) depends on 1,$(n\geq4)$ and $\g_{n+1,4}$ (or $l_{n+1,4}$) depends on
$n-5,(n\geq6)$. Those algebras define a
continuous family of algebras.

 \section{Classification of solvable indecomposable Leibniz algebras with
a nilradical $\mathcal{L}^2$}

Our goal in this Section is to find all possible right and left solvable indecomposable extensions
of the nilpotent Leibniz algebra $\mathcal{L}^2,$ which serves
as the nilradical of the extended algebra.
\begin{remarks}\label{remark5.1} It is assumed throughout this section that the
solvable indecomposable right Leibniz algebras $\g$ and the solvable indecomposable left Leibniz algebras 
$l$ have the nilradical $\mathcal{L}^2$; however, the brackets of the nilradical will be omitted.
\end{remarks}
\subsection{Solvable indecomposable right Leibniz algebras with a nilradical $\mathcal{L}^2$}
\subsubsection{One dimensional right solvable extensions of $\mathcal{L}^2$}
 The nilpotent Leibniz algebra $\mathcal{L}^2$ is defined in $(\ref{L2})$. Suppose $\{e_{n+1}\}$
 is in the complementary subspace to the nilradical $\mathcal{L}^2$ and $\g$ is the corresponding solvable right Leibniz algebra.
 Since $[\g,\g]\subseteq \mathcal{L}^2,$ therefore we have
\begin{equation}
\left\{
\begin{array}{l}
\displaystyle  [e_1,e_1]=e_{2},[e_i,e_1]=e_{i+1},(3\leq i\leq n-1),[e_1,e_3]=e_2-e_4,\\
\displaystyle [e_1,e_j]=-e_{j+1},(4\leq j\leq n-1),[e_r,e_{n+1}]=\sum_{s=1}^na_{s,r}e_s,[e_{n+1},e_k]=\sum_{s=1}^nb_{s,k}e_s,\\
\displaystyle (1\leq k\leq n,1\leq r\leq n+1).
\end{array} 
\right.
\label{BRLeibniz}
\end{equation}
\begin{theorem}\label{TheoremRL2} Set $a_{1,1}:=a$ and $a_{2,2}:=b$ in $(\ref{BRLeibniz})$. To satisfy the right Leibniz identity, there are the following cases based on the conditions involving parameters,
each gives a continuous family of solvable Leibniz algebras:
\begin{enumerate}[noitemsep, topsep=0pt]
\item[(1)] If $a\neq0,b\neq2a,(n=4)$ and $b\neq(4-n)a,a\neq0,b\neq2a,(n\geq5),$ then
\begin{equation}
\left\{
\begin{array}{l}
\displaystyle  \nonumber [e_1,e_{n+1}]=ae_1+a_{2,1}e_2-(2a-b)e_3+\sum_{k=4}^n{a_{k,1}e_k},[e_2,e_{n+1}]=be_2,
[e_3,e_{n+1}]=a_{2,3}e_2-\\
\displaystyle (a-b)e_3+A_{4,3}e_4+\sum_{k=5}^n{a_{k,3}e_k},[e_{i},e_{n+1}]=\left((i-4)a+b\right)e_{i}+A_{4,3}e_{i+1}+\sum_{k=i+2}^n{a_{k-i+3,3}e_k},\\
\displaystyle(4\leq i\leq n),
[e_{n+1},e_{n+1}]=a_{2,n+1}e_2,[e_{n+1},e_1]=-ae_1+b_{2,1}e_2+(2a-b)e_3-\sum_{k=4}^n{a_{k,1}e_k},\\
\displaystyle [e_{n+1},e_3]=B_{2,3}e_2+(a-b)e_3-A_{4,3}e_4-\sum_{k=5}^n{a_{k,3}e_k},[e_{n+1},e_4]=ae_2-be_4-A_{4,3}e_5-\\
\displaystyle \sum_{k=6}^n{a_{k-1,3}e_k}, [e_{n+1},e_j]=\left((4-j)a-b\right)e_j-A_{4,3}e_{j+1}-\sum_{k=j+2}^n{a_{k-j+3,3}e_k},(5\leq j\leq n),\\
\displaystyle where\,\,B_{2,3}:=-a_{2,3}+\frac{(a-b)b_{2,1}+a(a_{2,1}+a_{4,1})}{2a-b}\,\,and\\
\displaystyle A_{4,3}:=-\frac{b}{a}\cdot a_{2,3}+\frac{(a-b)b_{2,1}+a(a_{2,1}+a_{4,1})}{2a-b},
\end{array} 
\right.
\end{equation} 
$\r_{e_{n+1}}=\left[\begin{smallmatrix}
 a & 0 & 0 & 0&0&&\cdots &0& \cdots&0 & 0&0 \\
  a_{2,1} & b & a_{2,3}& 0 &0& & \cdots &0&\cdots  & 0& 0&0\\
  -2a+b & 0 & -a+b & 0 & 0& &\cdots &0&\cdots &0& 0&0 \\
  a_{4,1} & 0 &  A_{4,3} & b &0 & &\cdots&0&\cdots &0 & 0&0\\
  a_{5,1} & 0 & a_{5,3} & A_{4,3} & a+b  &&\cdots &0 &\cdots&0 & 0&0 \\
 \boldsymbol{\cdot} & \boldsymbol{\cdot} & \boldsymbol{\cdot} & a_{5,3} & A_{4,3}  &\ddots& &\vdots &&\vdots & \vdots&\vdots \\
  \vdots & \vdots & \vdots &\vdots &\vdots  &\ddots& \ddots&\vdots &&\vdots & \vdots&\vdots \\
  a_{i,1} & 0 & a_{i,3} & a_{i-1,3} & a_{i-2,3}  &\cdots&A_{4,3}& (i-4)a+b&\cdots&0 & 0&0\\
   \vdots  & \vdots  & \vdots &\vdots &\vdots&&\vdots &\vdots &&\vdots  & \vdots&\vdots \\
 a_{n-1,1} & 0 & a_{n-1,3}& a_{n-2,3}& a_{n-3,3}&\cdots &a_{n-i+3,3} &a_{n-i+2,3}&\cdots&A_{4,3} &(n-5)a+b& 0\\
 a_{n,1} & 0 & a_{n,3}& a_{n-1,3}& a_{n-2,3}&\cdots &a_{n-i+4,3}  &a_{n-i+3,3}&\cdots&a_{5,3} &A_{4,3}& (n-4)a+b
\end{smallmatrix}\right].$
\item[(2)] If $b:=(4-n)a,a\neq0,(n\geq5),$
then the brackets for the algebra are  
\begin{equation}
\left\{
\begin{array}{l}
\displaystyle  \nonumber [e_1,e_{n+1}]=ae_1+a_{2,1}e_2+(2-n)ae_3+\sum_{k=4}^n{a_{k,1}e_k},[e_2,e_{n+1}]=(4-n)ae_2,\\
\displaystyle [e_3,e_{n+1}]=a_{2,3}e_2+(3-n)ae_3+A_{4,3}e_4+\sum_{k=5}^n{a_{k,3}e_k},[e_{i},e_{n+1}]=\left(i-n\right)ae_{i}+A_{4,3}e_{i+1}+\\
\displaystyle \sum_{k=i+2}^n{a_{k-i+3,3}e_k},(4\leq i\leq n-1),
[e_{n+1},e_{n+1}]=a_{2,n+1}e_2+a_{n,n+1}e_n,[e_{n+1},e_1]=-ae_1+\\
\displaystyle b_{2,1}e_2+(n-2)ae_3-\sum_{k=4}^n{a_{k,1}e_k}, [e_{n+1},e_3]=B_{2,3}e_2+(n-3)ae_3-A_{4,3}e_4-\sum_{k=5}^n{a_{k,3}e_k},\\
\displaystyle [e_{n+1},e_4]=ae_2+(n-4)ae_4-A_{4,3}e_5-\sum_{k=6}^n{a_{k-1,3}e_k}, [e_{n+1},e_j]=\left(n-j\right)ae_j-A_{4,3}e_{j+1}-\\
\displaystyle \sum_{k=j+2}^n{a_{k-j+3,3}e_k},(5\leq j\leq n-1),\, where\,\,B_{2,3}:=-a_{2,3}+\frac{(n-3)b_{2,1}+a_{2,1}+a_{4,1}}{n-2} \,\,and\\
\displaystyle A_{4,3}:=(n-4)a_{2,3}+\frac{(n-3)b_{2,1}+a_{2,1}+a_{4,1}}{n-2},
\end{array} 
\right.
\end{equation}
$\r_{e_{n+1}}=\left[\begin{smallmatrix}
 a & 0 & 0 & 0&0&&\cdots &0& \cdots&0 & 0&0 \\
  a_{2,1} & (4-n)a & a_{2,3}& 0 &0& & \cdots &0&\cdots  & 0& 0&0\\
  (2-n)a & 0 & (3-n)a & 0 & 0& &\cdots &0&\cdots &0& 0&0 \\
  a_{4,1} & 0 &  A_{4,3} & (4-n)a &0 & &\cdots&0&\cdots &0 & 0&0\\
  a_{5,1} & 0 & a_{5,3} & A_{4,3} & (5-n)a  &&\cdots &0 &\cdots&0 & 0&0 \\
 \boldsymbol{\cdot} & \boldsymbol{\cdot} & \boldsymbol{\cdot} & a_{5,3} & A_{4,3}  &\ddots& &\vdots &&\vdots & \vdots&\vdots \\
  \vdots & \vdots & \vdots &\vdots &\vdots  &\ddots& \ddots&\vdots &&\vdots & \vdots&\vdots \\
  a_{i,1} & 0 & a_{i,3} & a_{i-1,3} & a_{i-2,3}  &\cdots&A_{4,3}& (i-n)a&\cdots&0 & 0&0\\
   \vdots  & \vdots  & \vdots &\vdots &\vdots&&\vdots &\vdots &&\vdots  & \vdots&\vdots \\
 a_{n-1,1} & 0 & a_{n-1,3}& a_{n-2,3}& a_{n-3,3}&\cdots &a_{n-i+3,3} &a_{n-i+2,3}&\cdots&A_{4,3} &-a& 0\\
 a_{n,1} & 0 & a_{n,3}& a_{n-1,3}& a_{n-2,3}&\cdots &a_{n-i+4,3}  &a_{n-i+3,3}&\cdots&a_{5,3} &A_{4,3}& 0
\end{smallmatrix}\right].$
 
\item[(3)] If $a=0$ and $b\neq0,(n\geq4),$ then
\begin{equation}
\left\{
\begin{array}{l}
\displaystyle  \nonumber [e_1,e_{n+1}]=a_{2,1}e_2+be_3+\sum_{k=4}^n{a_{k,1}e_k}, [e_2,e_{n+1}]=be_2,
[e_{i},e_{n+1}]=be_{i}+\sum_{k=i+1}^n{a_{k-i+3,3}e_k},\\
\displaystyle (3\leq i\leq n),[e_{n+1},e_{n+1}]=a_{2,n+1}e_2,[e_{n+1},e_1]=b_{2,1}e_2-be_3-\sum_{k=4}^n{a_{k,1}e_k},\\
\displaystyle[e_{n+1},e_3]=b_{2,1}e_2-be_3-\sum_{k=4}^n{a_{k,3}e_k},[e_{n+1},e_j]=-be_j-\sum_{k=j+1}^n{a_{k-j+3,3}e_k},(4\leq j\leq n),
\end{array} 
\right.
\end{equation} 
$\r_{e_{n+1}}=\left[\begin{smallmatrix}
 0 & 0 & 0 & 0&0&&\cdots &0& \cdots&0 & 0&0 \\
  a_{2,1} & b & 0& 0 &0& & \cdots &0&\cdots  & 0& 0&0\\
  b & 0 & b & 0 & 0& &\cdots &0&\cdots &0& 0&0 \\
  a_{4,1} & 0 &  a_{4,3} & b &0 & &\cdots&0&\cdots &0 & 0&0\\
  a_{5,1} & 0 & a_{5,3} & a_{4,3} & b  &&\cdots &0 &\cdots&0 & 0&0 \\
 \boldsymbol{\cdot} & \boldsymbol{\cdot} & \boldsymbol{\cdot} & a_{5,3} & a_{4,3}  &\ddots& &\vdots &&\vdots & \vdots&\vdots \\
  \vdots & \vdots & \vdots &\vdots &\vdots  &\ddots& \ddots&\vdots &&\vdots & \vdots&\vdots \\
  a_{i,1} & 0 & a_{i,3} & a_{i-1,3} & a_{i-2,3}  &\cdots&a_{4,3}& b&\cdots&0 & 0&0\\
   \vdots  & \vdots  & \vdots &\vdots &\vdots&&\vdots &\vdots &&\vdots  & \vdots&\vdots \\
 a_{n-1,1} & 0 & a_{n-1,3}& a_{n-2,3}& a_{n-3,3}&\cdots &a_{n-i+3,3} &a_{n-i+2,3}&\cdots&a_{4,3} &b& 0\\
 a_{n,1} & 0 & a_{n,3}& a_{n-1,3}& a_{n-2,3}&\cdots &a_{n-i+4,3}  &a_{n-i+3,3}&\cdots&a_{5,3} &a_{4,3}& b
\end{smallmatrix}\right].$
   \allowdisplaybreaks
\item[(4)] If $b:=2a,a\neq0,(n\geq4),$ then
\begin{equation}
\left\{
\begin{array}{l}
\displaystyle  \nonumber [e_1,e_{n+1}]=ae_1+a_{2,1}e_2+\sum_{k=4}^n{a_{k,1}e_k},[e_2,e_{n+1}]=2ae_2,
[e_3,e_{n+1}]=a_{2,3}e_2+ae_3+\\
\displaystyle A_{4,3}e_4+\sum_{k=5}^n{a_{k,3}e_k},[e_{i},e_{n+1}]=\left(i-2\right)ae_{i}+A_{4,3}e_{i+1}+\sum_{k=i+2}^n{a_{k-i+3,3}e_k},(4\leq i\leq n),\\
\displaystyle
[e_{n+1},e_{n+1}]=a_{2,n+1}e_2,[e_{n+1},e_1]=-ae_1+\left(a_{2,1}+a_{4,1}\right)e_2-\sum_{k=4}^n{a_{k,1}e_k},\\
\displaystyle [e_{n+1},e_3]=b_{2,3}e_2-ae_3-A_{4,3}e_4-\sum_{k=5}^n{a_{k,3}e_k},[e_{n+1},e_4]=ae_2-2ae_4-A_{4,3}e_5-\\
\displaystyle \sum_{k=6}^n{a_{k-1,3}e_k}, [e_{n+1},e_j]=\left(2-j\right)ae_j-A_{4,3}e_{j+1}-\sum_{k=j+2}^n{a_{k-j+3,3}e_k},(5\leq j\leq n),\\
\displaystyle where\,\, A_{4,3}:=b_{2,3}-a_{2,3},
\end{array} 
\right.
\end{equation} 
$\r_{e_{n+1}}=\left[\begin{smallmatrix}
 a & 0 & 0 & 0&0&&\cdots &0& \cdots&0 & 0&0 \\
  a_{2,1} & 2a & a_{2,3}& 0 &0& & \cdots &0&\cdots  & 0& 0&0\\
  0 & 0 & a & 0 & 0& &\cdots &0&\cdots &0& 0&0 \\
  a_{4,1} & 0 &  A_{4,3} & 2a &0 & &\cdots&0&\cdots &0 & 0&0\\
  a_{5,1} & 0 & a_{5,3} & A_{4,3} & 3a  &&\cdots &0 &\cdots&0 & 0&0 \\
 \boldsymbol{\cdot} & \boldsymbol{\cdot} & \boldsymbol{\cdot} & a_{5,3} & A_{4,3}  &\ddots& &\vdots &&\vdots & \vdots&\vdots \\
  \vdots & \vdots & \vdots &\vdots &\vdots  &\ddots& \ddots&\vdots &&\vdots & \vdots&\vdots \\
  a_{i,1} & 0 & a_{i,3} & a_{i-1,3} & a_{i-2,3}  &\cdots&A_{4,3}& (i-2)a&\cdots&0 & 0&0\\
   \vdots  & \vdots  & \vdots &\vdots &\vdots&&\vdots &\vdots &&\vdots  & \vdots&\vdots \\
 a_{n-1,1} & 0 & a_{n-1,3}& a_{n-2,3}& a_{n-3,3}&\cdots &a_{n-i+3,3} &a_{n-i+2,3}&\cdots&A_{4,3} &(n-3)a& 0\\
 a_{n,1} & 0 & a_{n,3}& a_{n-1,3}& a_{n-2,3}&\cdots &a_{n-i+4,3}  &a_{n-i+3,3}&\cdots&a_{5,3} &A_{4,3}& (n-2)a
\end{smallmatrix}\right].$

\end{enumerate}
\end{theorem}
\begin{proof} 
\begin{enumerate}[noitemsep, topsep=0pt]
\item[(1)] Suppose $b\neq(4-n)a,a\neq0,b\neq2a,(n\geq5).$ 
Then the proof is off-loaded to Table \ref{Right(L2)}. If $a\neq0$ and $b\neq2a,(n=4),$ then we consider applicable identities given in Table \ref{Right(L2)}. and most of calculations go the same way with two differences: first instead of
$2.$ we apply the Leibniz identity $\r[e_3,e_3]=[\r(e_3),e_3]+[e_3,\r(e_3)]$ to obtain that $a_{1,3}=0$ and $[e_3,e_5]=
a_{2,3}e_2+a_{3,3}e_3+a_{4,3}e_4$, second, when applying the identity $13.$ (the identity in $15.$ is the same as $13.$, when $n=4$), we deduce that $b_{2,4}:=a,$
$b_{4,4}:=-b$ and $[e_5,e_4]=ae_2-be_4.$
\item[(2)]  Suppose $b:=(4-n)a,a\neq0,(n\geq5).$ We apply the identities
given in Table \ref{Right(L2)}., except the identity $18.$
\item[(3)] Suppose $a=0$ and $b\neq0.$ If $(n\geq5),$ then we apply the identities $1.-18.$ given in Table \ref{Right(L2)}.,
but for the last two identities, we apply:
$\r_{e_3}\left([e_{n+1},e_{n+1}]\right)=[\r_{e_3}(e_{n+1}),e_{n+1}]+[e_{n+1},\r_{e_3}(e_{n+1})],$
$\r[e_{n+1},e_{1}]=[\r(e_{n+1}),e_{1}]+[e_{n+1},\r(e_{1})]$ and quod erat demonstrandum.
If $n=4,$ then we repeat case $(1),$ except for $19.$ and $20.$
applying two identities given above.
\item[(4)] Suppose $b:=2a,a\neq0.$ We repeat case $(1)$
with the difference that the identity $20.$ gives 
instead of the condition on the parameter $b_{2,3}$ that $b_{2,1}:=a_{2,1}+a_{4,1}.$
It implies that $[e_{n+1},e_1]=-ae_1+(a_{2,1}+a_{4,1})e_2-\sum_{k=4}^n{a_{k,1}e_k}.$
\end{enumerate}
\end{proof} 
\newpage
\begin{table}[htb]
\caption{Right identities in the generic case in Theorem \ref{TheoremRL2}, ($n\geq5$).}
\label{Right(L2)}
\begin{tabular}{|l|p{2.4cm}|p{12cm}|}
\hline
\scriptsize Steps &\scriptsize Ordered triple &\scriptsize
Result\\ \hline
\scriptsize $1.$ &\scriptsize $\r[e_1,e_{1}]$ &\scriptsize
$a_{1,2}=0,a_{3,1}:=-2a+b,a_{k,2}=0,(3\leq k\leq n)$
$\implies$ $[e_1,e_{n+1}]=ae_1+a_{2,1}e_2+(-2a+b)e_3+\sum_{k=4}^n{a_{k,1}e_k},[e_2,e_{n+1}]=be_2.$\\ \hline
\scriptsize $2.$ &\scriptsize $\r[e_{i},e_{3}]$ &\scriptsize
$a_{1,3}=0,a_{1,i}=0,(4\leq i\leq n-1)$
$\implies$ $[e_j,e_{n+1}]=\sum_{k=2}^n{a_{k,j}e_k},(3\leq j\leq n-1).$\\ \hline
\scriptsize $3.$ &\scriptsize $\r[e_{i},e_{1}]$ &\scriptsize
$a_{2,i+1}=a_{3,i+1}=0,a_{i+1,i+1}:=(i-2)a+a_{3,3},a_{k,i+1}:=a_{k-1,i},(4\leq k\leq n,k\neq i+1,3\leq i\leq n-2),$
where $i$ is fixed
$\implies$ $[e_{j},e_{n+1}]=\left((j-3)a+a_{3,3}\right)e_j+\sum_{k=j+1}^n{a_{k-j+3,3}e_k},(4\leq j\leq n-1).$\\ \hline
\scriptsize $4.$ &\scriptsize $\r[e_{n-1},e_{1}]$ &\scriptsize
$a_{k,n}=0,(1\leq k\leq n-1),$ $a_{n,n}:=(n-3)a+a_{3,3}$
$\implies$ $[e_{n},e_{n+1}]=\left((n-3)a+a_{3,3}\right)e_n.$ Combining with $3.,$
$[e_{i},e_{n+1}]=\left((i-3)a+a_{3,3}\right)e_i+\sum_{k=i+1}^n{a_{k-i+3,3}e_k},(4\leq i\leq n).$\\ \hline
\scriptsize $5.$ &\scriptsize $\r[e_1,e_{3}]$ &\scriptsize
$a_{3,3}:=-a+b$
$\implies$ $[e_3,e_{n+1}]=a_{2,3}e_2+(-a+b)e_3+\sum_{k=4}^n{a_{k,3}e_k},
[e_{i},e_{n+1}]=\left((i-4)a+b\right)e_i+\sum_{k=i+1}^n{a_{k-i+3,3}e_k},(4\leq i\leq n).$\\ \hline
\scriptsize $6.$ &\scriptsize $\r_{e_1}\left([e_{n+1},e_{1}]\right)$ &\scriptsize
$[e_{n+1},e_2]=0$
$\implies$ $b_{k,2}=0,(1\leq k\leq n).$\\ \hline
\scriptsize $7.$ &\scriptsize $\r_{e_1}\left([e_3,e_{n+1}]\right)$ &\scriptsize
$b_{1,1}:=-a$
$\implies$ $[e_{n+1},e_1]=-ae_1+\sum_{k=2}^{n}b_{k,1}e_k.$\\ \hline
\scriptsize $8.$ &\scriptsize $\r_{e_1}\left([e_{1},e_{n+1}]\right)$ &\scriptsize
$b_{3,1}:=2a-b,b_{k-1,1}:=-a_{k-1,1},(5\leq k\leq n)$
$\implies$ $[e_{n+1},e_1]=-ae_1+b_{2,1}e_2+(2a-b)e_3-\sum_{k=4}^{n-1}{a_{k,1}e_k}+b_{n,1}e_n.$ \\ \hline
\scriptsize $9.$ &\scriptsize $\r_{e_i}\left([e_{3},e_{n+1}]\right)$ &\scriptsize
$b_{1,i}=0$
$\implies$ $[e_{n+1},e_{i}]=\sum_{k=2}^n{b_{k,i}e_k},(3\leq i\leq n).$  \\ \hline
\scriptsize $10.$ &\scriptsize $\r_{e_3}\left([e_{1},e_{n+1}]\right)$ &\scriptsize
$b_{3,3}:=a-b,b_{k-1,3}:=-a_{k-1,3},(5\leq k\leq n)$
$\implies$ $[e_{n+1},e_{3}]=b_{2,3}e_2+(a-b)e_3-\sum_{k=4}^{n-1}{a_{k,3}e_k}+b_{n,3}e_n.$\\ \hline
\scriptsize $11.$ &\scriptsize $\r_{e_i}\left([e_{1},e_{n+1}]\right)$ &\scriptsize
$b_{3,i}=0,b_{i,i}:=(4-i)a-b, b_{k-1,i}=0,(5\leq k\leq i),$ $b_{k-1,i}:=-a_{k-i+2,3},(i+2\leq k\leq n),$
where $i$ is fixed.
$\implies$  $[e_{n+1},e_i]=b_{2,i}e_2+\left((4-i)a-b\right)e_i-\sum_{k=i+1}^{n-1}{a_{k-i+3,3}e_k}+b_{n,i}e_n,(4\leq i\leq n-1).$ \\ \hline
\scriptsize $12.$ &\scriptsize $\r_{e_{n}}\left([e_{1},e_{n+1}]\right)$ &\scriptsize
$b_{3,n}=0,b_{k-1,n}=0,(5\leq k\leq n)$
$\implies$ $[e_{n+1},e_{n}]=b_{2,n}e_2+b_{n,n}e_n.$\\ \hline
\scriptsize $13.$ &\scriptsize $\r_{e_{3}}\left([e_{n+1},e_{1}]\right)$ &\scriptsize
$b_{2,4}:=a,b_{n,4}:=-a_{n-1,3}$
$\implies$ $[e_{n+1},e_{4}]=ae_2-be_4-\sum_{k=5}^n{a_{k-1,3}e_k}.$\\ \hline
\scriptsize $14.$ &\scriptsize $\r_{e_{i}}\left([e_{n+1},e_{1}]\right)$ &\scriptsize
$b_{2,i+1}=0,b_{n,i+1}:=-a_{n-i+2,3},(4\leq i\leq n-2)$
$\implies$ $[e_{n+1},e_{j}]=\left((4-j)a-b\right)e_j-\sum_{k=j+1}^n{a_{k-j+3,3}e_k},(5\leq j\leq n-1).$\\ \hline
\scriptsize $15.$ &\scriptsize $\r_{e_{n-1}}\left([e_{n+1},e_{1}]\right)$ &\scriptsize
$b_{2,n}=0,b_{n,n}:=(4-n)a-b$
$\implies$ $[e_{n+1},e_{n}]=\left((4-n)a-b\right)e_n.$ Combining with $14.,$
$[e_{n+1},e_{i}]=\left((4-i)a-b\right)e_i-\sum_{k=i+1}^n{a_{k-i+3,3}e_k},(5\leq i\leq n).$\\ \hline
\scriptsize $16.$ &\scriptsize $\r[e_{3},e_{n+1}]$ &\scriptsize
$a_{1,n+1}=0$
$\implies$ $[e_{n+1},e_{n+1}]=\sum_{k=2}^n{a_{k,n+1}e_k}.$\\ \hline
\scriptsize $17.$ &\scriptsize $\r[e_{1},e_{n+1}]$ &\scriptsize
$a_{3,n+1}=0,a_{k-1,n+1}=0,(5\leq k\leq n)$
$\implies$ $[e_{n+1},e_{n+1}]=a_{2,n+1}e_2+a_{n,n+1}e_n.$\\ \hline
\scriptsize $18.$ &\scriptsize $\r[e_{n+1},e_{n+1}]$ &\scriptsize
$a_{n,n+1}=0$
$\implies$ $[e_{n+1},e_{n+1}]=a_{2,n+1}e_2.$\\ \hline
\scriptsize $19.$ &\scriptsize $\r[e_{n+1},e_{3}]$ &\scriptsize
$b_{n,3}:=-a_{n,3},a_{4,3}:=b_{2,3}+\frac{(a-b)a_{2,3}}{a}$
$\implies$ $[e_3,e_{n+1}]=a_{2,3}e_2+(-a+b)e_3+\left(b_{2,3}+\frac{(a-b)a_{2,3}}{a}\right)e_4+\sum_{k=5}^n{a_{k,3}e_k},
[e_{i},e_{n+1}]=\left((i-4)a+b\right)e_i+\left(b_{2,3}+\frac{(a-b)a_{2,3}}{a}\right)e_{i+1}+\sum_{k=i+2}^n{a_{k-i+3,3}e_k},(4\leq i\leq n),[e_{n+1},e_{3}]=
b_{2,3}e_2+(a-b)e_3-\left(b_{2,3}+\frac{(a-b)a_{2,3}}{a}\right)e_4-\sum_{k=5}^{n}{a_{k,3}e_k},
[e_{n+1},e_{4}]=ae_2-be_4-\left(b_{2,3}+\frac{(a-b)a_{2,3}}{a}\right)e_5-\sum_{k=6}^n{a_{k-1,3}e_k},[e_{n+1},e_{j}]=
\left((4-j)a-b\right)e_j-\left(b_{2,3}+\frac{(a-b)a_{2,3}}{a}\right)e_{j+1}-\sum_{k=j+2}^n{a_{k-j+3,3}e_k},(5\leq j\leq n).$\\ \hline
\scriptsize $20.$ &\scriptsize $\r_{e_1}\left([e_{n+1},e_{n+1}]\right)$ &\scriptsize
$b_{n,1}:=-a_{n,1},b_{2,3}:=-a_{2,3}+\frac{(a-b)b_{2,1}+a(a_{4,1}+a_{2,1})}{2a-b}.$\\ \hline
\end{tabular}
\end{table}
\begin{theorem}\label{TheoremRL2Absorption} Applying the technique of ``absorption'' (see Section \ref{Solvable left Leibniz algebras}), we can further simplify the algebras 
in each of the four cases in Theorem \ref{TheoremRL2} as follows:
\begin{enumerate}[noitemsep, topsep=0pt]
\allowdisplaybreaks
\item[(1)] If $a\neq0,b\neq2a$ when $n=4$ and $b\neq(4-n)a,a\neq0,b\neq2a,(n\geq5),$ then
\begin{equation}
\left\{
\begin{array}{l}
\displaystyle  \nonumber [e_1,e_{n+1}]=ae_1+\left(a_{2,1}+a_{4,1}-A_{4,3}\right)e_2-(2a-b)e_3,[e_2,e_{n+1}]=be_2,
[e_3,e_{n+1}]=a_{2,3}e_2-\\
\displaystyle (a-b)e_3+\sum_{k=5}^n{a_{k,3}e_k},[e_{i},e_{n+1}]=\left((i-4)a+b\right)e_{i}+\sum_{k=i+2}^n{a_{k-i+3,3}e_k},(4\leq i\leq n),\\
\displaystyle[e_{n+1},e_{n+1}]=a_{2,n+1}e_2,[e_{n+1},e_1]=-ae_1+\left(b_{2,1}-A_{4,3}\right)e_2+(2a-b)e_3,\\
\displaystyle [e_{n+1},e_3]=\frac{(b-a)a_{2,3}}{a}e_2+(a-b)e_3-\sum_{k=5}^n{a_{k,3}e_k},[e_{n+1},e_4]=ae_2-be_4-\sum_{k=6}^n{a_{k-1,3}e_k},\\
\displaystyle [e_{n+1},e_j]=\left((4-j)a-b\right)e_j-\sum_{k=j+2}^n{a_{k-j+3,3}e_k},(5\leq j\leq n),\\
\displaystyle where\,\, A_{4,3}:=-\frac{b}{a}\cdot a_{2,3}+\frac{(a-b)b_{2,1}+a(a_{2,1}+a_{4,1})}{2a-b},
\end{array} 
\right.
\end{equation} 
$\r_{e_{n+1}}=\left[\begin{smallmatrix}
 a & 0 & 0 & 0&0&0&\cdots && 0&0 & 0\\
  a_{2,1}+a_{4,1}-A_{4,3} & b & a_{2,3}& 0 &0&0 & \cdots &&0  & 0& 0\\
  -2a+b & 0 & -a+b & 0 & 0&0 &\cdots &&0 &0& 0\\
  0 & 0 &  0 & b &0 &0 &\cdots&&0 &0 & 0\\
 0 & 0 & a_{5,3} & 0 & a+b  &0&\cdots & &0&0 & 0\\
  0 & 0 &\boldsymbol{\cdot} & a_{5,3} & 0 &2a+b&\cdots & &0&0 & 0\\
    0 & 0 &\boldsymbol{\cdot} & \boldsymbol{\cdot} & \ddots &0&\ddots &&\vdots&\vdots &\vdots\\
  \vdots & \vdots & \vdots &\vdots &  &\ddots&\ddots &\ddots &\vdots&\vdots & \vdots\\
 0 & 0 & a_{n-2,3}& a_{n-3,3}& \cdots&\cdots &a_{5,3}&0&(n-6)a+b &0& 0\\
0 & 0 & a_{n-1,3}& a_{n-2,3}& \cdots&\cdots &\boldsymbol{\cdot}&a_{5,3}&0 &(n-5)a+b& 0\\
 0 & 0 & a_{n,3}& a_{n-1,3}& \cdots&\cdots &\boldsymbol{\cdot}&\boldsymbol{\cdot}&a_{5,3} &0& (n-4)a+b
\end{smallmatrix}\right].$
\item[(2)] If $b:=(4-n)a,a\neq0,(n\geq5),$
then the brackets for the algebra are  
\begin{equation}
\left\{
\begin{array}{l}
\displaystyle  \nonumber [e_1,e_{n+1}]=ae_1+\left(a_{2,1}+a_{4,1}-A_{4,3}\right)e_2+(2-n)ae_3,[e_2,e_{n+1}]=(4-n)ae_2,\\
\displaystyle [e_3,e_{n+1}]=a_{2,3}e_2+(3-n)ae_3+\sum_{k=5}^n{a_{k,3}e_k},[e_{i},e_{n+1}]=\left(i-n\right)ae_{i}+\sum_{k=i+2}^n{a_{k-i+3,3}e_k},\\
\displaystyle (4\leq i\leq n-1),
[e_{n+1},e_{n+1}]=a_{n,n+1}e_n,[e_{n+1},e_1]=-ae_1+(b_{2,1}-A_{4,3})e_2+(n-2)ae_3,\\
\displaystyle [e_{n+1},e_3]=(3-n)a_{2,3}e_2+(n-3)ae_3-\sum_{k=5}^n{a_{k,3}e_k},[e_{n+1},e_4]=ae_2+(n-4)ae_4-\\
\displaystyle \sum_{k=6}^n{a_{k-1,3}e_k}, [e_{n+1},e_j]=\left(n-j\right)ae_j-\sum_{k=j+2}^n{a_{k-j+3,3}e_k},(5\leq j\leq n-1),\\
\displaystyle where\,\,A_{4,3}:=(n-4)a_{2,3}+\frac{(n-3)b_{2,1}+a_{2,1}+a_{4,1}}{n-2},
\end{array} 
\right.
\end{equation}
$\r_{e_{n+1}}=\left[\begin{smallmatrix}
 a & 0 & 0 & 0&0&0&\cdots && 0&0 & 0\\
  a_{2,1}+a_{4,1}-A_{4,3} & (4-n)a & a_{2,3}& 0 &0&0 & \cdots &&0  & 0& 0\\
  (2-n)a & 0 & (3-n)a & 0 & 0&0 &\cdots &&0 &0& 0\\
  0 & 0 &  0 & (4-n)a &0 &0 &\cdots&&0 &0 & 0\\
 0 & 0 & a_{5,3} & 0 & (5-n)a  &0&\cdots & &0&0 & 0\\
  0 & 0 &\boldsymbol{\cdot} & a_{5,3} & 0 &(6-n)a&\cdots & &0&0 & 0\\
    0 & 0 &\boldsymbol{\cdot} & \boldsymbol{\cdot} & \ddots &0&\ddots &&\vdots&\vdots &\vdots\\
  \vdots & \vdots & \vdots &\vdots &  &\ddots&\ddots &\ddots &\vdots&\vdots & \vdots\\
 0 & 0 & a_{n-2,3}& a_{n-3,3}& \cdots&\cdots &a_{5,3}&0&-2a &0& 0\\
0 & 0 & a_{n-1,3}& a_{n-2,3}& \cdots&\cdots &\boldsymbol{\cdot}&a_{5,3}&0 &-a& 0\\
 0 & 0 & a_{n,3}& a_{n-1,3}& \cdots&\cdots &\boldsymbol{\cdot}&\boldsymbol{\cdot}&a_{5,3} &0&0
\end{smallmatrix}\right].$
 
\item[(3)] If $a=0$ and $b\neq0,(n\geq4),$ then
\begin{equation}
\left\{
\begin{array}{l}
\displaystyle  \nonumber [e_1,e_{n+1}]=a_{2,1}e_2+be_3, [e_2,e_{n+1}]=be_2,
[e_{i},e_{n+1}]=be_{i}+\sum_{k=i+2}^n{a_{k-i+3,3}e_k},(3\leq i\leq n),\\
\displaystyle [e_{n+1},e_1]=b_{2,1}e_2-be_3,[e_{n+1},e_3]=b_{2,1}e_2-be_3-\sum_{k=5}^n{a_{k,3}e_k},\\
\displaystyle [e_{n+1},e_j]=-be_j-\sum_{k=j+2}^n{a_{k-j+3,3}e_k},(4\leq j\leq n),
\end{array} 
\right.
\end{equation} 
$\r_{e_{n+1}}=\left[\begin{smallmatrix}
 0 & 0 & 0 & 0&0&0&\cdots && 0&0 & 0\\
  a_{2,1} & b & 0& 0 &0&0 & \cdots &&0  & 0& 0\\
  b & 0 & b & 0 & 0&0 &\cdots &&0 &0& 0\\
  0 & 0 &  0 & b &0 &0 &\cdots&&0 &0 & 0\\
 0 & 0 & a_{5,3} & 0 & b  &0&\cdots & &0&0 & 0\\
  0 & 0 &\boldsymbol{\cdot} & a_{5,3} & 0 &b&\cdots & &0&0 & 0\\
    0 & 0 &\boldsymbol{\cdot} & \boldsymbol{\cdot} & \ddots &0&\ddots &&\vdots&\vdots &\vdots\\
  \vdots & \vdots & \vdots &\vdots &  &\ddots&\ddots &\ddots &\vdots&\vdots & \vdots\\
 0 & 0 & a_{n-2,3}& a_{n-3,3}& \cdots&\cdots &a_{5,3}&0&b &0& 0\\
0 & 0 & a_{n-1,3}& a_{n-2,3}& \cdots&\cdots &\boldsymbol{\cdot}&a_{5,3}&0 &b& 0\\
 0 & 0 & a_{n,3}& a_{n-1,3}& \cdots&\cdots &\boldsymbol{\cdot}&\boldsymbol{\cdot}&a_{5,3} &0&b
\end{smallmatrix}\right].$

   \allowdisplaybreaks
\item[(4)] If $b:=2a,a\neq0,(n\geq4),$ then
\begin{equation}
\left\{
\begin{array}{l}
\displaystyle  \nonumber [e_1,e_{n+1}]=ae_1+a_{2,1}e_2,[e_2,e_{n+1}]=2ae_2,
[e_3,e_{n+1}]=a_{2,3}e_2+ae_3+\sum_{k=5}^n{a_{k,3}e_k},\\
\displaystyle [e_{i},e_{n+1}]=\left(i-2\right)ae_{i}+\sum_{k=i+2}^n{a_{k-i+3,3}e_k},(4\leq i\leq n),[e_{n+1},e_1]=-ae_1+a_{2,1}e_2,\\
\displaystyle [e_{n+1},e_3]=a_{2,3}e_2-ae_3-\sum_{k=5}^n{a_{k,3}e_k},[e_{n+1},e_4]=ae_2-2ae_4-\sum_{k=6}^n{a_{k-1,3}e_k},\\
\displaystyle [e_{n+1},e_j]=\left(2-j\right)ae_j-\sum_{k=j+2}^n{a_{k-j+3,3}e_k},(5\leq j\leq n),
\end{array} 
\right.
\end{equation} 
$\r_{e_{n+1}}=\left[\begin{smallmatrix}
 a & 0 & 0 & 0&0&0&\cdots && 0&0 & 0\\
  a_{2,1} & 2a & a_{2,3}& 0 &0&0 & \cdots &&0  & 0& 0\\
 0 & 0 & a & 0 & 0&0 &\cdots &&0 &0& 0\\
  0 & 0 &  0 & 2a &0 &0 &\cdots&&0 &0 & 0\\
 0 & 0 & a_{5,3} & 0 & 3a  &0&\cdots & &0&0 & 0\\
  0 & 0 &\boldsymbol{\cdot} & a_{5,3} & 0 &4a&\cdots & &0&0 & 0\\
    0 & 0 &\boldsymbol{\cdot} & \boldsymbol{\cdot} & \ddots &0&\ddots &&\vdots&\vdots &\vdots\\
  \vdots & \vdots & \vdots &\vdots &  &\ddots&\ddots &\ddots &\vdots&\vdots & \vdots\\
 0 & 0 & a_{n-2,3}& a_{n-3,3}& \cdots&\cdots &a_{5,3}&0&(n-4)a &0& 0\\
0 & 0 & a_{n-1,3}& a_{n-2,3}& \cdots&\cdots &\boldsymbol{\cdot}&a_{5,3}&0 &(n-3)a& 0\\
 0 & 0 & a_{n,3}& a_{n-1,3}& \cdots&\cdots &\boldsymbol{\cdot}&\boldsymbol{\cdot}&a_{5,3} &0& (n-2)a
\end{smallmatrix}\right].$

\end{enumerate}
\end{theorem}
\begin{proof}
\begin{enumerate}[noitemsep, topsep=0pt]
\item[(1)] Suppose $a\neq0,b\neq2a,(n=4)$ and $b\neq(4-n)a,a\neq0,b\neq2a,(n\geq5).$ The
left multiplication operator (not a derivation) restricted to the nilradical is given below:
$$\L_{e_{n+1}}=\left[\begin{smallmatrix}
 -a & 0 & 0 & 0&0&&\cdots &0& \cdots&0 & 0&0 \\
  b_{2,1} & 0 & B_{2,3}& a &0& & \cdots &0&\cdots  & 0& 0&0\\
  2a-b & 0 & a-b & 0 & 0& &\cdots &0&\cdots &0& 0&0 \\
  -a_{4,1} & 0 &  -A_{4,3} & -b &0 & &\cdots&0&\cdots &0 & 0&0\\
  -a_{5,1} & 0 & -a_{5,3} & -A_{4,3} & -a-b  &&\cdots &0 &\cdots&0 & 0&0 \\
 \boldsymbol{\cdot} & \boldsymbol{\cdot} & \boldsymbol{\cdot} & -a_{5,3} & -A_{4,3}  &\ddots& &\vdots &&\vdots & \vdots&\vdots \\
  \vdots & \vdots & \vdots &\vdots &\vdots  &\ddots& \ddots&\vdots &&\vdots & \vdots&\vdots \\
  -a_{i,1} & 0 & -a_{i,3} & -a_{i-1,3} & -a_{i-2,3}  &\cdots&-A_{4,3}& (4-i)a-b&\cdots&0 & 0&0\\
   \vdots  & \vdots  & \vdots &\vdots &\vdots&&\vdots &\vdots &&\vdots  & \vdots&\vdots \\
 -a_{n-1,1} & 0 & -a_{n-1,3}& -a_{n-2,3}& -a_{n-3,3}&\cdots &-a_{n-i+3,3} &-a_{n-i+2,3}&\cdots&-A_{4,3} &(5-n)a-b& 0\\
 -a_{n,1} & 0 & -a_{n,3}& -a_{n-1,3}& -a_{n-2,3}&\cdots &-a_{n-i+4,3}  &-a_{n-i+3,3}&\cdots&-a_{5,3} &-A_{4,3}& (4-n)a-b
\end{smallmatrix}\right].$$
\allowdisplaybreaks
\begin{itemize}
\item The transformation $e^{\prime}_k=e_k,(1\leq k\leq n),e^{\prime}_{n+1}=e_{n+1}-A_{4,3}e_1$
removes $A_{4,3}$ in $\r_{e_{n+1}}$ and $-A_{4,3}$ in $\L_{e_{n+1}}$ from the $(i,i-1)^{st}$ entries, where $(4\leq i\leq n),$ 
 but it affects other entries as well,
such as
the entry in the $(2,1)^{st}$ position in $\r_{e_{n+1}}$ and $\L_{e_{n+1}},$
which we change to $a_{2,1}-A_{4,3}$ and $b_{2,1}-A_{4,3},$ respectively.
It also changes the entry in the $(2,3)^{rd}$ position in $\L_{e_{n+1}}$ to 
$\frac{(b-a)a_{2,3}}{a}.$
At the same time, it affects the coefficient in front of $e_2$ in the bracket $[e_{n+1},e_{n+1}],$ which we change back to $a_{2,n+1}$.
\item Then we apply the transformation $e^{\prime}_i=e_i,(1\leq i\leq n),e^{\prime}_{n+1}=e_{n+1}+\sum_{k=3}^{n-1}a_{k+1,1}e_{k}$
to remove $a_{k+1,1}$ and $-a_{k+1,1}$ from the entries in the $(k+1,1)^{st}$
positions, where $(3\leq k\leq n-1)$ in $\r_{e_{n+1}}$ and $\L_{e_{n+1}},$ respectively. The transformation also changes the entry in the $(2,1)^{st}$ position in
$\r_{e_{n+1}}$ to $a_{2,1}+a_{4,1}-A_{4,3}$
and the coefficient in front of $e_2$ in $[e_{n+1},e_{n+1}],$ which
we name back by $a_{2,n+1}.$
\end{itemize}
\item[(2)] Suppose $b:=(4-n)a,a\neq0,(n\geq5).$ The left multiplication operator (not a derivation) restricted to the nilradical is

$\L_{e_{n+1}}=\left[\begin{smallmatrix}
 -a & 0 & 0 & 0&0&&\cdots &0& \cdots&0 & 0&0 \\
  b_{2,1} & 0 & B_{2,3}& a &0& & \cdots &0&\cdots  & 0& 0&0\\
  (n-2)a & 0 & (n-3)a & 0 & 0& &\cdots &0&\cdots &0& 0&0 \\
  -a_{4,1} & 0 &  -A_{4,3} & (n-4)a &0 & &\cdots&0&\cdots &0 & 0&0\\
  -a_{5,1} & 0 & -a_{5,3} & -A_{4,3} & (n-5)a  &&\cdots &0 &\cdots&0 & 0&0 \\
 \boldsymbol{\cdot} & \boldsymbol{\cdot} & \boldsymbol{\cdot} & -a_{5,3} & -A_{4,3}  &\ddots& &\vdots &&\vdots & \vdots&\vdots \\
  \vdots & \vdots & \vdots &\vdots &\vdots  &\ddots& \ddots&\vdots &&\vdots & \vdots&\vdots \\
  -a_{i,1} & 0 & -a_{i,3} & -a_{i-1,3} & -a_{i-2,3}  &\cdots&-A_{4,3}& (n-i)a&\cdots&0 & 0&0\\
   \vdots  & \vdots  & \vdots &\vdots &\vdots&&\vdots &\vdots &&\vdots  & \vdots&\vdots \\
 -a_{n-1,1} & 0 & -a_{n-1,3}& -a_{n-2,3}& -a_{n-3,3}&\cdots &-a_{n-i+3,3} &-a_{n-i+2,3}&\cdots&-A_{4,3} &a& 0\\
 -a_{n,1} & 0 & -a_{n,3}& -a_{n-1,3}& -a_{n-2,3}&\cdots &-a_{n-i+4,3}  &-a_{n-i+3,3}&\cdots&-a_{5,3} &-A_{4,3}& 0
\end{smallmatrix}\right].$
\begin{itemize}
\item Applying the transformation $e^{\prime}_k=e_k,(1\leq k\leq n),e^{\prime}_{n+1}=e_{n+1}-A_{4,3}e_1,$
we remove $A_{4,3}$ and $-A_{4,3}$ in $\r_{e_{n+1}}$ and $\L_{e_{n+1}},$ respectively, from the $(i,i-1)^{st}$ positions, where $(4\leq i\leq n),$ 
 but other entries are affected as well,
such as
the entry in the $(2,1)^{st}$ position in $\r_{e_{n+1}}$ and $\L_{e_{n+1}},$
which we change to $a_{2,1}-A_{4,3}$ and $b_{2,1}-A_{4,3},$ respectively.
It also changes the entry in the $(2,3)^{rd}$ position in $\L_{e_{n+1}}$ to $(3-n)a_{2,3}.$
At the same time, it affects the coefficient in front of $e_2$ in the bracket $[e_{n+1},e_{n+1}],$ which we change back to $a_{2,n+1}$.
\item Then we apply the transformation $e^{\prime}_i=e_i,(1\leq i\leq n),e^{\prime}_{n+1}=e_{n+1}+\sum_{k=3}^{n-1}a_{k+1,1}e_{k}$
to remove $a_{k+1,1}$ in $\r_{e_{n+1}}$ and $-a_{k+1,1}$ in $\L_{e_{n+1}}$ from the entries in the $(k+1,1)^{st}$
positions, where $(3\leq k\leq n-1).$ It changes the entry in the $(2,1)^{st}$ position in
$\r_{e_{n+1}}$ to $a_{2,1}+a_{4,1}-A_{4,3}$
and the coefficient in front of $e_2$ in $[e_{n+1},e_{n+1}],$ which
we rename back by $a_{2,n+1}.$
\item The transformation $e^{\prime}_j=e_j,(1\leq j\leq n),e^{\prime}_{n+1}=e_{n+1}+\frac{a_{2,n+1}}{(n-4)a}e_2$ 
removes the coefficient $a_{2,n+1}$ in front of $e_2$ in $[e_{n+1},e_{n+1}]$ and we prove the result.
\end{itemize}
\item[(3)] Suppose $a=0$ and $b\neq0,(n\geq4).$ The left multiplication operator (not a derivation)
restricted to the nilradical is given below:

$\L_{e_{n+1}}=\left[\begin{smallmatrix}
 0 & 0 & 0 & 0&0&&\cdots &0& \cdots&0 & 0&0 \\
  b_{2,1} & 0 &b_{2,1}& 0 &0& & \cdots &0&\cdots  & 0& 0&0\\
  -b & 0 & -b & 0 & 0& &\cdots &0&\cdots &0& 0&0 \\
  -a_{4,1} & 0 &  -a_{4,3} & -b &0 & &\cdots&0&\cdots &0 & 0&0\\
  -a_{5,1} & 0 & -a_{5,3} & -a_{4,3} & -b  &&\cdots &0 &\cdots&0 & 0&0 \\
 \boldsymbol{\cdot} & \boldsymbol{\cdot} & \boldsymbol{\cdot} & -a_{5,3} & -a_{4,3}  &\ddots& &\vdots &&\vdots & \vdots&\vdots \\
  \vdots & \vdots & \vdots &\vdots &\vdots  &\ddots& \ddots&\vdots &&\vdots & \vdots&\vdots \\
  -a_{i,1} & 0 & -a_{i,3} & -a_{i-1,3} & -a_{i-2,3}  &\cdots&-a_{4,3}& -b&\cdots&0 & 0&0\\
   \vdots  & \vdots  & \vdots &\vdots &\vdots&&\vdots &\vdots &&\vdots  & \vdots&\vdots \\
 -a_{n-1,1} & 0 & -a_{n-1,3}& -a_{n-2,3}& -a_{n-3,3}&\cdots &-a_{n-i+3,3} &-a_{n-i+2,3}&\cdots&-a_{4,3} &-b& 0\\
 -a_{n,1} & 0 & -a_{n,3}& -a_{n-1,3}& -a_{n-2,3}&\cdots &-a_{n-i+4,3}  &-a_{n-i+3,3}&\cdots&-a_{5,3} &-a_{4,3}& -b
\end{smallmatrix}\right].$
\begin{itemize}
\item The transformation $e^{\prime}_k=e_k,(1\leq k\leq n),e^{\prime}_{n+1}=e_{n+1}-a_{4,3}e_1$
removes $a_{4,3}$ in $\r_{e_{n+1}}$ and $-a_{4,3}$ in $\L_{e_{n+1}}$ from the $(i,i-1)^{st}$ entries, where $(4\leq i\leq n),$ 
 but it affects other entries as well,
such as
the entry in the $(2,1)^{st}$ position in $\r_{e_{n+1}}$ and $\L_{e_{n+1}},$
which we change to $a_{2,1}-a_{4,3}$ and $b_{2,1}-a_{4,3},$ respectively.
It also changes the entry in the $(2,3)^{rd}$ position in $\L_{e_{n+1}}$ to $b_{2,1}-a_{4,3}$.
At the same time, it affects the coefficient in front of $e_2$ in the bracket $[e_{n+1},e_{n+1}],$ which we change back to $a_{2,n+1}$.
\item Then we apply the transformation $e^{\prime}_i=e_i,(1\leq i\leq n),e^{\prime}_{n+1}=e_{n+1}+\sum_{k=3}^{n-1}a_{k+1,1}e_{k}$
to remove $a_{k+1,1}$ and $-a_{k+1,1}$ in $\r_{e_{n+1}}$ and $\L_{e_{n+1}},$ respectively, from the entries in the $(k+1,1)^{st}$
positions, where $(3\leq k\leq n-1).$ It changes the entry in the $(2,1)^{st}$ position in
$\r_{e_{n+1}}$ to $a_{2,1}+a_{4,1}-a_{4,3}$
and the coefficient in front of $e_2$ in $[e_{n+1},e_{n+1}],$ which
we rename back by $a_{2,n+1}.$
\item We assign $a_{2,1}+a_{4,1}-a_{4,3}:=a_{2,1}$ and $b_{2,1}-a_{4,3}:=b_{2,1}$
and finally we apply the transformation $e^{\prime}_j=e_j,(1\leq j\leq n),e^{\prime}_{n+1}=e_{n+1}-\frac{a_{2,n+1}}{b}e_2$ 
to remove the coefficient $a_{2,n+1}$ in front of $e_2$ in $[e_{n+1},e_{n+1}].$
\end{itemize}

\item[(4)] Suppose $b:=2a,a\neq0,(n\geq4).$ The left multiplication operator (not a derivation)
restricted to the nilradical is

$\L_{e_{n+1}}=\left[\begin{smallmatrix}
 -a & 0 & 0 & 0&0&&\cdots &0& \cdots&0 & 0&0 \\
  a_{2,1}+a_{4,1} & 0 & b_{2,3}& a &0& & \cdots &0&\cdots  & 0& 0&0\\
  0 & 0 & -a & 0 & 0& &\cdots &0&\cdots &0& 0&0 \\
 - a_{4,1} & 0 &  -A_{4,3} & -2a &0 & &\cdots&0&\cdots &0 & 0&0\\
  -a_{5,1} & 0 & -a_{5,3} & -A_{4,3} & -3a  &&\cdots &0 &\cdots&0 & 0&0 \\
 \boldsymbol{\cdot} & \boldsymbol{\cdot} & \boldsymbol{\cdot} & -a_{5,3} & -A_{4,3}  &\ddots& &\vdots &&\vdots & \vdots&\vdots \\
  \vdots & \vdots & \vdots &\vdots &\vdots  &\ddots& \ddots&\vdots &&\vdots & \vdots&\vdots \\
  -a_{i,1} & 0 & -a_{i,3} & -a_{i-1,3} & -a_{i-2,3}  &\cdots&-A_{4,3}& (2-i)a&\cdots&0 & 0&0\\
   \vdots  & \vdots  & \vdots &\vdots &\vdots&&\vdots &\vdots &&\vdots  & \vdots&\vdots \\
 -a_{n-1,1} & 0 & -a_{n-1,3}& -a_{n-2,3}& -a_{n-3,3}&\cdots &-a_{n-i+3,3} &-a_{n-i+2,3}&\cdots&-A_{4,3} &(3-n)a& 0\\
 -a_{n,1} & 0 & -a_{n,3}& -a_{n-1,3}& -a_{n-2,3}&\cdots &-a_{n-i+4,3}  &-a_{n-i+3,3}&\cdots&-a_{5,3} &-A_{4,3}& (2-n)a
\end{smallmatrix}\right].$
\begin{itemize}
\item The transformation $e^{\prime}_k=e_k,(1\leq k\leq n),e^{\prime}_{n+1}=e_{n+1}-A_{4,3}e_1,$
where $A_{4,3}:=b_{2,3}-a_{2,3}$
removes $A_{4,3}$ in $\r_{e_{n+1}}$ and $-A_{4,3}$ in $\L_{e_{n+1}}$ from the $(i,i-1)^{st}$ positions, where $(4\leq i\leq n),$ 
 but it affects other entries as well,
such as
the entry in the $(2,1)^{st}$ position in $\r_{e_{n+1}}$ and $\L_{e_{n+1}},$
which we change to $a_{2,1}-A_{4,3}$ and $a_{2,1}+a_{4,1}-A_{4,3},$ respectively.
It also changes the entry in the $(2,3)^{rd}$ position in $\L_{e_{n+1}}$ to $a_{2,3}$.
At the same time, the transformation affects the coefficient in front of $e_2$ in the bracket $[e_{n+1},e_{n+1}],$ which we change back to $a_{2,n+1}$.
\item Applying the transformation $e^{\prime}_i=e_i,(1\leq i\leq n),e^{\prime}_{n+1}=e_{n+1}+\sum_{k=3}^{n-1}a_{k+1,1}e_{k},$
we remove $a_{k+1,1}$ and $-a_{k+1,1}$ in $\r_{e_{n+1}}$ snd $\L_{e_{n+1}},$ respectively, from the entries in the $(k+1,1)^{st}$
positions, where $(3\leq k\leq n-1).$ It changes the entry in the $(2,1)^{st}$ position in
$\r_{e_{n+1}}$ to $a_{2,1}+a_{4,1}-A_{4,3}$
and the coefficient in front of $e_2$ in $[e_{n+1},e_{n+1}],$ which
we rename back by $a_{2,n+1}.$ Then we rename the entries in the $(2,1)^{st}$ positions in $\r_{e_{n+1}}$
and $\L_{e_{n+1}}$ by $a_{2,1}$.
\item The transformation $e^{\prime}_j=e_j,(1\leq j\leq n),e^{\prime}_{n+1}=e_{n+1}-\frac{a_{2,n+1}}{2a}e_2$ 
removes the coefficient $a_{2,n+1}$ in front of $e_2$ in $[e_{n+1},e_{n+1}]$ and we prove the result.
\end{itemize}
\end{enumerate}
\end{proof}
\allowdisplaybreaks
 \begin{theorem}\label{RL2(Change of Basis)} There are four solvable
indecomposable right Leibniz algebras up to isomorphism with a codimension one nilradical
$\mathcal{L}^2,(n\geq4),$ which are given below:
\begin{equation}
\begin{array}{l}
\displaystyle \nonumber (i)\,\,\, \g_{n+1,1}: [e_1,e_{n+1}]=e_1+(a-2)e_3,[e_2,e_{n+1}]=ae_2,
[e_{i},e_{n+1}]=\left(a+i-4\right)e_{i},\\
\displaystyle (3\leq i\leq n),[e_{n+1},e_1]=-e_1-(a-2)e_3,[e_{n+1},e_3]=(1-a)e_3,[e_{n+1},e_4]=e_2-ae_4,\\
\displaystyle [e_{n+1},e_j]=\left(4-j-a\right)e_j,(5\leq j\leq n);DS=[n+1,n,n-2,0],LS=[n+1,n,n,...],\\
\displaystyle  \nonumber(ii)\,\,\, \g_{n+1,2}: [e_1,e_{n+1}]=e_1-2e_3+\delta e_5,
[e_{i},e_{n+1}]=(i-4)e_{i},(3\leq i\leq n),\\
\displaystyle [e_{n+1},e_1]=-e_1+2e_3-\delta e_5,[e_{n+1},e_3]=e_3,[e_{n+1},e_4]=e_2,[e_{n+1},e_j]=(4-j)e_j,\\
\displaystyle (\delta=\pm1,5\leq j\leq n,n\geq5);DS=[n+1,n,n-2,0],LS=[n+1,n,n,...],\\
\displaystyle (iii)\,\,\,\g_{n+1,3}:
 [e_1,e_{n+1}]=e_1+(2-n)e_3,[e_2,e_{n+1}]=(4-n)e_2,
[e_{i},e_{n+1}]=(i-n)e_{i},\\
\displaystyle (3\leq i\leq n-1),[e_{n+1},e_{n+1}]=\delta e_n,[e_{n+1},e_1]=-e_1+(n-2)e_3,[e_{n+1},e_3]=(n-3)e_3,\\
\displaystyle [e_{n+1},e_4]=e_2+(n-4)e_4, [e_{n+1},e_j]=(n-j)e_j,(\delta=\pm1,5\leq j\leq n-1, n\geq5),\\
\displaystyle DS=[n+1,n,n-2,0],LS=[n+1,n,n,...],\\ 
\displaystyle (iv)\,\,\,\g_{n+1,4}: [e_1,e_{n+1}]=e_3,[e_2,e_{n+1}]=e_2,
[e_{i},e_{n+1}]=e_{i}+\epsilon e_{i+2}+\sum_{k=i+3}^n{b_{k-i-2}e_k},\\
\displaystyle [e_{n+1},e_1]=-e_3,[e_{n+1},e_i]=-e_i-\epsilon e_{i+2}-\sum_{k=i+3}^n{b_{k-i-2}e_k},(\epsilon=0,\pm1,3\leq i\leq n),\\  
\displaystyle DS=[n+1,n-1,0],LS=[n+1,n-1,n-1,...].
\end{array} 
\end{equation} 
\end{theorem}
\vskip 20pt    
\begin{proof}
One applies the change of basis transformations keeping the nilradical $\mathcal{L}^2$ given in $(\ref{L2})$ unchanged.
\begin{enumerate}[noitemsep, topsep=1pt]
\allowdisplaybreaks
\item[(1)] Suppose $a\neq0,b\neq2a,(n=4)$ and $b\neq(4-n)a,a\neq0,b\neq2a,(n\geq5).$ We have the right (a derivation) and the left (not a derivation) multiplication operators restricted to the nilradical: 
$$\r_{e_{n+1}}=\left[\begin{smallmatrix}
 a & 0 & 0 & 0&0&0&\cdots && 0&0 & 0\\
  a_{2,1}+a_{4,1}-A_{4,3} & b & a_{2,3}& 0 &0&0 & \cdots &&0  & 0& 0\\
  -2a+b & 0 & -a+b & 0 & 0&0 &\cdots &&0 &0& 0\\
  0 & 0 &  0 & b &0 &0 &\cdots&&0 &0 & 0\\
 0 & 0 & a_{5,3} & 0 & a+b  &0&\cdots & &0&0 & 0\\
  0 & 0 &\boldsymbol{\cdot} & a_{5,3} & 0 &2a+b&\cdots & &0&0 & 0\\
    0 & 0 &\boldsymbol{\cdot} & \boldsymbol{\cdot} & \ddots &0&\ddots &&\vdots&\vdots &\vdots\\
  \vdots & \vdots & \vdots &\vdots &  &\ddots&\ddots &\ddots &\vdots&\vdots & \vdots\\
 0 & 0 & a_{n-2,3}& a_{n-3,3}& \cdots&\cdots &a_{5,3}&0&(n-6)a+b &0& 0\\
0 & 0 & a_{n-1,3}& a_{n-2,3}& \cdots&\cdots &\boldsymbol{\cdot}&a_{5,3}&0 &(n-5)a+b& 0\\
 0 & 0 & a_{n,3}& a_{n-1,3}& \cdots&\cdots &\boldsymbol{\cdot}&\boldsymbol{\cdot}&a_{5,3} &0& (n-4)a+b
\end{smallmatrix}\right],$$
 $$\L_{e_{n+1}}=\left[\begin{smallmatrix}
 -a & 0 & 0 & 0&0&0&\cdots && 0&0 & 0\\
  b_{2,1}-A_{4,3} & 0 & \frac{(b-a)a_{2,3}}{a}& a &0&0 & \cdots &&0  & 0& 0\\
  2a-b & 0 & a-b & 0 & 0&0 &\cdots &&0 &0& 0\\
  0 & 0 &  0 & -b &0 &0 &\cdots&&0 &0 & 0\\
 0 & 0 & -a_{5,3} & 0 & -a-b  &0&\cdots & &0&0 & 0\\
  0 & 0 &\boldsymbol{\cdot} & -a_{5,3} & 0 &-2a-b&\cdots & &0&0 & 0\\
    0 & 0 &\boldsymbol{\cdot} & \boldsymbol{\cdot} & \ddots &0&\ddots &&\vdots&\vdots &\vdots\\
  \vdots & \vdots & \vdots &\vdots &  &\ddots&\ddots &\ddots &\vdots&\vdots & \vdots\\
 0 & 0 & -a_{n-2,3}& -a_{n-3,3}& \cdots&\cdots &-a_{5,3}&0&(6-n)a-b &0& 0\\
0 & 0 & -a_{n-1,3}& -a_{n-2,3}& \cdots&\cdots &\boldsymbol{\cdot}&-a_{5,3}&0 &(5-n)a-b& 0\\
 0 & 0 & -a_{n,3}& -a_{n-1,3}& \cdots&\cdots &\boldsymbol{\cdot}&\boldsymbol{\cdot}&-a_{5,3} &0& (4-n)a-b
\end{smallmatrix}\right].$$
\begin{itemize}[noitemsep, topsep=0pt]
\allowdisplaybreaks 
\item We apply the transformation $e^{\prime}_1=e_1,e^{\prime}_2=e_2,e^{\prime}_i=e_i-\frac{a_{k-i+3,3}}{(k-i)a}e_k,(3\leq i\leq n-2,
i+2\leq k\leq n,n\geq5),e^{\prime}_{j}=e_{j},(n-1\leq j\leq n+1),$ where $k$ is fixed, renaming all the affected entries back.
This transformation removes $a_{5,3},a_{6,3},...,a_{n,3}$ in $\r_{e_{n+1}}$ and $-a_{5,3},-a_{6,3},...,-a_{n,3}$ in $\L_{e_{n+1}}.$
Besides it introduces the entries in the $(5,1)^{st},(6,1)^{st},...,(n,1)^{st}$ positions in $\r_{e_{n+1}}$ and $\L_{e_{n+1}},$ 
which we call by $a_{5,1},a_{6,1},...,a_{n,1}$ and $-a_{5,1},-a_{6,1},...,-a_{n,1},$ respectively.
\item Applying the transformation $e^{\prime}_1=e_1-\frac{1}{a}\left(b_{2,1}-A_{4,3}+\frac{(2a-b)a_{2,3}}{a}\right)e_2,e^{\prime}_2=e_2,
e^{\prime}_3=e_3-\frac{a_{2,3}}{a}e_2,e^{\prime}_{i}=e_{i},
(4\leq i\leq n,n\geq4),e^{\prime}_{n+1}=e_{n+1}-\frac{a_{2,n+1}}{a}e_4+\sum_{k=5}^{n-1}a_{k+1,1}e_{k},$
we
remove $a_{2,1}+a_{4,1}-A_{4,3}$ and $b_{2,1}-A_{4,3}$ from the $(2,1)^{st}$ positions in $\r_{e_{n+1}}$ and $\L_{e_{n+1}},$
respectively. It also removes $a_{2,3}$ and $\frac{(b-a)a_{2,3}}{a}$
from the entries in the $(2,3)^{rd}$ positions in $\r_{e_{n+1}}$ and $\L_{e_{n+1}};$ 
$a_{k+1,1}$ in $\r_{e_{n+1}}$ and $-a_{k+1,1}$ in $\L_{e_{n+1}}$ from the entries in the $(k+1,1)^{st},$ $(5\leq k\leq n-1)$
positions and the coefficient $a_{2,n+1}$ in front of $e_2$ in $[e_{n+1},e_{n+1}]$.
This transformation affects the entries in the $(5,1)^{st}$ positions in $\r_{e_{n+1}}$ and $\L_{e_{n+1}},$
which we change back.
\item Then we scale $a$ to unity applying the transformation $e^{\prime}_i=e_i,(1\leq i\leq n,n\geq4),e^{\prime}_{n+1}=\frac{e_{n+1}}{a}.$ 
Renaming $\frac{b}{a}$ by $b$ and $\frac{a_{5,1}}{a}$ by $c,$ we obtain a continuous family of Leibniz algebras, where $c=0$ when $n=4:$
\begin{equation}
\left\{
\begin{array}{l}
\displaystyle  \nonumber [e_1,e_{n+1}]=e_1+(b-2)e_3+ce_5,[e_2,e_{n+1}]=be_2,[e_{i},e_{n+1}]=\left(b+i-4\right)e_{i},(3\leq i\leq n),\\
\displaystyle[e_{n+1},e_1]=-e_1+(2-b)e_3-ce_5, [e_{n+1},e_3]=(1-b)e_3,[e_{n+1},e_4]=e_2-be_4,\\
\displaystyle [e_{n+1},e_j]=\left(4-j-b\right)e_j,(5\leq j\leq n,n\geq4),\\
\displaystyle (b\neq2,n=4\,\,or\,\,b\neq2,b\neq4-n,n\geq5).
\end{array} 
\right.
\end{equation} 
Then we have the following two cases:

(I) If $b\neq0,(n\geq5),$ then we apply the transformation
$e^{\prime}_1=e_1-\frac{c}{b}e_5,e^{\prime}_i=e_i,(2\leq i\leq n+1)$
to remove $c,-c$
from the $(5,1)^{st}$ positions in $\r_{e_{n+1}}$ and $\L_{e_{n+1}},$ respectively. We have the following continuous family of Leibniz
algebras:
\allowdisplaybreaks
\begin{equation}
\begin{array}{l}
\displaystyle  [e_1,e_{n+1}]=e_1+(b-2)e_3,[e_2,e_{n+1}]=be_2,
[e_{i},e_{n+1}]=\left(b+i-4\right)e_{i},\\
\displaystyle (3\leq i\leq n),[e_{n+1},e_1]=-e_1-(b-2)e_3,[e_{n+1},e_3]=(1-b)e_3,\\
\displaystyle [e_{n+1},e_4]=e_2-be_4, [e_{n+1},e_j]=\left(4-j-b\right)e_j,(5\leq j\leq n,n\geq4),\\
\displaystyle(b\neq0, b\neq2,n=4\,\,or\,\,b\neq0,b\neq2,b\neq4-n,n\geq5). \end{array} 
\label{g(L2(c=0))}
\end{equation} 
(II) Suppose $b=0.$ If $c=0,$ then we have a limiting case of (\ref{g(L2(c=0))}) with $b=0,(n\geq4).$
If $c\neq0$ in particular greater than zero or less than zero, respectively,
 then we apply the transformation
$e^{\prime}_1=\sqrt{\pm c}e_1,e^{\prime}_2=\pm ce_2,e^{\prime}_i=\left(\pm c\right)^{\frac{i-2}{2}}e_i,(3\leq i\leq n, n\geq5),e^{\prime}_{n+1}=e_{n+1}$
to scale $c$ to $\pm1$ and we obtain a Leibniz algebra $\g_{n+1,2}$ given below:
\begin{equation}
\begin{array}{l}
\displaystyle \nonumber [e_1,e_{n+1}]=e_1-2e_3+\delta e_5,
[e_{i},e_{n+1}]=(i-4)e_{i},(3\leq i\leq n),[e_{n+1},e_1]=-e_1+2e_3-\delta e_5,\\
\displaystyle[e_{n+1},e_3]=e_3,[e_{n+1},e_4]=e_2,[e_{n+1},e_j]=(4-j)e_j, (\delta=\pm1,5\leq j\leq n,n\geq5). 
\end{array} 
\end{equation}  
\end{itemize}
\item[(2)] Suppose $b:=(4-n)a,a\neq0,(n\geq5).$ We have the right (a derivation) and the left (not a derivation) multiplication operators restricted to the nilradical: 
$$\r_{e_{n+1}}=\left[\begin{smallmatrix}
 a & 0 & 0 & 0&0&0&\cdots && 0&0 & 0\\
  a_{2,1}+a_{4,1}-A_{4,3} & (4-n)a & a_{2,3}& 0 &0&0 & \cdots &&0  & 0& 0\\
(2-n)a & 0 & (3-n)a & 0 & 0&0 &\cdots &&0 &0& 0\\
  0 & 0 &  0 & (4-n)a &0 &0 &\cdots&&0 &0 & 0\\
 0 & 0 & a_{5,3} & 0 & (5-n)a  &0&\cdots & &0&0 & 0\\
  0 & 0 &\boldsymbol{\cdot} & a_{5,3} & 0 &(6-n)a&\cdots & &0&0 & 0\\
    0 & 0 &\boldsymbol{\cdot} & \boldsymbol{\cdot} & \ddots &0&\ddots &&\vdots&\vdots &\vdots\\
  \vdots & \vdots & \vdots &\vdots &  &\ddots&\ddots &\ddots &\vdots&\vdots & \vdots\\
 0 & 0 & a_{n-2,3}& a_{n-3,3}& \cdots&\cdots &a_{5,3}&0&-2a &0& 0\\
0 & 0 & a_{n-1,3}& a_{n-2,3}& \cdots&\cdots &\boldsymbol{\cdot}&a_{5,3}&0 &-a& 0\\
 0 & 0 & a_{n,3}& a_{n-1,3}& \cdots&\cdots &\boldsymbol{\cdot}&\boldsymbol{\cdot}&a_{5,3} &0&0
\end{smallmatrix}\right],$$
 $$\L_{e_{n+1}}=\left[\begin{smallmatrix}
 -a & 0 & 0 & 0&0&0&\cdots && 0&0 & 0\\
  b_{2,1}-A_{4,3} & 0 & (3-n)a_{2,3}& a &0&0 & \cdots &&0  & 0& 0\\
 (n-2)a & 0 & (n-3)a & 0 & 0&0 &\cdots &&0 &0& 0\\
  0 & 0 &  0 & (n-4)a &0 &0 &\cdots&&0 &0 & 0\\
 0 & 0 & -a_{5,3} & 0 & (n-5)a  &0&\cdots & &0&0 & 0\\
  0 & 0 &\boldsymbol{\cdot} & -a_{5,3} & 0 &(n-6)a&\cdots & &0&0 & 0\\
    0 & 0 &\boldsymbol{\cdot} & \boldsymbol{\cdot} & \ddots &0&\ddots &&\vdots&\vdots &\vdots\\
  \vdots & \vdots & \vdots &\vdots &  &\ddots&\ddots &\ddots &\vdots&\vdots & \vdots\\
 0 & 0 & -a_{n-2,3}& -a_{n-3,3}& \cdots&\cdots &-a_{5,3}&0&2a &0& 0\\
0 & 0 & -a_{n-1,3}& -a_{n-2,3}& \cdots&\cdots &\boldsymbol{\cdot}&-a_{5,3}&0 &a& 0\\
 0 & 0 & -a_{n,3}& -a_{n-1,3}& \cdots&\cdots &\boldsymbol{\cdot}&\boldsymbol{\cdot}&-a_{5,3} &0& 0
\end{smallmatrix}\right].$$
\begin{itemize}[noitemsep, topsep=0pt]
\allowdisplaybreaks 
\item We apply the transformation $e^{\prime}_1=e_1,e^{\prime}_2=e_2,e^{\prime}_i=e_i-\frac{a_{k-i+3,3}}{(k-i)a}e_k,(3\leq i\leq n-2,
i+2\leq k\leq n),e^{\prime}_{j}=e_{j},(n-1\leq j\leq n+1),$ where $k$ is fixed, renaming all the affected entries the way they were.
This transformation removes $a_{5,3},a_{6,3},...,a_{n,3}$ in $\r_{e_{n+1}}$ and $-a_{5,3},-a_{6,3},...,-a_{n,3}$ in $\L_{e_{n+1}}.$
Besides it introduces the entries in the $(5,1)^{st},(6,1)^{st},...,(n,1)^{st}$ positions in $\r_{e_{n+1}}$ and $\L_{e_{n+1}},$ 
which we call by $a_{5,1},a_{6,1},...,a_{n,1}$ and $-a_{5,1},-a_{6,1},...,-a_{n,1},$ respectively.
\item The transformation $e^{\prime}_1=e_1-\frac{1}{a}\left(b_{2,1}-A_{4,3}+(n-2)a_{2,3}\right)e_2,e^{\prime}_2=e_2,
e^{\prime}_3=e_3-\frac{a_{2,3}}{a}e_2,e^{\prime}_{i}=e_{i},
(4\leq i\leq n),e^{\prime}_{n+1}=e_{n+1}+\sum_{k=5}^{n-1}a_{k+1,1}e_{k}$
removes $a_{2,1}+a_{4,1}-A_{4,3}$ and $b_{2,1}-A_{4,3}$ from the $(2,1)^{st}$ positions in $\r_{e_{n+1}}$ and $\L_{e_{n+1}},$
respectively. It also removes $a_{2,3}$ and $(3-n)a_{2,3}$
from the entries in the $(2,3)^{rd}$ positions in $\r_{e_{n+1}}$ and $\L_{e_{n+1}};$ 
$a_{k+1,1}$ in $\r_{e_{n+1}}$ and $-a_{k+1,1}$ in $\L_{e_{n+1}}$ from the entries in the $(k+1,1)^{st}$
positions, where $(5\leq k\leq n-1)$ keeping other entries unchanged.
\item Applying the transformation $e^{\prime}_i=e_i,(1\leq i\leq n),e^{\prime}_{n+1}=e_{n+1}+\frac{a_{5,1}}{n-4}e_2+a_{5,1}e_4,$
we remove $a_{5,1}$ and $-a_{5,1}$ from the $(5,1)^{st}$ positions in $\r_{e_{n+1}}$ and
$\L_{e_{n+1}}$, respectively. Therefore we have a continuous family
of Leibniz algebras:
\begin{equation}
\begin{array}{l}
\displaystyle   \nonumber[e_1,e_{n+1}]=ae_1+(2-n)ae_3,[e_2,e_{n+1}]=(4-n)ae_2,
[e_{i},e_{n+1}]=(i-n)ae_{i},(3\leq i\leq n-1),\\
\displaystyle [e_{n+1},e_{n+1}]=a_{n,n+1}e_n,[e_{n+1},e_1]=-ae_1+(n-2)ae_3,[e_{n+1},e_3]=(n-3)ae_3,\\
\displaystyle [e_{n+1},e_4]=ae_2+(n-4)ae_4, [e_{n+1},e_j]=(n-j)ae_j,(5\leq j\leq n-1). \end{array} 
\end{equation} 
\item To scale $a$ to unity, we apply the transformation $e^{\prime}_i=e_i,(1\leq i\leq n),e^{\prime}_{n+1}=\frac{e_{n+1}}{a}$ 
renaming the coefficient in front of $e_n$ in $[e_{n+1},e_{n+1}]$ back by $a_{n,n+1}.$ We obtain a Leibniz algebra
\allowdisplaybreaks
\begin{equation}
\begin{array}{l}
\displaystyle   \nonumber[e_1,e_{n+1}]=e_1+(2-n)e_3,[e_2,e_{n+1}]=(4-n)e_2,
[e_{i},e_{n+1}]=(i-n)e_{i},(3\leq i\leq n-1),\\
\displaystyle [e_{n+1},e_{n+1}]=a_{n,n+1}e_n,[e_{n+1},e_1]=-e_1+(n-2)e_3,[e_{n+1},e_3]=(n-3)e_3,\\
\displaystyle [e_{n+1},e_4]=e_2+(n-4)e_4, [e_{n+1},e_j]=(n-j)e_j,(5\leq j\leq n-1). \end{array} 
\end{equation} 
 \end{itemize}
If $a_{n,n+1}=0,$ then we have a limiting case of (\ref{g(L2(c=0))}) with $b=4-n,(n\geq5)$. If $a_{n,n+1}\neq0$ in particular greater than zero or less than zero, respectively,
 then we apply the transformation
$e^{\prime}_i=\left(\pm a_{n,n+1}\right)^{\frac{i}{n-2}}e_i,(1\leq i\leq 2),e^{\prime}_k=\left(\pm a_{n,n+1}\right)^{\frac{k-2}{n-2}}e_k,(3\leq k\leq n),
e^{\prime}_{n+1}=e_{n+1}$
to scale $a_{n,n+1}$ to $\pm1$. We have the algebra $\g_{n+1,3}$ given below:
\begin{equation}
\begin{array}{l}
\displaystyle   \nonumber[e_1,e_{n+1}]=e_1+(2-n)e_3,[e_2,e_{n+1}]=(4-n)e_2,
[e_{i},e_{n+1}]=(i-n)e_{i},(3\leq i\leq n-1),\\
\displaystyle [e_{n+1},e_{n+1}]=\delta e_n,[e_{n+1},e_1]=-e_1+(n-2)e_3,[e_{n+1},e_3]=(n-3)e_3,\\
\displaystyle [e_{n+1},e_4]=e_2+(n-4)e_4, [e_{n+1},e_j]=(n-j)e_j,(\delta=\pm1,5\leq j\leq n-1, n\geq5). \end{array} 
\label{g_{n+1,3}}
\end{equation} 
\item[(3)] Suppose $a=0$ and $b\neq0,(n\geq4).$ We have the right (a derivation) and the left (not a derivation) multiplication operators restricted to the nilradical: 
$$\r_{e_{n+1}}=\left[\begin{smallmatrix}
 0 & 0 & 0 & 0&0&0&\cdots && 0&0 & 0\\
  a_{2,1} & b & 0& 0 &0&0 & \cdots &&0  & 0& 0\\
  b & 0 & b & 0 & 0&0 &\cdots &&0 &0& 0\\
  0 & 0 &  0 & b &0 &0 &\cdots&&0 &0 & 0\\
 0 & 0 & a_{5,3} & 0 & b  &0&\cdots & &0&0 & 0\\
  0 & 0 &\boldsymbol{\cdot} & a_{5,3} & 0 &b&\cdots & &0&0 & 0\\
    0 & 0 &\boldsymbol{\cdot} & \boldsymbol{\cdot} & \ddots &0&\ddots &&\vdots&\vdots &\vdots\\
  \vdots & \vdots & \vdots &\vdots &  &\ddots&\ddots &\ddots &\vdots&\vdots & \vdots\\
 0 & 0 & a_{n-2,3}& a_{n-3,3}& \cdots&\cdots &a_{5,3}&0&b &0& 0\\
0 & 0 & a_{n-1,3}& a_{n-2,3}& \cdots&\cdots &\boldsymbol{\cdot}&a_{5,3}&0 &b& 0\\
 0 & 0 & a_{n,3}& a_{n-1,3}& \cdots&\cdots &\boldsymbol{\cdot}&\boldsymbol{\cdot}&a_{5,3} &0& b
\end{smallmatrix}\right],$$
 $$\L_{e_{n+1}}=\left[\begin{smallmatrix}
 0 & 0 & 0 & 0&0&0&\cdots && 0&0 & 0\\
  b_{2,1} & 0 & b_{2,1}& 0 &0&0 & \cdots &&0  & 0& 0\\
  -b & 0 & -b & 0 & 0&0 &\cdots &&0 &0& 0\\
  0 & 0 &  0 & -b &0 &0 &\cdots&&0 &0 & 0\\
 0 & 0 & -a_{5,3} & 0 & -b  &0&\cdots & &0&0 & 0\\
  0 & 0 &\boldsymbol{\cdot} & -a_{5,3} & 0 &-b&\cdots & &0&0 & 0\\
    0 & 0 &\boldsymbol{\cdot} & \boldsymbol{\cdot} & \ddots &0&\ddots &&\vdots&\vdots &\vdots\\
  \vdots & \vdots & \vdots &\vdots &  &\ddots&\ddots &\ddots &\vdots&\vdots & \vdots\\
 0 & 0 & -a_{n-2,3}& -a_{n-3,3}& \cdots&\cdots &-a_{5,3}&0&-b &0& 0\\
0 & 0 & -a_{n-1,3}& -a_{n-2,3}& \cdots&\cdots &\boldsymbol{\cdot}&-a_{5,3}&0 &-b& 0\\
 0 & 0 & -a_{n,3}& -a_{n-1,3}& \cdots&\cdots &\boldsymbol{\cdot}&\boldsymbol{\cdot}&-a_{5,3} &0&-b
\end{smallmatrix}\right].$$
\begin{itemize}[noitemsep, topsep=0pt]
\allowdisplaybreaks 
\item Applying the transformation $e^{\prime}_1=e_1-\frac{a_{2,1}+b_{2,1}}{b}e_2,e^{\prime}_2=e_2,
e^{\prime}_3=e_3-\frac{b_{2,1}}{b}e_2,e^{\prime}_{i}=e_{i},
(4\leq i\leq n+1),$
we
remove $a_{2,1}$ and $b_{2,1}$ from the $(2,1)^{st}$ positions in $\r_{e_{n+1}}$ and $\L_{e_{n+1}},$
respectively. It also removes $b_{2,1}$
from the entry in the $(2,3)^{rd}$ position in $\L_{e_{n+1}}$ keeping other entries unchanged.
 Therefore we have a continuous family
of Leibniz algebras given below:
\begin{equation}
\begin{array}{l}
\displaystyle   \nonumber[e_1,e_{n+1}]=be_3,[e_2,e_{n+1}]=be_2,
[e_{i},e_{n+1}]=be_{i}+\sum_{k=i+2}^n{a_{k-i+3,3}e_k},\\
\displaystyle [e_{n+1},e_1]=-be_3,[e_{n+1},e_i]=-be_i-\sum_{k=i+2}^n{a_{k-i+3,3}e_k},(3\leq i\leq n,n\geq4).\end{array} 
\end{equation} 
\item To scale $b$ to unity, we apply the transformation $e^{\prime}_i=e_i,(1\leq i\leq n),e^{\prime}_{n+1}=\frac{e_{n+1}}{b}.$ 
Then we rename $\frac{a_{5,3}}{b},\frac{a_{6,3}}{b},...,\frac{a_{n,3}}{b}$ by $a_{5,3},a_{6,3},...,a_{n,3},$ respectively.
We obtain a Leibniz algebra:
\begin{equation}
\begin{array}{l}
\displaystyle   \nonumber[e_1,e_{n+1}]=e_3,[e_2,e_{n+1}]=e_2,
[e_{i},e_{n+1}]=e_{i}+\sum_{k=i+2}^n{a_{k-i+3,3}e_k},\\
\displaystyle [e_{n+1},e_1]=-e_3,[e_{n+1},e_i]=-e_i-\sum_{k=i+2}^n{a_{k-i+3,3}e_k},(3\leq i\leq n,n\geq4).\end{array} 
\end{equation}
If $a_{5,3}\neq0,(n\geq5),$ precisely, greater than zero and less than zero, respectively, then applying
the transformation $e^{\prime}_i=\left(\pm a_{5,3}\right)^{\frac{i}{2}}e_i,(1\leq i\leq 2),e^{\prime}_k=\left(\pm a_{5,3}\right)^{\frac{k-2}{2}}e_k,
(3\leq k\leq n),e^{\prime}_{n+1}=e_{n+1},$ we scale it to $\pm1.$ We also rename all the affected entries back
and then we rename $a_{6,3},...,a_{n,3}$ by $b_1,...,b_{n-5},$ respectively.
We combine with the case, when $a_{5,3}=0$ and obtain a Leibniz algebra $\g_{n+1,4}$ given below:
\begin{equation}
\begin{array}{l}
\displaystyle   \nonumber[e_1,e_{n+1}]=e_3,[e_2,e_{n+1}]=e_2,
[e_{i},e_{n+1}]=e_{i}+\epsilon e_{i+2}+\sum_{k=i+3}^n{b_{k-i-2}e_k},[e_{n+1},e_1]=-e_3,\\
\displaystyle [e_{n+1},e_i]=-e_i-\epsilon e_{i+2}-\sum_{k=i+3}^n{b_{k-i-2}e_k},(\epsilon=0,\pm1,3\leq i\leq n,n\geq4).\end{array} 
\end{equation}
\end{itemize}
\item[(4)] Suppose $b:=2a,a\neq0,(n\geq4).$ We have the right (a derivation) and the left (not a derivation) multiplication operators restricted to the nilradical: 
$$\r_{e_{n+1}}=\left[\begin{smallmatrix}
 a & 0 & 0 & 0&0&0&\cdots && 0&0 & 0\\
  a_{2,1} & 2a & a_{2,3}& 0 &0&0 & \cdots &&0  & 0& 0\\
  0 & 0 & a & 0 & 0&0 &\cdots &&0 &0& 0\\
  0 & 0 &  0 & 2a &0 &0 &\cdots&&0 &0 & 0\\
 0 & 0 & a_{5,3} & 0 & 3a  &0&\cdots & &0&0 & 0\\
  0 & 0 &\boldsymbol{\cdot} & a_{5,3} & 0 &4a&\cdots & &0&0 & 0\\
    0 & 0 &\boldsymbol{\cdot} & \boldsymbol{\cdot} & \ddots &0&\ddots &&\vdots&\vdots &\vdots\\
  \vdots & \vdots & \vdots &\vdots &  &\ddots&\ddots &\ddots &\vdots&\vdots & \vdots\\
 0 & 0 & a_{n-2,3}& a_{n-3,3}& \cdots&\cdots &a_{5,3}&0&(n-4)a &0& 0\\
0 & 0 & a_{n-1,3}& a_{n-2,3}& \cdots&\cdots &\boldsymbol{\cdot}&a_{5,3}&0 &(n-3)a& 0\\
 0 & 0 & a_{n,3}& a_{n-1,3}& \cdots&\cdots &\boldsymbol{\cdot}&\boldsymbol{\cdot}&a_{5,3} &0& (n-2)a
\end{smallmatrix}\right],$$
 $$\L_{e_{n+1}}=\left[\begin{smallmatrix}
 -a & 0 & 0 & 0&0&0&\cdots && 0&0 & 0\\
  a_{2,1} & 0 & a_{2,3}& a &0&0 & \cdots &&0  & 0& 0\\
  0 & 0 & -a & 0 & 0&0 &\cdots &&0 &0& 0\\
  0 & 0 &  0 & -2a &0 &0 &\cdots&&0 &0 & 0\\
 0 & 0 & -a_{5,3} & 0 & -3a  &0&\cdots & &0&0 & 0\\
  0 & 0 &\boldsymbol{\cdot} & -a_{5,3} & 0 &-4a&\cdots & &0&0 & 0\\
    0 & 0 &\boldsymbol{\cdot} & \boldsymbol{\cdot} & \ddots &0&\ddots &&\vdots&\vdots &\vdots\\
  \vdots & \vdots & \vdots &\vdots &  &\ddots&\ddots &\ddots &\vdots&\vdots & \vdots\\
 0 & 0 & -a_{n-2,3}& -a_{n-3,3}& \cdots&\cdots &-a_{5,3}&0&(4-n)a &0& 0\\
0 & 0 & -a_{n-1,3}& -a_{n-2,3}& \cdots&\cdots &\boldsymbol{\cdot}&-a_{5,3}&0 &(3-n)a& 0\\
 0 & 0 & -a_{n,3}& -a_{n-1,3}& \cdots&\cdots &\boldsymbol{\cdot}&\boldsymbol{\cdot}&-a_{5,3} &0& (2-n)a
\end{smallmatrix}\right].$$
\begin{itemize}[noitemsep, topsep=0pt]
\allowdisplaybreaks 
\item We apply the transformation $e^{\prime}_1=e_1,e^{\prime}_2=e_2,e^{\prime}_i=e_i-\frac{a_{k-i+3,3}}{(k-i)a}e_k,(3\leq i\leq n-2,
i+2\leq k\leq n,n\geq5),e^{\prime}_{j}=e_{j},(n-1\leq j\leq n+1),$ where $k$ is fixed renaming all the affected entries back.
This transformation removes $a_{5,3},a_{6,3},...,a_{n,3}$ and $-a_{5,3},-a_{6,3},...,-a_{n,3}$ in $\r_{e_{n+1}}$ and $\L_{e_{n+1}},$
respectively.
It does not affect any other entries.
\item Applying the transformation $e^{\prime}_1=e_1-\frac{a_{2,1}}{a}e_2,e^{\prime}_2=e_2,
e^{\prime}_3=e_3-\frac{a_{2,3}}{a}e_2,e^{\prime}_{i}=e_{i},
(4\leq i\leq n+1),$
we
remove $a_{2,1}$ from the $(2,1)^{st}$ positions in $\r_{e_{n+1}}$ and $\L_{e_{n+1}},$
$a_{2,3}$ 
from the entries in the $(2,3)^{rd}$ positions in $\r_{e_{n+1}}$ and $\L_{e_{n+1}}$ keeping other entries unchanged.
\item The transformation $e^{\prime}_i=e_i,(1\leq i\leq n),e^{\prime}_{n+1}=\frac{e_{n+1}}{a}$ scales $a$ to unity.
We obtain the following algebra:
\begin{equation}
\begin{array}{l}
\displaystyle  \nonumber [e_1,e_{n+1}]=e_1,[e_2,e_{n+1}]=2e_2,[e_{i},e_{n+1}]=\left(i-2\right)e_{i},(3\leq i\leq n),[e_{n+1},e_1]=-e_1,\\
\displaystyle [e_{n+1},e_3]=-e_3,[e_{n+1},e_4]=e_2-2e_4,[e_{n+1},e_j]=\left(2-j\right)e_j,(5\leq j\leq n),
\end{array} 
\end{equation} 
which is a limiting case of (\ref{g(L2(c=0))}) with $b=2,(n\geq4).$ Altogether (\ref{g(L2(c=0))}) and all its limiting
cases after replacing $b$ with $a$ give us a Leibniz algebra $\g_{n+1,1}$ given below:
\begin{equation}
\begin{array}{l}
\displaystyle \nonumber [e_1,e_{n+1}]=e_1+(a-2)e_3,[e_2,e_{n+1}]=ae_2,
[e_{i},e_{n+1}]=\left(a+i-4\right)e_{i},\\
\displaystyle (3\leq i\leq n),[e_{n+1},e_1]=-e_1-(a-2)e_3,[e_{n+1},e_3]=(1-a)e_3,\\
\displaystyle [e_{n+1},e_4]=e_2-ae_4, [e_{n+1},e_j]=\left(4-j-a\right)e_j,(5\leq j\leq n,n\geq4). \end{array} 
\end{equation} 
\end{itemize}
 \end{enumerate}
\end{proof}

\subsubsection{Two dimensional right solvable extensions of $\mathcal{L}^2$}\label{Section5.1.2}
The non-zero inner derivations of $\mathcal{L}^2,(n\geq4)$ are
given by
 \[
\r_{e_1}=\left[\begin{smallmatrix}
0&0 & 0 & 0 & \cdots & 0 & 0  & 0 \\
1&0 & 0 & 0 & \cdots & 0 & 0  & 0 \\
0&0 & 0 & 0 & \cdots & 0 & 0 & 0 \\
 0&0 & 1 & 0 & \cdots & 0 & 0 & 0 \\
0& 0 & 0 & 1 & \cdots & 0 & 0 & 0 \\
\vdots& \vdots  & \vdots  & \vdots  & \ddots & \vdots & \vdots & \vdots\\
 0& 0 & 0 & 0&\cdots & 1 & 0 &0\\
  0& 0 & 0 & 0&\cdots & 0 & 1 &0
\end{smallmatrix}\right],\r_{e_3}=E_{2,1}-E_{4,1},
\r_{e_i}=-E_{i+1,1}=\left[\begin{smallmatrix} 0 & 0&0&\cdots &  0 \\
 0& 0&0&\cdots &  0 \\
  0 & 0&0&\cdots &  0 \\
    0 & 0&0&\cdots &  0 \\
 \vdots &\vdots &\vdots& & \vdots\\
 -1 &0&0& \cdots & 0\\
  \vdots &\vdots &\vdots& & \vdots\\
  \boldsymbol{\cdot} & 0&0&\cdots &  0
 \end{smallmatrix}\right]\,(4\leq i\leq n-1),\] where $E_{2,1}, E_{4,1}$ and $E_{i+1,1}$ are $n\times n$ matrices that has $1$ in the 
 $(2,1)^{st},(4,1)^{st}$ and $(i+1,1)^{st}$
positions, respectively, and all other entries are zero.

One could notice that two outer derivations $\r_{e_{n+1}}$ and $\r_{e_{n+2}}$ are ``nil-independent'' (see Section \ref{Pr}.) only in case $(1)$
in Theorem \ref{TheoremRL2Absorption}.
 \begin{remark} \label{TwoOuterDerivations}
If we have three or more outer derivations, then they are
``nil-dependent''(see Section \ref{Pr}.).
Therefore the solvable algebras we are constructing are of codimension at most two.
\end{remark}

By taking a linear combination of $\r_{e_{n+1}}$ and $\r_{e_{n+2}}$ and keeping in mind that no nontrivial linear combination of the matrices $\r_{e_{n+1}}$ and $\r_{e_{n+2}}$ 
can be a nilpotent matrix, one could set $\left(
\begin{array}{c}
  a \\
 b
\end{array}\right)=\left(
\begin{array}{c}
  1\\
 0
\end{array}\right)$ and $\left(
\begin{array}{c}
 \alpha \\
 \beta
\end{array}\right)=\left(
\begin{array}{c}
  1\\
 1
\end{array}\right).$
Therefore the vector space of outer derivations as the $n \times n$ matrices is as follows:
{ $$\r_{e_{n+1}}=\left[\begin{smallmatrix}
 1 & 0 & 0 & 0&0&0&\cdots && 0&0 & 0\\
  \frac{a_{2,1}+a_{4,1}-b_{2,1}}{2} & 0 & a_{2,3}& 0 &0&0 & \cdots &&0  & 0& 0\\
  -2 & 0 & -1 & 0 & 0&0 &\cdots &&0 &0& 0\\
  0 & 0 &  0 & 0 &0 &0 &\cdots&&0 &0 & 0\\
 0 & 0 & a_{5,3} & 0 & 1  &0&\cdots & &0&0 & 0\\
  0 & 0 &\boldsymbol{\cdot} & a_{5,3} & 0 &2&\cdots & &0&0 & 0\\
    0 & 0 &\boldsymbol{\cdot} & \boldsymbol{\cdot} & \ddots &0&\ddots &&\vdots&\vdots &\vdots\\
  \vdots & \vdots & \vdots &\vdots &  &\ddots&\ddots &\ddots &\vdots&\vdots & \vdots\\
 0 & 0 & a_{n-2,3}& a_{n-3,3}& \cdots&\cdots &a_{5,3}&0&n-6 &0& 0\\
0 & 0 & a_{n-1,3}& a_{n-2,3}& \cdots&\cdots &\boldsymbol{\cdot}&a_{5,3}&0 &n-5& 0\\
 0 & 0 & a_{n,3}& a_{n-1,3}& \cdots&\cdots &\boldsymbol{\cdot}&\boldsymbol{\cdot}&a_{5,3} &0& n-4
\end{smallmatrix}\right],$$}
{ $$ \r_{e_{n+2}}=\left[\begin{smallmatrix}
 1 & 0 & 0 & 0&0&0&\cdots && 0&0 & 0\\
  \alpha_{2,3} & 1 & \alpha_{2,3}& 0 &0&0 & \cdots &&0  & 0& 0\\
  -1 & 0 & 0 & 0 & 0&0 &\cdots &&0 &0& 0\\
  0 & 0 &  0 & 1 &0 &0 &\cdots&&0 &0 & 0\\
 0 & 0 & \alpha_{5,3} & 0 & 2  &0&\cdots & &0&0 & 0\\
  0 & 0 &\boldsymbol{\cdot} & \alpha_{5,3} & 0 &3&\cdots & &0&0 & 0\\
    0 & 0 &\boldsymbol{\cdot} & \boldsymbol{\cdot} & \ddots &0&\ddots &&\vdots&\vdots &\vdots\\
  \vdots & \vdots & \vdots &\vdots &  &\ddots&\ddots &\ddots &\vdots&\vdots & \vdots\\
 0 & 0 & \alpha_{n-2,3}& \alpha_{n-3,3}& \cdots&\cdots &\alpha_{5,3}&0&n-5 &0& 0\\
0 & 0 & \alpha_{n-1,3}& \alpha_{n-2,3}& \cdots&\cdots &\boldsymbol{\cdot}&\alpha_{5,3}&0 &n-4& 0\\
 0 & 0 & \alpha_{n,3}& \alpha_{n-1,3}& \cdots&\cdots &\boldsymbol{\cdot}&\boldsymbol{\cdot}&\alpha_{5,3} &0& n-3
\end{smallmatrix}\right].$$}
\begin{center}\textit{General approach to find right solvable Leibniz
algebras with a codimension two nilradical $\mathcal{L}^2$.}\footnote{When we work with the left Leibniz algebras, we first change the right multiplication operator to the left  everywhere and the right Leibniz identity to the left Leibniz identity in step $(iii).$ We also interchange
$s$ and  $r$ on the very left in steps $(i)$ and $(ii)$ as well. }
\end{center}
\begin{enumerate}[noitemsep, topsep=0pt]
\item[(i)] We consider $\r_{[e_r,e_s]}=[\r_{e_s},\r_{e_r}]=\r_{e_s}\r_{e_r}-\r_{e_r}\r_{e_s},\,(n+1\leq r\leq n+2,\,1\leq s\leq n+2)$
and compare with $c_1\r_{e_1}+\sum_{k=3}^{n-1}c_k\r_{e_k},$ because $e_{2}$
and $e_n$ are in the center of $\mathcal{L}^2,\,(n\geq4)$ defined in $(\ref{L2})$ to find all the unknown commutators.
\item[(ii)] We write down
$[e_{r},e_{s}],\,(n+1\leq r\leq n+2,\,1\leq s\leq n+2)$ including a linear combination of
$e_{2}$ and $e_n$ as well. We add the brackets of the nilradical $\mathcal{L}^2$ and outer derivations $\r_{e_{n+1}}$ and $\r_{e_{n+2}}.$
\item[(iii)] One satisfies the right Leibniz identity: $[[e_r,e_s],e_t]=[[e_r,e_t],e_s]+[e_r,[e_s,e_t]]$ 
or, equivalently, $\r_{e_t}\left([e_r,e_s]\right)=[\r_{e_t}(e_r),e_s]+[e_r,\r_{e_t}(e_s)],\,(1\leq r,s,t\leq n+2)$ for all the brackets obtained in step $(ii).$
\item[(iv)] Then we carry out the technique of ``absorption'' (see Section \ref{Solvable left Leibniz algebras}) to remove as many parameters as possible to simplify the algebra. 
\item[(v)]  We apply the change of basis transformations without affecting the nilradical to remove as many parameters as possible.\end{enumerate}
\noindent $(i)$ Considering $\r_{[e_{n+1},e_{n+2}]},$ which is the same as $[\r_{e_{n+2}},\r_{e_{n+1}}],$ we deduce that $\alpha_{2,3}:=a_{2,3},
\alpha_{i,3}:=a_{i,3},(5\leq i\leq n).$ As a result, the outer derivation $\r_{e_{n+2}}$ becomes:
$$ \r_{e_{n+2}}=\left[\begin{smallmatrix}
 1 & 0 & 0 & 0&0&0&\cdots && 0&0 & 0\\
 a_{2,3} & 1 & a_{2,3}& 0 &0&0 & \cdots &&0  & 0& 0\\
  -1 & 0 & 0 & 0 & 0&0 &\cdots &&0 &0& 0\\
  0 & 0 &  0 & 1 &0 &0 &\cdots&&0 &0 & 0\\
 0 & 0 & a_{5,3} & 0 & 2  &0&\cdots & &0&0 & 0\\
  0 & 0 &\boldsymbol{\cdot} & a_{5,3} & 0 &3&\cdots & &0&0 & 0\\
    0 & 0 &\boldsymbol{\cdot} & \boldsymbol{\cdot} & \ddots &0&\ddots &&\vdots&\vdots &\vdots\\
  \vdots & \vdots & \vdots &\vdots &  &\ddots&\ddots &\ddots &\vdots&\vdots & \vdots\\
 0 & 0 & a_{n-2,3}& a_{n-3,3}& \cdots&\cdots &a_{5,3}&0&n-5 &0& 0\\
0 & 0 & a_{n-1,3}& a_{n-2,3}& \cdots&\cdots &\boldsymbol{\cdot}&a_{5,3}&0 &n-4& 0\\
 0 & 0 & a_{n,3}& a_{n-1,3}& \cdots&\cdots &\boldsymbol{\cdot}&\boldsymbol{\cdot}&a_{5,3} &0& n-3
\end{smallmatrix}\right].$$
We also find the following commutators:
\allowdisplaybreaks
\begin{equation}
\left\{
\begin{array}{l}
\displaystyle  \nonumber \r_{[e_{n+1},e_{1}]}=-\r_{e_1}+2\r_{e_3},\r_{[e_{n+1},e_2]}=0,\r_{[e_{n+1},e_i]}=(4-i)\r_{e_i}-\sum_{k=i+2}^{n-1}{a_{k-i+3,3}\r_{e_k}},\\
\displaystyle \r_{[e_{n+1},e_j]}=0,(n\leq j\leq n+1),\r_{[e_{n+1},e_{n+2}]}=\sum_{k=4}^{n-1}{a_{k+1,3}\r_{e_{k}}},\r_{[e_{n+2},e_1]}=-\r_{e_1}+
\r_{e_3},\\
\displaystyle \r_{[e_{n+2},e_2]}=0,\r_{[e_{n+2},e_{i}]}=(3-i)\r_{e_i}-\sum_{k=i+2}^{n-1}{a_{k-i+3,3}\r_{e_k}},(3\leq i\leq n-1),\\
\displaystyle \r_{[e_{n+2},e_n]}=0,\r_{[e_{n+2},e_{n+1}]}=-\sum_{k=4}^{n-1}{a_{k+1,3}\r_{e_{k}}},\r_{[e_{n+2},e_{n+2}]}=0.
\end{array} 
\right.
\end{equation} 
\noindent $(ii)$ We include a linear combination of $e_2$ and $e_n$:
\begin{equation}
\left\{
\begin{array}{l}
\displaystyle  \nonumber [e_{n+1},e_{1}]=-e_1+c_{2,1}e_2+2e_3+c_{n,1}e_n,[e_{n+1},e_2]=c_{2,2}e_2+c_{n,2}e_n,[e_{n+1},e_i]=c_{2,i}e_2\\
\displaystyle+(4-i)e_i-\sum_{k=i+2}^{n-1}{a_{k-i+3,3}e_k}+c_{n,i}e_n,[e_{n+1},e_j]=c_{2,j}e_2+c_{n,j}e_n,(n\leq j\leq n+1),\\
\displaystyle [e_{n+1},e_{n+2}]=c_{2,n+2}e_2+\sum_{k=4}^{n-1}{a_{k+1,3}e_{k}}+
c_{n,n+2}e_n,[e_{n+2},e_1]=-e_1+d_{2,1}e_2+e_3+d_{n,1}e_n,\\
\displaystyle [e_{n+2},e_2]=d_{2,2}e_2+d_{n,2}e_n, [e_{n+2},e_{i}]=d_{2,i}e_2+(3-i)e_i-\sum_{k=i+2}^{n-1}{a_{k-i+3,3}e_k}+d_{n,i}e_n,\\
\displaystyle (3\leq i\leq n-1),[e_{n+2},e_n]=d_{2,n}e_2+d_{n,n}e_n,[e_{n+2},e_{n+1}]=d_{2,n+1}e_2-\sum_{k=4}^{n-1}{a_{k+1,3}e_{k}}\\
\displaystyle +d_{n,n+1}e_n, [e_{n+2},e_{n+2}]=d_{2,n+2}e_2+d_{n,n+2}e_n.
\end{array} 
\right.
\end{equation} 
Besides we have the brackets from $\mathcal{L}^2$ and outer derivations $\r_{e_{n+1}}$ and $\r_{e_{n+2}}$ as well.

\noindent $(iii)$ We satisfy the right Leibniz identity. There are two cases: $(n=4)$ and $(n\geq5)$.
Case $(n\geq5)$ is shown in Table \ref{RightCodimTwo(L2)}.
If $(n=4)$, then we apply the following identities one after another: $1.-3.,6.,
\r_{e_5}[e_5,e_3]=[\r_{e_5}(e_5),e_3]+[e_5,\r_{e_5}(e_3)], \r_{e_5}[e_5,e_5]=[\r_{e_5}(e_5),e_5]+[e_5,\r_{e_5}(e_5)],
\r_{e_5}[e_5,e_1]=[\r_{e_5}(e_5),e_1]+[e_5,\r_{e_5}(e_1)],$ $12.$ (gives that $d_{2,3}=0$ as well, but instead that $d_{4,3}=0$),
$\r_{e_6}[e_6,e_6]=[\r_{e_6}(e_6),e_6]+[e_6,\r_{e_6}(e_6)],14.,
\r_{e_6}[e_5,e_5]=[\r_{e_6}(e_5),e_5]+[e_5,\r_{e_6}(e_5)],\r_{e_5}[e_6,e_6]=[\r_{e_5}(e_6),e_6]+[e_6,\r_{e_5}(e_6)].$
We finish with the identity $15.$, which gives us that $d_{2,5}:=c_{2,5}$. 
\begin{table}[htb]
\caption{Right Leibniz identities in the codimension two nilradical $\mathcal{L}^2,(n\geq5)$.}
\label{RightCodimTwo(L2)}
\begin{tabular}{|l|p{2.7cm}|p{12cm}|}
\hline
\scriptsize Steps &\scriptsize Ordered triple &\scriptsize
Result\\ \hline
\scriptsize $1.$ &\scriptsize $\r_{e_1}\left([e_{n+1},e_{1}]\right)$ &\scriptsize
$[e_{n+1},e_2]=0$
$\implies$ $c_{2,2}=c_{n,2}=0.$\\ \hline
\scriptsize $2.$ &\scriptsize $\r_{e_1}\left([e_{n+2},e_{1}]\right)$ &\scriptsize
$[e_{n+2},e_2]=0$
$\implies$ $d_{2,2}=d_{n,2}=0.$\\ \hline
\scriptsize $3.$ &\scriptsize $\r_{e_3}\left([e_{n+1},e_{1}]\right)$ &\scriptsize
$c_{2,4}:=1,c_{n,4}:=-a_{n-1,3},$ where $a_{4,3}=0$ 
$\implies$  $[e_{n+1},e_4]=e_2-\sum_{k=6}^{n}{a_{k-1,3}e_k}.$ \\ \hline
\scriptsize $4.$ &\scriptsize $\r_{e_i}\left([e_{n+1},e_{1}]\right)$ &\scriptsize
$c_{2,i+1}=0,c_{n,i+1}:=-a_{n-i+2,3},(4\leq i\leq n-2),$ where $a_{4,3}=0$ 
$\implies$  $[e_{n+1},e_j]=\left(4-j\right)e_j-\sum_{k=j+2}^{n}{a_{k-j+3,3}e_k},(5\leq j\leq n-1).$ \\ \hline
\scriptsize $5.$ &\scriptsize $\r_{e_{n-1}}\left([e_{n+1},e_{1}]\right)$ &\scriptsize
$c_{2,n}=0,$ $c_{n,n}:=4-n$
$\implies$  $[e_{n+1},e_{n}]=\left(4-n\right)e_{n}.$ Altogether with $4.,$
  $[e_{n+1},e_i]=\left(4-i\right)e_i-\sum_{k=i+2}^{n}{a_{k-i+3,3}e_k},(5\leq i\leq n).$  \\ \hline
   \scriptsize $6.$ &\scriptsize $\r_{e_3}\left([e_{n+2},e_{1}]\right)$ &\scriptsize
$d_{2,4}:=1,d_{n,4}:=-a_{n-1,3},$ where $a_{4,3}=0$ 
$\implies$  $[e_{n+2},e_4]=e_2-e_4-\sum_{k=6}^{n}{a_{k-1,3}e_k}.$   \\ \hline  
   \scriptsize $7.$ &\scriptsize $\r_{e_i}\left([e_{n+2},e_{1}]\right)$ &\scriptsize
$d_{2,i+1}=0,d_{n,i+1}:=-a_{n-i+2,3},(4\leq i\leq n-2),$ where $a_{4,3}=0$ 
$\implies$  $[e_{n+2},e_j]=\left(3-j\right)e_j-\sum_{k=j+2}^{n}{a_{k-j+3,3}e_k},(5\leq j\leq n-1).$   \\ \hline
\scriptsize $8.$ &\scriptsize $\r_{e_{n-1}}\left([e_{n+2},e_{1}]\right)$ &\scriptsize
 $d_{2,n}=0,$ $d_{n,n}:=3-n$ 
$\implies$  $[e_{n+2},e_{n}]=\left(3-n\right)e_{n}.$ Altogether with $7.,$
  $[e_{n+2},e_i]=\left(3-i\right)e_i-\sum_{k=i+2}^{n}{a_{k-i+3,3}e_k},(5\leq i\leq n).$   \\ \hline
  \scriptsize $9.$ &\scriptsize $\r_{e_{3}}\left([e_{n+1},e_{n+1}]\right)$ &\scriptsize
$c_{2,3}:=-a_{2,3},$ $c_{n,3}:=-a_{n,3}$ 
$\implies$  $[e_{n+1},e_{3}]=-a_{2,3}e_2+e_3-\sum_{k=5}^n{a_{k,3}e_k}.$   \\ \hline
\scriptsize $10.$ &\scriptsize $\r_{e_{n+2}}\left([e_{n+1},e_{n+1}]\right)$ &\scriptsize
$c_{2,n+1}:=a_{5,3},c_{n,n+1}=0$
$\implies$ $[e_{n+1},e_{n+1}]=a_{5,3}e_2.$ \\ \hline
\scriptsize $11.$ &\scriptsize $\r_{e_1}\left([e_{n+1},e_{n+1}]\right)$ &\scriptsize
 $c_{2,1}:=-\frac{1}{2}\left(a_{2,1}+a_{4,1}-b_{2,1}\right), c_{n,1}=0$
$\implies$  $[e_{n+1},e_1]=-e_1-\frac{1}{2}\left(a_{2,1}+a_{4,1}-b_{2,1}\right)e_2+2e_3.$\\ \hline
\scriptsize $12.$ &\scriptsize $\r_{e_{n+2}}\left([e_{n+2},e_{3}]\right)$ &\scriptsize
$d_{2,3}=0,$ $d_{n,3}:=-a_{n,3}$ 
$\implies$  $[e_{n+2},e_{3}]=-\sum_{k=5}^n{a_{k,3}e_k}.$   \\ \hline
\scriptsize $13.$ &\scriptsize $\r_{e_{n+1}}\left([e_{n+2},e_{n+2}]\right)$ &\scriptsize
$d_{2,n+1}:=a_{5,3},d_{n,n+2}=0$
$\implies$  $[e_{n+2},e_{n+1}]=a_{5,3}e_2-\sum_{k=4}^{n-1}{a_{k+1,3}e_k}+d_{n,n+1}e_n,[e_{n+2},e_{n+2}]=d_{2,n+2}e_2$. \\ \hline
\scriptsize $14.$ &\scriptsize $\r_{e_{n+1}}\left([e_{n+2},e_{1}]\right)$ &\scriptsize
$d_{2,1}:=a_{2,3}-\frac{a_{2,1}}{2}-\frac{a_{4,1}}{2}+\frac{b_{2,1}}{2},d_{n,1}=0,(n\geq6)$
$\implies$ $[e_{n+2},e_1]=-e_1+\left(a_{2,3}-\frac{a_{2,1}}{2}-\frac{a_{4,1}}{2}+\frac{b_{2,1}}{2}\right)e_2+e_3.$
(\textbf{Remark}: If $n=5,$ then we apply $\r_{e_7}\left([e_7,e_1]\right)$ to obtain that $d_{5,1}=0.$)\\ \hline
\scriptsize $15.$ &\scriptsize $\r_{e_{n+2}}\left([e_{n+2},e_{n+1}]\right)$ &\scriptsize
$d_{n,n+1}:=-c_{n,n+2}$
$\implies$ $[e_{n+2},e_{n+1}]=a_{5,3}e_2-\sum_{k=4}^{n-1}{a_{k+1,3}e_{k}}-
c_{n,n+2}e_n.$\\ \hline
\end{tabular}
\end{table}
In both cases $\r_{e_{n+1}}$ and $\r_{e_{n+2}}$ restricted to the nilradical do not change, but we assign $\frac{a_{2,1}+a_{4,1}-b_{2,1}}{2}:=a_{2,1},$
make corresponding changes in $\r_{e_{n+1}}$ and in the remaining brackets below:
   \allowdisplaybreaks
\begin{equation}
\begin{array}{l}
\displaystyle  \nonumber (1)[e_{5},e_{1}]=-e_1-a_{2,1}e_2+2e_3,
[e_{5},e_3]=-a_{2,3}e_2+e_3, [e_{5},e_4]=e_2,[e_{5},e_{5}]=c_{2,5}e_2,\\
\displaystyle[e_{5},e_{6}]=c_{2,6}e_2+c_{2,5}e_4,[e_{6},e_1]=-e_1+\left(a_{2,3}-a_{2,1}\right)e_2+e_3,[e_{6},e_4]=e_2-e_4,\\
\displaystyle [e_{6},e_{5}]=c_{2,5}\left(e_2-e_{4}\right),[e_{6},e_{6}]=d_{2,6}e_2,(n=4).\\
\displaystyle (2) [e_{n+1},e_{1}]=-e_1-a_{2,1}e_2+2e_3,
[e_{n+1},e_3]=-a_{2,3}e_2+e_3-\sum_{k=5}^n{a_{k,3}e_k},\\
\displaystyle [e_{n+1},e_4]=e_2-\sum_{k=6}^n{a_{k-1,3}e_k},[e_{n+1},e_i]=(4-i)e_i-\sum_{k=i+2}^{n}{a_{k-i+3,3}e_k},[e_{n+1},e_{n+1}]=a_{5,3}e_2,\\
\displaystyle [e_{n+1},e_{n+2}]=c_{2,n+2}e_2+\sum_{k=4}^{n-1}{a_{k+1,3}e_{k}}+
c_{n,n+2}e_n,[e_{n+2},e_1]=-e_1+\left(a_{2,3}-a_{2,1}\right)e_2+e_3,\\
\displaystyle [e_{n+2},e_3]=-\sum_{k=5}^n{a_{k,3}e_k}, [e_{n+2},e_4]=e_2-e_4-\sum_{k=6}^n{a_{k-1,3}e_k},[e_{n+2},e_{i}]=(3-i)e_i-\sum_{k=i+2}^{n}{a_{k-i+3,3}e_k},\\
\displaystyle [e_{n+2},e_{n+1}]=a_{5,3}e_2-\sum_{k=4}^{n-1}{a_{k+1,3}e_{k}}-c_{n,n+2}e_n, [e_{n+2},e_{n+2}]=d_{2,n+2}e_2,(5\leq i\leq n,n\geq5).\end{array} 
\end{equation}

Altogether the nilradical $\mathcal{L}^2$ $(\ref{L2}),$ the outer derivations $\r_{e_{n+1}}$ and $\r_{e_{n+2}}$
 written in the bracket notation and the remaining brackets given above define a continuous family of two solvable right Leibniz algebras
depending on the parameters. Then we apply the technique of ``absorption''  according to step $(iv)$ to the algebras in $(1)$ and $(2).$ 
\begin{itemize}[noitemsep, topsep=0pt]
\item First we apply the transformation $e^{\prime}_i=e_i,(1\leq i\leq n,n\geq4),e^{\prime}_{n+1}=e_{n+1}-c_{2,n+2}e_2,
e^{\prime}_{n+2}=e_{n+2}-d_{2,n+2}e_2.$
This transformation removes the coefficients $c_{2,n+2}$ and $d_{2,n+2}$ in front of $e_2$ in
$[e_{n+1},e_{n+2}]$ and $[e_{n+2},e_{n+2}],$ respectively, without affecting other entries.

\item Then we apply the transformation
$e^{\prime}_i=e_i,(1\leq i\leq n,n\geq4),e^{\prime}_{n+1}=e_{n+1}-a_{5,3}e_4,
e^{\prime}_{n+2}=e_{n+2}.$ It removes the coefficient $a_{5,3}$
in front of $e_2$ in $[e_{n+1},e_{n+1}]$ and $[e_{n+2},e_{n+1}]$. Besides it removes
$a_{5,3}$ and $-a_{5,3}$ in front of $e_4$ in $[e_{n+1},e_{n+2}]$ and  $[e_{n+2},e_{n+1}],$ respectively. 
This transformation affects the coefficients in front of $e_{k},(6\leq k\leq n-1)$
in $[e_{n+1},e_{n+2}]$ and $[e_{n+2},e_{n+1}],$ which we rename by $a_{k+1,3}-a_{5,3}a_{k-1,3}$ and 
$-a_{k+1,3}+a_{5,3}a_{k-1,3},$ respectively. It also affects the coefficients in front $e_n,(n\geq6)$ in
$[e_{n+1},e_{n+2}]$ and $[e_{n+2},e_{n+1}],$ which we rename back by $c_{n,n+2}$ and $-c_{n,n+2},$
respectively. It introduces $a_{5,3}$ and $-a_{5,3}$ in the $(5,1)^{st}$ position in $\r_{e_{n+1}}$ and $\L_{e_{n+1}},$
respectively.
\begin{remark}
If $n=4,$ then we change $a_{5,3}$ to $c_{2,5}$. In this case the transformation does not affect any other entries.
\end{remark}
\item
Applying the transformation $e^{\prime}_j=e_j,(1\leq j\leq n+1,n\geq6),
e^{\prime}_{n+2}=e_{n+2}+a_{6,3}e_5+\sum_{k=6}^{n-1}{\frac{A_{k+1,3}}{k-4}e_k},$
where $A_{k+1,3}:=a_{k+1,3}-\sum_{i=5}^6a_{i,3}a_{k-i+4,3}-\sum_{i=8}^k{\frac{A_{i-1,3}a_{k-i+5,3}}{i-6}},$
$(6\leq k\leq n-1,n\geq7)$ such that $a_{4,3}=0,$
we remove the coefficients $a_{6,3}$ and $-a_{6,3}$ in front of $e_5$ in $[e_{n+1},e_{n+2}]$
and $[e_{n+2},e_{n+1}],$ respectively. Besides we remove $a_{k+1,3}-a_{5,3}a_{k-1,3}$ and $-a_{k+1,3}+a_{5,3}a_{k-1,3}$
in front of $e_k,(6\leq k\leq n-1)$ in $[e_{n+1},e_{n+2}]$ and $[e_{n+2},e_{n+1}],$ respectively.
This transformation introduces $-a_{6,3}$ and $a_{6,3}$ in the $(6,1)^{st}$ position;
$\frac{A_{k+1,3}}{4-k}$ and $\frac{A_{k+1,3}}{k-4}$ in the $(k+1,1)^{st},(6\leq k\leq n-1)$ positions in $\r_{e_{n+2}}$
and $\L_{e_{n+2}},$ respectively. It also affects the coefficients in front $e_n,(n\geq7)$ in
$[e_{n+1},e_{n+2}]$ and $[e_{n+2},e_{n+1}],$ which we rename back by $c_{n,n+2}$ and $-c_{n,n+2},$
respectively.
\item Finally applying the transformation $e^{\prime}_i=e_i,(1\leq i\leq n+1,n\geq5),
e^{\prime}_{n+2}=e_{n+2}+\frac{c_{n,n+2}}{n-4}e_n,$
we remove $c_{n,n+2}$ and $-c_{n,n+2}$ in front of $e_n$
in $[e_{n+1},e_{n+2}]$ and $[e_{n+2},e_{n+1}],$ respectively,
without affecting other entries. We obtain that $\r_{e_{n+1}}$ and $\r_{e_{n+2}}$ are as follows:
 \end{itemize}
{$$\r_{e_{n+1}}=\left[\begin{smallmatrix}
 1 & 0 & 0 & 0&0&0&\cdots && 0&0 & 0\\
  a_{2,1} & 0 & a_{2,3}& 0 &0&0 & \cdots &&0  & 0& 0\\
  -2 & 0 & -1 & 0 & 0&0 &\cdots &&0 &0& 0\\
  0 & 0 &  0 & 0 &0 &0 &\cdots&&0 &0 & 0\\
 a_{5,3} & 0 & a_{5,3} & 0 & 1  &0&\cdots & &0&0 & 0\\
  0 & 0 &\boldsymbol{\cdot} & a_{5,3} & 0 &2&\cdots & &0&0 & 0\\
    0 & 0 &\boldsymbol{\cdot} & \boldsymbol{\cdot} & \ddots &0&\ddots &&\vdots&\vdots &\vdots\\
  \vdots & \vdots & \vdots &\vdots &  &\ddots&\ddots &\ddots &\vdots&\vdots & \vdots\\
 0 & 0 & a_{n-2,3}& a_{n-3,3}& \cdots&\cdots &a_{5,3}&0&n-6 &0& 0\\
0 & 0 & a_{n-1,3}& a_{n-2,3}& \cdots&\cdots &\boldsymbol{\cdot}&a_{5,3}&0 &n-5& 0\\
 0 & 0 & a_{n,3}& a_{n-1,3}& \cdots&\cdots &\boldsymbol{\cdot}&\boldsymbol{\cdot}&a_{5,3} &0& n-4
\end{smallmatrix}\right],$$}
$$\r_{e_{n+2}}=\left[\begin{smallmatrix}
 1 & 0 & 0 & 0&0&0&0&\cdots && 0&0 & 0\\
 a_{2,3} & 1 & a_{2,3}& 0 &0&0 & 0&\cdots &&0  & 0& 0\\
  -1 & 0 & 0 & 0 & 0&0 &0&\cdots &&0 &0& 0\\
  0 & 0 &  0 & 1 &0 &0 &0&\cdots&&0 &0 & 0\\
  0 & 0 & a_{5,3} & 0 & 2  &0&0&\cdots & &0&0 & 0\\
 -a_{6,3} & 0 &\boldsymbol{\cdot} & a_{5,3} & 0 &3&0&\cdots & &0&0 & 0\\
  -\frac{1}{2}A_{7,3} & 0 &\boldsymbol{\cdot} & \boldsymbol{\cdot} & a_{5,3}&0 &4& &&\vdots&\vdots &\vdots\\
      -\frac{1}{3}A_{8,3} & 0 &\boldsymbol{\cdot} & \boldsymbol{\cdot} &\boldsymbol{\cdot} &a_{5,3}&0 &\ddots&&\vdots&\vdots &\vdots\\
  \vdots & \vdots & \vdots &\vdots &  &&\ddots&\ddots &\ddots &\vdots&\vdots & \vdots\\
  -\frac{1}{n-7}A_{n-2,3} & 0 & a_{n-2,3}& a_{n-3,3}& \cdots&\cdots&\cdots &a_{5,3}&0&n-5 &0& 0\\
    -\frac{1}{n-6}A_{n-1,3} & 0 & a_{n-1,3}& a_{n-2,3}& \cdots&\cdots&\cdots &\boldsymbol{\cdot}&a_{5,3}&0 &n-4& 0\\
 -\frac{1}{n-5}A_{n,3} & 0 & a_{n,3}&a_{n-1,3}& \cdots&\cdots &\cdots&\boldsymbol{\cdot}&\boldsymbol{\cdot}&a_{5,3} &0& n-3
\end{smallmatrix}\right],(n\geq4).$$
The remaining brackets are given below:
\begin{equation}
\left\{
\begin{array}{l}
\displaystyle  \nonumber [e_{n+1},e_{1}]=-e_1-a_{2,1}e_2+2e_3-a_{5,3}e_5,(a_{5,3}=0,\,when\,\,n=4),[e_{n+1},e_3]=-a_{2,3}e_2+e_3-\\
\displaystyle \sum_{k=5}^n{a_{k,3}e_k},[e_{n+1},e_4]=e_2-\sum_{k=6}^n{a_{k-1,3}e_k},[e_{n+1},e_i]=(4-i)e_i-\sum_{k=i+2}^{n}{a_{k-i+3,3}e_k},\\
\displaystyle [e_{n+2},e_1]=-e_1+\left(a_{2,3}-a_{2,1}\right)e_2+e_3+
\sum_{k=6}^n{\frac{A_{k,3}}{k-5}e_k};where\,A_{6,3}:=a_{6,3},\\
\displaystyle [e_{n+2},e_3]=-\sum_{k=5}^n{a_{k,3}e_k},[e_{n+2},e_4]=e_2-e_4-\sum_{k=6}^n{a_{k-1,3}e_k},\\
\displaystyle [e_{n+2},e_{i}]=(3-i)e_i-\sum_{k=i+2}^{n}{a_{k-i+3,3}e_k},(5\leq i\leq n,n\geq4).
\end{array} 
\right.
\end{equation} 
\noindent $(v)$ Finally we apply the change of basis transformation:
$e^{\prime}_1=e_1+\left(a_{2,1}-2a_{2,3}\right)e_2+\frac{a_{5,3}}{2}e_5+
\frac{2a_{6,3}}{3}e_6+\sum_{k=7}^{n}{\frac{C_{1,k}}{k-5}e_k},e^{\prime}_2=e_2,
e^{\prime}_3=e_3-a_{2,3}e_2,
e^{\prime}_i=e_i-\sum_{k=i+2}^n{\frac{B_{k-i+3,3}}{k-i}e_k},$
$(3\leq i\leq n-2),e^{\prime}_{n-1}=e_{n-1},e^{\prime}_{n}=e_{n},e^{\prime}_{n+1}=e_{n+1},e^{\prime}_{n+2}=e_{n+2},$
where $B_{j,3}:=a_{j,3}-\sum_{k=7}^j{\frac{B_{k-2,3}a_{j-k+5,3}}{k-5}},(5\leq j\leq n)$ and
$C_{1,k}:=\frac{2B_{k,3}}{k-3}-\sum_{i=5}^6{\frac{i-4}{i-3}a_{i,3}a_{k-i+3,3}}-\sum_{i=9}^k{\frac{C_{1,i-2}a_{k-i+5,3}}{i-7}},(7\leq k\leq n),$
where $a_{4,3}=0.$
\begin{remark} We have that $a_{5,3}=a_{6,3}=0$ in the transformation, when $(n=4)$
and $a_{6,3}=0,$ when $(n=5).$
\end{remark}
 We summarize a result 
in the following theorem: 
 \begin{theorem}\label{RCodim2L2} There is one solvable
indecomposable right Leibniz algebra up to isomorphism with a codimension two nilradical
$\mathcal{L}^2,(n\geq4),$ which is given below:
\begin{equation}
\begin{array}{l}
\displaystyle  \nonumber \g_{n+2,1}: [e_1,e_{n+1}]=e_1-2e_3,
 [e_{i},e_{n+1}]=(i-4)e_i,[e_{n+1},e_{1}]=-e_1+2e_3,[e_{n+1},e_3]=e_3,\\
\displaystyle [e_{n+1},e_4]=e_2,[e_{n+1},e_{j}]=(4-j)e_j,[e_{1},e_{n+2}]=e_1-e_3,[e_2,e_{n+2}]=e_2,[e_i,e_{n+2}]=(i-3)e_i,\\
\displaystyle (3\leq i\leq n),[e_{n+2},e_1]=-e_1+e_3,[e_{n+2},e_4]=e_2-e_4,[e_{n+2},e_j]=(3-j)e_j,(5\leq j\leq n),\\
\displaystyle DS=[n+2,\,n,\,n-2,\,0],LS=[n+2,\,n,\,n,...].
\end{array} 
\end{equation} 
\end{theorem}

\subsection{Solvable indecomposable left Leibniz algebras with a nilradical $\mathcal{L}^2$}
\subsubsection{One dimensional left solvable extensions of $\mathcal{L}^2$}
Classification follows the same steps, but with different cases per theorem. However we start with the same equation $(\ref{BRLeibniz})$.
\begin{theorem}\label{TheoremLL2} Set $a_{1,1}:=a$ and $b_{2,2}:=-b$ in $(\ref{BRLeibniz})$. To satisfy the left Leibniz identity, there are the following cases based on the conditions involving parameters,
each gives a continuous family of solvable Leibniz algebras:
\begin{enumerate}[noitemsep, topsep=0pt]
\item[(1)] If $a\neq0,b\neq a,(n=4)$ and $b\neq(4-n)a,a\neq0,b\neq a,(n\geq5),$ then
\begin{equation}
\left\{
\begin{array}{l}
\displaystyle  \nonumber [e_1,e_{n+1}]=ae_1+a_{2,1}e_2-(2a-b)e_3+A_{4,1}e_4+\sum_{k=5}^n{a_{k,1}e_k},
[e_3,e_{n+1}]=a_{2,3}e_2-\\
\displaystyle (a-b)e_3+A_{4,3}e_4+\sum_{k=5}^n{a_{k,3}e_k},[e_4,e_{n+1}]=(a-b)e_2+be_4+A_{4,3}e_5+\sum_{k=6}^n{a_{k-1,3}e_k},\\
\displaystyle [e_{j},e_{n+1}]=\left((j-4)a+b\right)e_{j}+A_{4,3}e_{j+1}+\sum_{k=j+2}^n{a_{k-j+3,3}e_k},(5\leq j\leq n),\\
\displaystyle[e_{n+1},e_{n+1}]=a_{2,n+1}e_2,[e_{n+1},e_1]=-ae_1+b_{2,1}e_2+(2a-b)e_3-A_{4,1}e_4-\sum_{k=5}^n{a_{k,1}e_k},\\
\displaystyle 
[e_{n+1},e_2]=-be_2,[e_{n+1},e_3]=b_{2,3}e_2+(a-b)e_3-A_{4,3}e_4-\sum_{k=5}^n{a_{k,3}e_k},\\
\displaystyle [e_{n+1},e_i]=\left((4-i)a-b\right)e_i-A_{4,3}e_{i+1}-\sum_{k=i+2}^n{a_{k-i+3,3}e_k},(4\leq i\leq n),\\
\displaystyle where\,\,A_{4,1}:=-a_{2,1}+\frac{(2a-b)(a_{2,3}+b_{2,3})-a\cdot b_{2,1}}{a-b}\,\,and\,\,A_{4,3}:=b_{2,3}+\frac{a\cdot a_{2,3}}{a-b},
\end{array} 
\right.
\end{equation} 
$$\L_{e_{n+1}}=\left[\begin{smallmatrix}
 -a & 0 & 0 & 0&0&&\cdots &0& \cdots&0 & 0&0 \\
  b_{2,1} & -b & b_{2,3}& 0 &0& & \cdots &0&\cdots  & 0& 0&0\\
  2a-b & 0 & a-b & 0 & 0& &\cdots &0&\cdots &0& 0&0 \\
  -A_{4,1} & 0 &  -A_{4,3} & -b &0 & &\cdots&0&\cdots &0 & 0&0\\
  -a_{5,1} & 0 & -a_{5,3} & -A_{4,3} & -a-b  &&\cdots &0 &\cdots&0 & 0&0 \\
 \boldsymbol{\cdot} & \boldsymbol{\cdot} & \boldsymbol{\cdot} & -a_{5,3} & -A_{4,3}  &\ddots& &\vdots &&\vdots & \vdots&\vdots \\
  \vdots & \vdots & \vdots &\vdots &\vdots  &\ddots& \ddots&\vdots &&\vdots & \vdots&\vdots \\
  -a_{i,1} & 0 & -a_{i,3} & -a_{i-1,3} & -a_{i-2,3}  &\cdots&-A_{4,3}& (4-i)a-b&\cdots&0 & 0&0\\
   \vdots  & \vdots  & \vdots &\vdots &\vdots&&\vdots &\vdots &&\vdots  & \vdots&\vdots \\
 -a_{n-1,1} & 0 & -a_{n-1,3}& -a_{n-2,3}& -a_{n-3,3}&\cdots &-a_{n-i+3,3} &-a_{n-i+2,3}&\cdots&-A_{4,3} &(5-n)a-b& 0\\
 -a_{n,1} & 0 & -a_{n,3}& -a_{n-1,3}& -a_{n-2,3}&\cdots &-a_{n-i+4,3}  &-a_{n-i+3,3}&\cdots&-a_{5,3} &-A_{4,3}& (4-n)a-b
\end{smallmatrix}\right].$$
\item[(2)] If $b:=(4-n)a,a\neq0,(n\geq5),$
then the brackets for the algebra are  
\begin{equation}
\left\{
\begin{array}{l}
\displaystyle  \nonumber [e_1,e_{n+1}]=ae_1+a_{2,1}e_2+(2-n)ae_3+A_{4,1}e_4+\sum_{k=5}^n{a_{k,1}e_k}, [e_3,e_{n+1}]=a_{2,3}e_2+\\
\displaystyle (3-n)ae_3+A_{4,3}e_4+\sum_{k=5}^n{a_{k,3}e_k},[e_4,e_{n+1}]=
(n-3)ae_2+(4-n)ae_4+A_{4,3}e_5+\\
\displaystyle \sum_{k=6}^n{a_{k-1,3}e_k},[e_{j},e_{n+1}]=\left(j-n\right)ae_{j}+A_{4,3}e_{j+1}+\sum_{k=j+2}^n{a_{k-j+3,3}e_k},(5\leq j\leq n),\\
\displaystyle [e_{n+1},e_{n+1}]=a_{2,n+1}e_2+a_{n,n+1}e_n,[e_{n+1},e_1]=-ae_1+b_{2,1}e_2+(n-2)ae_3-A_{4,1}e_4-\\
\displaystyle \sum_{k=5}^n{a_{k,1}e_k},[e_{n+1},e_2]=(n-4)ae_2,[e_{n+1},e_3]=b_{2,3}e_2+(n-3)ae_3-A_{4,3}e_4-\sum_{k=5}^n{a_{k,3}e_k},\\
\displaystyle [e_{n+1},e_i]=\left(n-i\right)ae_i-A_{4,3}e_{i+1}- \sum_{k=i+2}^n{a_{k-i+3,3}e_k},(4\leq i\leq n),\\
\displaystyle where\,\,A_{4,1}:=-a_{2,1}+\frac{(n-2)(a_{2,3}+b_{2,3})-b_{2,1}}{n-3}\,\,and\,\,A_{4,3}:=b_{2,3}+\frac{a_{2,3}}{n-3},
\end{array} 
\right.
\end{equation}
$$\L_{e_{n+1}}=\left[\begin{smallmatrix}
 -a & 0 & 0 & 0&0&&\cdots &0& \cdots&0 & 0&0 \\
  b_{2,1} & (n-4)a & b_{2,3}& 0 &0& & \cdots &0&\cdots  & 0& 0&0\\
  (n-2)a & 0 & (n-3)a & 0 & 0& &\cdots &0&\cdots &0& 0&0 \\
  -A_{4,1} & 0 &  -A_{4,3} & (n-4)a &0 & &\cdots&0&\cdots &0 & 0&0\\
  -a_{5,1} & 0 & -a_{5,3} & -A_{4,3} & (n-5)a  &&\cdots &0 &\cdots&0 & 0&0 \\
 \boldsymbol{\cdot} & \boldsymbol{\cdot} & \boldsymbol{\cdot} & -a_{5,3} & -A_{4,3}  &\ddots& &\vdots &&\vdots & \vdots&\vdots \\
  \vdots & \vdots & \vdots &\vdots &\vdots  &\ddots& \ddots&\vdots &&\vdots & \vdots&\vdots \\
  -a_{i,1} & 0 & -a_{i,3} & -a_{i-1,3} & -a_{i-2,3}  &\cdots&-A_{4,3}& (n-i)a&\cdots&0 & 0&0\\
   \vdots  & \vdots  & \vdots &\vdots &\vdots&&\vdots &\vdots &&\vdots  & \vdots&\vdots \\
 -a_{n-1,1} & 0 & -a_{n-1,3}& -a_{n-2,3}& -a_{n-3,3}&\cdots &-a_{n-i+3,3} &-a_{n-i+2,3}&\cdots&-A_{4,3} &a& 0\\
 -a_{n,1} & 0 & -a_{n,3}& -a_{n-1,3}& -a_{n-2,3}&\cdots &-a_{n-i+4,3}  &-a_{n-i+3,3}&\cdots&-a_{5,3} &-A_{4,3}& 0
\end{smallmatrix}\right].$$
 
\item[(3)] If $a=0$ and $b\neq0,(n\geq4),$ then
\begin{equation}
\left\{
\begin{array}{l}
\displaystyle  \nonumber [e_1,e_{n+1}]=a_{2,1}e_2+be_3+A_{4,1}e_4+\sum_{k=5}^n{a_{k,1}e_k}, [e_3,e_{n+1}]=
a_{2,3}e_2+be_3+\sum_{k=4}^n{a_{k,3}e_k},\\
\displaystyle [e_4,e_{n+1}]=-be_2+be_4+\sum_{k=5}^n{a_{k-1,3}e_k},
[e_{j},e_{n+1}]=be_{j}+\sum_{k=j+1}^n{a_{k-j+3,3}e_k},(5\leq j\leq n),\\
\displaystyle [e_{n+1},e_{n+1}]=a_{2,n+1}e_2,[e_{n+1},e_1]=b_{2,1}e_2-be_3-A_{4,1}e_4-\sum_{k=5}^n{a_{k,1}e_k},[e_{n+1},e_{2}]=-be_2,\\
\displaystyle[e_{n+1},e_3]=a_{4,3}e_2-be_3-\sum_{k=4}^n{a_{k,3}e_k},[e_{n+1},e_i]=-be_i-\sum_{k=i+1}^n{a_{k-i+3,3}e_k},(4\leq i\leq n),\\
\displaystyle where\,\,A_{4,1}:=-a_{2,1}+a_{4,3}+a_{2,3},
\end{array} 
\right.
\end{equation} 
$$\L_{e_{n+1}}=\left[\begin{smallmatrix}
 0 & 0 & 0 & 0&0&&\cdots &0& \cdots&0 & 0&0 \\
  b_{2,1} & -b & a_{4,3}& 0 &0& & \cdots &0&\cdots  & 0& 0&0\\
 -b & 0 & -b & 0 & 0& &\cdots &0&\cdots &0& 0&0 \\
  -A_{4,1} & 0 &  -a_{4,3} & -b &0 & &\cdots&0&\cdots &0 & 0&0\\
  -a_{5,1} & 0 & -a_{5,3} & -a_{4,3} & -b  &&\cdots &0 &\cdots&0 & 0&0 \\
 \boldsymbol{\cdot} & \boldsymbol{\cdot} & \boldsymbol{\cdot} & -a_{5,3} & -a_{4,3}  &\ddots& &\vdots &&\vdots & \vdots&\vdots \\
  \vdots & \vdots & \vdots &\vdots &\vdots  &\ddots& \ddots&\vdots &&\vdots & \vdots&\vdots \\
  -a_{i,1} & 0 & -a_{i,3} & -a_{i-1,3} & -a_{i-2,3}  &\cdots&-a_{4,3}& -b&\cdots&0 & 0&0\\
   \vdots  & \vdots  & \vdots &\vdots &\vdots&&\vdots &\vdots &&\vdots  & \vdots&\vdots \\
 -a_{n-1,1} & 0 & -a_{n-1,3}& -a_{n-2,3}& -a_{n-3,3}&\cdots &-a_{n-i+3,3} &-a_{n-i+2,3}&\cdots&-a_{4,3} &-b& 0\\
 -a_{n,1} & 0 & -a_{n,3}& -a_{n-1,3}& -a_{n-2,3}&\cdots &-a_{n-i+4,3}  &-a_{n-i+3,3}&\cdots&-a_{5,3} &-a_{4,3}& -b
\end{smallmatrix}\right].$$

   \allowdisplaybreaks
\item[(4)] If $b:=a,a\neq0,(n\geq4),$ then
\begin{equation}
\left\{
\begin{array}{l}
\displaystyle  \nonumber [e_1,e_{n+1}]=ae_1+a_{2,1}e_2-ae_3+\sum_{k=4}^n{a_{k,1}e_k},[e_{i},e_{n+1}]=\left(i-3\right)ae_{i}+\sum_{k=i+1}^n{a_{k-i+3,3}e_k},\\
\displaystyle 
(3\leq i\leq n),
[e_{n+1},e_{n+1}]=a_{2,n+1}e_2,[e_{n+1},e_1]=-ae_1+b_{2,3}e_2+ae_3-\sum_{k=4}^n{a_{k,1}e_k},\\
\displaystyle [e_{n+1},e_{2}]=-ae_2,[e_{n+1},e_3]=b_{2,3}e_2-\sum_{k=4}^n{a_{k,3}e_k},[e_{n+1},e_j]=\left(3-j\right)ae_j-\sum_{k=j+1}^n{a_{k-j+3,3}e_k},\\
\displaystyle (4\leq j\leq n),
\end{array} 
\right.
\end{equation} 
$$\L_{e_{n+1}}=\left[\begin{smallmatrix}
 -a & 0 & 0 & 0&0&&\cdots &0& \cdots&0 & 0&0 \\
  b_{2,3} & -a & b_{2,3}& 0 &0& & \cdots &0&\cdots  & 0& 0&0\\
  a & 0 & 0 & 0 & 0& &\cdots &0&\cdots &0& 0&0 \\
  -a_{4,1} & 0 &  -a_{4,3} & -a &0 & &\cdots&0&\cdots &0 & 0&0\\
  -a_{5,1} & 0 & -a_{5,3} & -a_{4,3} & -2a  &&\cdots &0 &\cdots&0 & 0&0 \\
 \boldsymbol{\cdot} & \boldsymbol{\cdot} & \boldsymbol{\cdot} & -a_{5,3} & -a_{4,3}  &\ddots& &\vdots &&\vdots & \vdots&\vdots \\
  \vdots & \vdots & \vdots &\vdots &\vdots  &\ddots& \ddots&\vdots &&\vdots & \vdots&\vdots \\
  -a_{i,1} & 0 & -a_{i,3} & -a_{i-1,3} & -a_{i-2,3}  &\cdots&-a_{4,3}& (3-i)a&\cdots&0 & 0&0\\
   \vdots  & \vdots  & \vdots &\vdots &\vdots&&\vdots &\vdots &&\vdots  & \vdots&\vdots \\
 -a_{n-1,1} & 0 & -a_{n-1,3}& -a_{n-2,3}& -a_{n-3,3}&\cdots &-a_{n-i+3,3} &-a_{n-i+2,3}&\cdots&-a_{4,3} &(4-n)a& 0\\
 -a_{n,1} & 0 & -a_{n,3}& -a_{n-1,3}& -a_{n-2,3}&\cdots &-a_{n-i+4,3}  &-a_{n-i+3,3}&\cdots&-a_{5,3} &-a_{4,3}& (3-n)a
\end{smallmatrix}\right].$$

\end{enumerate}
\end{theorem}
\begin{proof} 
\begin{enumerate}[noitemsep, topsep=2pt]
\item[(1)] Suppose $b\neq(4-n)a,a\neq0,b\neq a,(n\geq5).$ 
Then the proof is off-loaded to Table \ref{Left(L2)}. If $a\neq0,b\neq a,(n=4),$ then we consider applicable identities given in Table \ref{Left(L2)}. and calculations go the same way with one exception: 
in $12.$ we apply two identities:
$\L[e_5,e_3]=[\L(e_5),e_3]+[e_5,\L(e_3)]$ to obtain that $a_{1,5}=0$ and 
$\L[e_3,e_5]=[\L(e_3),e_5]+[e_3,\L(e_5)]$ to have that $b_{4,3}:=-a_{4,3}, a_{4,3}:=b_{2,3}+\frac{a\cdot a_{2,3}}{a-b}.$ 
\item[(2)]  Suppose $b:=(4-n)a,a\neq0,(n\geq5).$ We recalculate the identities given
in Table \ref{Left(L2)}., except the identity $15.$
\item[(3)] Suppose $a=0$ and $b\neq0.$ If $(n\geq5),$ then we apply the identities given in Table \ref{Left(L2)}.,
except $12.$ and $14,$ where we apply, respectively: $\L_{e_3}[e_{n+1},e_{n+1}]=[\L_{e_3}(e_{n+1}),e_{n+1}]+[e_{n+1},\L_{e_3}(e_{n+1})]$
and $\L[e_1,e_{n+1}]=[\L(e_1),e_{n+1}]+[e_1,\L(e_{n+1})]$ with the same results.
For $n=4,$ we apply the same identities as for $(n\geq5)$ except that in $12.$ we
apply two identities one after another: $\L[e_5,e_3]=[\L(e_5),e_3]+[e_5,\L(e_3)]$ and $\L_{e_3}[e_5,e_5]=[\L_{e_3}(e_5),e_5]+[e_5,\L_{e_3}(e_5)].$

\item[(4)] Suppose $b:=a,a\neq0.$ For $(n\geq5),$ we apply the same identities as in Table \ref{Left(L2)}.,
but we obtain slightly different conditions on the parameters in $12.$ and $14.$ The identity $12.$ gives that $a_{1,n+1}=0,b_{n,3}:=-a_{n,3},a_{2,3}=0$ and $14.$ that $b_{n,1}:=-a_{n,1},b_{2,1}:=b_{2,3}.$
For $n=4,$ the difference is that in $12.$ we apply the following two identities one after another:
$\L[e_5,e_3]=[\L(e_5),e_3]+[e_5,\L(e_3)]$ and $\L[e_3,e_5]=[\L(e_3),e_5]+[e_3,\L(e_5)].$
\end{enumerate}
\end{proof} 
\newpage
\begin{table}[h!]
\caption{Left Leibniz identities in the generic case in Theorem \ref{TheoremLL2}, ($n\geq5$).}
\label{Left(L2)}
\begin{tabular}{|l|p{2.4cm}|p{12cm}|}
\hline
\scriptsize Steps &\scriptsize Ordered triple &\scriptsize
Result\\ \hline
\scriptsize $1.$ &\scriptsize $\L_{e_1}\left([e_1,e_{n+1}]\right)$ &\scriptsize
$[e_{2},e_{n+1}]=0$
$\implies$ $a_{k,2}=0,(1\leq k\leq n).$\\ \hline
\scriptsize $2.$ &\scriptsize $\L_{e_{i}}\left([e_{n+1},e_{3}]\right)$ &\scriptsize
 $a_{1,i}=0,(3\leq i\leq n-1)\implies b_{1,3}=0$ $\implies$ 
 $[e_i,e_{n+1}]=\sum_{k=2}^n{a_{k,i}e_k},
 [e_{n+1},e_{3}]=\sum_{k=2}^n{b_{k,3}e_k}.$ \\ \hline
  \scriptsize $3.$ &\scriptsize $\L_{e_n}\left([e_{n+1},e_{3}]\right)$ &\scriptsize
 $a_{1,n}=0$ $\implies$ 
 $[e_{n},e_{n+1}]=\sum_{k=2}^n{a_{k,n}e_k}.$ Combining with $2.,$
 $[e_i,e_{n+1}]=\sum_{k=2}^n{a_{k,i}e_k},(3\leq i\leq n).$\\ \hline
  \scriptsize $4.$ &\scriptsize $\L[e_{1},e_{1}]$ &\scriptsize
 $b_{1,2}=0,b_{k,2}=0,(3\leq k\leq n),b_{3,1}:=-2b_{1,1}-b$ $\implies$ 
 $[e_{n+1},e_{2}]=-be_2,[e_{n+1},e_1]=b_{1,1}e_1+b_{2,1}e_2-(2b_{1,1}+b)e_3+\sum_{k=4}^n{b_{k,1}e_k}.$ \\ \hline
 \scriptsize $5.$ &\scriptsize $\L_{e_1}\left([e_{n+1},e_{1}]\right)$ &\scriptsize
 $b_{k-1,1}:=-a_{k-1,1},(5\leq k\leq n),b_{1,1}:=-a,a_{3,1}:=-2a+b$ $\implies$ 
 $[e_{1},e_{n+1}]=ae_1+a_{2,1}e_2-(2a-b)e_3+\sum_{k=4}^n{a_{k,1}e_k},
 [e_{n+1},e_1]=-ae_1+b_{2,1}e_2+(2a-b)e_3-\sum_{k=4}^{n-1}{a_{k,1}e_k}+b_{n,1}e_n.$ \\ \hline
 \scriptsize $6.$ &\scriptsize $\L_{e_1}\left([e_{3},e_{n+1}]\right)$ &\scriptsize
 $a_{k,4}:=a_{k-1,3},(5\leq k\leq n),a_{2,4}:=-a_{3,3},a_{3,4}=0,a_{4,4}:=a_{3,3}+a$ $\implies$ 
 $[e_{4},e_{n+1}]=-a_{3,3}e_2+(a+a_{3,3})e_4+\sum_{k=5}^n{a_{k-1,3}e_k}.$ \\ \hline
 \scriptsize $7.$ &\scriptsize $\L_{e_1}\left([e_{i},e_{n+1}]\right)$ &\scriptsize
 $a_{2,i+1}=a_{3,i+1}=a_{4,i+1}=0,a_{i+1,i+1}:=a_{3,3}+(i-2)a,a_{k,i+1}:=a_{k-1,i},(5\leq k\leq n,k\neq i+1,4\leq i\leq n-1)$ $\implies$ 
 $[e_{j},e_{n+1}]=\left(a_{3,3}+(j-3)a\right)e_j+\sum_{k=j+1}^n{a_{k-j+3,3}e_k},(5\leq j\leq n).$ \\ \hline
 \scriptsize $8.$ &\scriptsize $\L_{e_3}\left([e_{n+1},e_{i}]\right)$ &\scriptsize
 $b_{1,i}=0$ $\implies$ 
 $[e_{n+1},e_{i}]=\sum_{k=2}^n{b_{k,i}e_k},(4\leq i\leq n).$
 Combining with $2.,$ $[e_{n+1},e_{j}]=\sum_{k=2}^n{b_{k,j}e_k},(3\leq j\leq n)$. \\ \hline
 \scriptsize $9.$ &\scriptsize $\L_{e_1}\left([e_{n+1},e_{3}]\right)$ &\scriptsize
 $b_{2,4}:=a-b-b_{3,3},b_{4,4}:=b_{3,3}-a,b_{3,4}=0,b_{k,4}:=b_{k-1,3},(5\leq k\leq n)$ $\implies$ 
 $[e_{n+1},e_{4}]=\left(a-b-b_{3,3}\right)e_2+\left(b_{3,3}-a\right)e_4+\sum_{k=5}^n{b_{k-1,3}e_k}.$ \\ \hline
 \scriptsize $10.$ &\scriptsize $\L_{e_1}\left([e_{n+1},e_{i}]\right)$ &\scriptsize
 $b_{2,i+1}=b_{3,i+1}=b_{4,i+1}=0,b_{i+1,i+1}:=b_{3,3}-(i-2)a,b_{k,i+1}:=b_{k-1,i},(5\leq k\leq n,k\neq i+1,4\leq i\leq n-1)$ $\implies$ 
$[e_{n+1},e_j]=\left(b_{3,3}-(j-3)a\right)e_j+\sum_{k=j+1}^n{b_{k-j+3,3}e_k},(5\leq j\leq n).$ \\ \hline
  \scriptsize $11.$ &\scriptsize $\L_{e_3}\left([e_{n+1},e_{1}]\right)$ &\scriptsize
 $b_{3,3}:=a-b,a_{3,3}:=-a+b,b_{k-1,3}:=-a_{k-1,3},(5\leq k\leq n)$ $\implies$ 
$[e_3,e_{n+1}]=a_{2,3}e_2-(a-b)e_3+\sum_{k=4}^n{a_{k,3}e_k},
[e_{4},e_{n+1}]=(a-b)e_2+be_4+\sum_{k=5}^n{a_{k-1,3}e_k},
[e_{j},e_{n+1}]=\left((j-4)a+b\right)e_j+\sum_{k=j+1}^n{a_{k-j+3,3}e_k},(5\leq j\leq n),
[e_{n+1},e_{3}]=b_{2,3}e_2+(a-b)e_3-\sum_{k=4}^{n-1}{a_{k,3}e_k}+b_{n,3}e_n,
[e_{n+1},e_i]=\left((4-i)a-b\right)e_i-\sum_{k=i+1}^n{a_{k-i+3,3}e_k},(4\leq i\leq n).$\\ \hline
 \scriptsize $12.$ &\scriptsize $\L[e_{3},e_{n+1}]$ &\scriptsize
 $a_{1,n+1}=0;b_{n,3}:=-a_{n,3},(because\,\,a\neq0),a_{4,3}:=b_{2,3}+\frac{a\cdot a_{2,3}}{a-b},(a\neq b)$ $\implies$ 
 $[e_{n+1},e_{n+1}]=\sum_{k=2}^n{a_{k,n+1}e_k},
 [e_3,e_{n+1}]=a_{2,3}e_2-(a-b)e_3+A_{4,3}e_4+\sum_{k=5}^n{a_{k,3}e_k},
 [e_{4},e_{n+1}]=(a-b)e_2+be_4+A_{4,3}e_5+\sum_{k=6}^n{a_{k-1,3}e_k},
 [e_{j},e_{n+1}]=\left((j-4)a+b\right)e_j+A_{4,3}e_{j+1}+\sum_{k=j+2}^n{a_{k-j+3,3}e_k},(5\leq j\leq n),
 [e_{n+1},e_{3}]=b_{2,3}e_2+(a-b)e_3-A_{4,3}e_4-\sum_{k=5}^{n}{a_{k,3}e_k},
 [e_{n+1},e_i]=\left((4-i)a-b\right)e_i-A_{4,3}e_{i+1}-\sum_{k=i+2}^n{a_{k-i+3,3}e_k},(4\leq i\leq n),$
 where $A_{4,3}:=b_{2,3}+\frac{a\cdot a_{2,3}}{a-b}.$ \\ \hline
 \scriptsize $13.$ &\scriptsize $\L[e_{n+1},e_{1}]$ &\scriptsize
 $a_{k-1,n+1}=0,(4\leq k\leq n)$ $\implies$ 
 $[e_{n+1},e_{n+1}]=a_{2,n+1}e_2+a_{n,n+1}e_n.$ \\ \hline
  \scriptsize $14.$ &\scriptsize $\L_{e_1}\left([e_{n+1},e_{n+1}]\right)$ &\scriptsize
 $b_{n,1}:=-a_{n,1},a_{4,1}:=-a_{2,1}+\frac{(2a-b)(a_{2,3}+b_{2,3})-a\cdot b_{2,1}}{a-b}$ $\implies$ 
 $[e_{n+1},e_{1}]=-ae_1+b_{2,1}e_2+(2a-b)e_3-A_{4,1}e_4-\sum_{k=5}^{n}{a_{k,1}e_k},
 [e_{1},e_{n+1}]=ae_1+a_{2,1}e_2-(2a-b)e_3+A_{4,1}e_4+\sum_{k=5}^n{a_{k,1}e_k},$
 where $A_{4,1}:=-a_{2,1}+\frac{(2a-b)(a_{2,3}+b_{2,3})-a\cdot b_{2,1}}{a-b}.$ \\ \hline
 \scriptsize $15.$ &\scriptsize $\L[e_{n+1},e_{n+1}]$ &\scriptsize
 $a_{n,n+1}=0$ $\implies$ 
 $[e_{n+1},e_{n+1}]=a_{2,n+1}e_2.$ \\ \hline
\end{tabular}
\end{table}
\begin{theorem}\label{TheoremL(L2)Absorption} Applying the technique of ``absorption'' (see Section \ref{Solvable left Leibniz algebras}), we can further simplify the algebras
in each of the four cases in Theorem \ref{TheoremLL2} as follows:
\begin{enumerate}[noitemsep, topsep=0pt]
\item[(1)] If $a\neq0,b\neq a,(n=4)$ and $b\neq(4-n)a,a\neq0,b\neq a,(n\geq5),$ then
\begin{equation}
\left\{
\begin{array}{l}
\displaystyle  \nonumber [e_1,e_{n+1}]=ae_1+\left(a_{2,1}+A_{4,1}-A_{4,3}\right)e_2-(2a-b)e_3,
[e_3,e_{n+1}]=a_{2,3}e_2-(a-b)e_3+\\
\displaystyle \sum_{k=5}^n{a_{k,3}e_k},[e_4,e_{n+1}]=(a-b)e_2+be_4+\sum_{k=6}^n{a_{k-1,3}e_k},[e_{j},e_{n+1}]=\left((j-4)a+b\right)e_{j}+\\
\displaystyle \sum_{k=j+2}^n{a_{k-j+3,3}e_k},(5\leq j\leq n),[e_{n+1},e_{n+1}]=a_{2,n+1}e_2,[e_{n+1},e_1]=-ae_1+\left(b_{2,1}-A_{4,3}\right)e_2+\\
\displaystyle(2a-b)e_3,[e_{n+1},e_2]=-be_2,[e_{n+1},e_3]=\frac{a\cdot a_{2,3}}{b-a}e_2+(a-b)e_3-\sum_{k=5}^n{a_{k,3}e_k},\\
\displaystyle [e_{n+1},e_i]=\left((4-i)a-b\right)e_i-\sum_{k=i+2}^n{a_{k-i+3,3}e_k},(4\leq i\leq n),
\end{array} 
\right.
\end{equation} 
$$\L_{e_{n+1}}=\left[\begin{smallmatrix}
 -a & 0 & 0 & 0&0&0&\cdots && 0&0 & 0\\
  b_{2,1}-A_{4,3} & -b & \frac{a\cdot a_{2,3}}{b-a}& 0 &0&0 & \cdots &&0  & 0& 0\\
  2a-b & 0 & a-b & 0 & 0&0 &\cdots &&0 &0& 0\\
  0 & 0 &  0 & -b &0 &0 &\cdots&&0 &0 & 0\\
 0 & 0 & -a_{5,3} & 0 & -a-b  &0&\cdots & &0&0 & 0\\
  0 & 0 &\boldsymbol{\cdot} & -a_{5,3} & 0 &-2a-b&\cdots & &0&0 & 0\\
    0 & 0 &\boldsymbol{\cdot} & \boldsymbol{\cdot} & \ddots &0&\ddots &&\vdots&\vdots &\vdots\\
  \vdots & \vdots & \vdots &\vdots &  &\ddots&\ddots &\ddots &\vdots&\vdots & \vdots\\
 0 & 0 & -a_{n-2,3}& -a_{n-3,3}& \cdots&\cdots &-a_{5,3}&0&(6-n)a-b &0& 0\\
0 & 0 & -a_{n-1,3}& -a_{n-2,3}& \cdots&\cdots &\boldsymbol{\cdot}&-a_{5,3}&0 &(5-n)a-b& 0\\
 0 & 0 & -a_{n,3}& -a_{n-1,3}& \cdots&\cdots &\boldsymbol{\cdot}&\boldsymbol{\cdot}&-a_{5,3} &0& (4-n)a-b
\end{smallmatrix}\right].$$
\item[(2)] If $b:=(4-n)a,a\neq0,(n\geq5),$
then the brackets for the algebra are  
\begin{equation}
\left\{
\begin{array}{l}
\displaystyle  \nonumber [e_1,e_{n+1}]=ae_1+\left(a_{2,1}+A_{4,1}-A_{4,3}\right)e_2+(2-n)ae_3, [e_3,e_{n+1}]=a_{2,3}e_2+(3-n)ae_3+\\
\displaystyle \sum_{k=5}^n{a_{k,3}e_k},[e_4,e_{n+1}]=(n-3)ae_2+
(4-n)ae_4+ \sum_{k=6}^n{a_{k-1,3}e_k},[e_{j},e_{n+1}]=\left(j-n\right)ae_{j}+\\
\displaystyle \sum_{k=j+2}^n{a_{k-j+3,3}e_k},(5\leq j\leq n),[e_{n+1},e_{n+1}]=a_{n,n+1}e_n,[e_{n+1},e_1]=-ae_1+\left(b_{2,1}-A_{4,3}\right)e_2+\\
\displaystyle (n-2)ae_3,[e_{n+1},e_2]=(n-4)ae_2,[e_{n+1},e_3]=\frac{a_{2,3}}{3-n}e_2+(n-3)ae_3-\sum_{k=5}^n{a_{k,3}e_k},\\
\displaystyle [e_{n+1},e_i]=\left(n-i\right)ae_i-\sum_{k=i+2}^n{a_{k-i+3,3}e_k},(4\leq i\leq n),
\end{array} 
\right.
\end{equation}
$$\L_{e_{n+1}}=\left[\begin{smallmatrix}
 -a & 0 & 0 & 0&0&0&\cdots && 0&0 & 0\\
  b_{2,1}-A_{4,3} & (n-4)a & \frac{a_{2,3}}{3-n}& 0 &0&0 & \cdots &&0  & 0& 0\\
  (n-2)a & 0 & (n-3)a & 0 & 0&0 &\cdots &&0 &0& 0\\
  0 & 0 &  0 & (n-4)a &0 &0 &\cdots&&0 &0 & 0\\
 0 & 0 & -a_{5,3} & 0 & (n-5)a  &0&\cdots & &0&0 & 0\\
  0 & 0 &\boldsymbol{\cdot} & -a_{5,3} & 0 &(n-6)a&\cdots & &0&0 & 0\\
    0 & 0 &\boldsymbol{\cdot} & \boldsymbol{\cdot} & \ddots &0&\ddots &&\vdots&\vdots &\vdots\\
  \vdots & \vdots & \vdots &\vdots &  &\ddots&\ddots &\ddots &\vdots&\vdots & \vdots\\
 0 & 0 & -a_{n-2,3}& -a_{n-3,3}& \cdots&\cdots &-a_{5,3}&0&2a &0& 0\\
0 & 0 & -a_{n-1,3}& -a_{n-2,3}& \cdots&\cdots &\boldsymbol{\cdot}&-a_{5,3}&0 &a& 0\\
 0 & 0 & -a_{n,3}& -a_{n-1,3}& \cdots&\cdots &\boldsymbol{\cdot}&\boldsymbol{\cdot}&-a_{5,3} &0&0
\end{smallmatrix}\right].$$
 
\item[(3)] If $a=0$ and $b\neq0,(n\geq4),$ then
\begin{equation}
\left\{
\begin{array}{l}
\displaystyle  \nonumber [e_1,e_{n+1}]=a_{2,3}e_2+be_3, [e_3,e_{n+1}]=
a_{2,3}e_2+be_3+\sum_{k=5}^n{a_{k,3}e_k},[e_4,e_{n+1}]=-be_2+be_4+\\
\displaystyle \sum_{k=6}^n{a_{k-1,3}e_k},
[e_{j},e_{n+1}]=be_{j}+\sum_{k=j+2}^n{a_{k-j+3,3}e_k},(5\leq j\leq n),[e_{n+1},e_1]=b_{2,1}e_2-be_3,\\
\displaystyle [e_{n+1},e_{2}]=-be_2,[e_{n+1},e_i]=-be_i-\sum_{k=i+2}^n{a_{k-i+3,3}e_k},(3\leq i\leq n),
\end{array} 
\right.
\end{equation} 
$$\L_{e_{n+1}}=\left[\begin{smallmatrix}
 0 & 0 & 0 & 0&0&0&\cdots && 0&0 & 0\\
  b_{2,1} & -b & 0& 0 &0&0 & \cdots &&0  & 0& 0\\
 - b & 0 & -b & 0 & 0&0 &\cdots &&0 &0& 0\\
  0 & 0 &  0 & -b &0 &0 &\cdots&&0 &0 & 0\\
 0 & 0 & -a_{5,3} & 0 & -b  &0&\cdots & &0&0 & 0\\
  0 & 0 &\boldsymbol{\cdot} & -a_{5,3} & 0 &-b&\cdots & &0&0 & 0\\
    0 & 0 &\boldsymbol{\cdot} & \boldsymbol{\cdot} & \ddots &0&\ddots &&\vdots&\vdots &\vdots\\
  \vdots & \vdots & \vdots &\vdots &  &\ddots&\ddots &\ddots &\vdots&\vdots & \vdots\\
 0 & 0 & -a_{n-2,3}& -a_{n-3,3}& \cdots&\cdots &-a_{5,3}&0&-b &0& 0\\
0 & 0 & -a_{n-1,3}& -a_{n-2,3}& \cdots&\cdots &\boldsymbol{\cdot}&-a_{5,3}&0 &-b& 0\\
 0 & 0 & -a_{n,3}& -a_{n-1,3}& \cdots&\cdots &\boldsymbol{\cdot}&\boldsymbol{\cdot}&-a_{5,3} &0&-b
\end{smallmatrix}\right].$$

   \allowdisplaybreaks
\item[(4)] If $b:=a,a\neq0,(n\geq4),$ then
\begin{equation}
\left\{
\begin{array}{l}
\displaystyle  \nonumber [e_1,e_{n+1}]=ae_1+a_{2,1}e_2-ae_3,[e_{i},e_{n+1}]=\left(i-3\right)ae_{i}+\sum_{k=i+2}^n{a_{k-i+3,3}e_k},(3\leq i\leq n),\\
\displaystyle 
[e_{n+1},e_1]=-ae_1+b_{2,3}e_2+ae_3,[e_{n+1},e_{2}]=-ae_2,[e_{n+1},e_3]=b_{2,3}e_2-\sum_{k=5}^n{a_{k,3}e_k},\\
\displaystyle [e_{n+1},e_j]=\left(3-j\right)ae_j-\sum_{k=j+2}^n{a_{k-j+3,3}e_k},(4\leq j\leq n),
\end{array} 
\right.
\end{equation} 
$$\L_{e_{n+1}}=\left[\begin{smallmatrix}
 -a & 0 & 0 & 0&0&0&\cdots && 0&0 & 0\\
  b_{2,3} & -a & b_{2,3}& 0 &0&0 & \cdots &&0  & 0& 0\\
  a& 0 & 0 & 0 & 0&0 &\cdots &&0 &0& 0\\
  0 & 0 &  0 & -a &0 &0 &\cdots&&0 &0 & 0\\
 0 & 0 & -a_{5,3} & 0 & -2a  &0&\cdots & &0&0 & 0\\
  0 & 0 &\boldsymbol{\cdot} & -a_{5,3} & 0 &-3a&\cdots & &0&0 & 0\\
    0 & 0 &\boldsymbol{\cdot} & \boldsymbol{\cdot} & \ddots &0&\ddots &&\vdots&\vdots &\vdots\\
  \vdots & \vdots & \vdots &\vdots &  &\ddots&\ddots &\ddots &\vdots&\vdots & \vdots\\
 0 & 0 & -a_{n-2,3}& -a_{n-3,3}& \cdots&\cdots &-a_{5,3}&0&(5-n)a &0& 0\\
0 & 0 & -a_{n-1,3}& -a_{n-2,3}& \cdots&\cdots &\boldsymbol{\cdot}&-a_{5,3}&0 &(4-n)a& 0\\
 0 & 0 & -a_{n,3}& -a_{n-1,3}& \cdots&\cdots &\boldsymbol{\cdot}&\boldsymbol{\cdot}&-a_{5,3} &0& (3-n)a
\end{smallmatrix}\right].$$

\end{enumerate}
\end{theorem}
\begin{proof}
\begin{enumerate}
\item[(1)] Suppose $a\neq0,b\neq a,(n=4)$ and $b\neq(4-n)a,a\neq0,b\neq a,(n\geq5).$ The right multiplication
operator (not a derivation) restricted to the nilradical is given below:
$$\r_{e_{n+1}}=\left[\begin{smallmatrix}
 a & 0 & 0 & 0&0&&\cdots &0& \cdots&0 & 0&0 \\
 a_{2,1} & 0 & a_{2,3}& a-b &0& & \cdots &0&\cdots  & 0& 0&0\\
  -2a+b & 0 & -a+b & 0 & 0& &\cdots &0&\cdots &0& 0&0 \\
  A_{4,1} & 0 &  A_{4,3} & b &0 & &\cdots&0&\cdots &0 & 0&0\\
  a_{5,1} & 0 & a_{5,3} & A_{4,3} & a+b  &&\cdots &0 &\cdots&0 & 0&0 \\
 \boldsymbol{\cdot} & \boldsymbol{\cdot} & \boldsymbol{\cdot} & a_{5,3} & A_{4,3}  &\ddots& &\vdots &&\vdots & \vdots&\vdots \\
  \vdots & \vdots & \vdots &\vdots &\vdots  &\ddots& \ddots&\vdots &&\vdots & \vdots&\vdots \\
  a_{i,1} & 0 & a_{i,3} & a_{i-1,3} & a_{i-2,3}  &\cdots&A_{4,3}& (i-4)a+b&\cdots&0 & 0&0\\
   \vdots  & \vdots  & \vdots &\vdots &\vdots&&\vdots &\vdots &&\vdots  & \vdots&\vdots \\
 a_{n-1,1} & 0 & a_{n-1,3}& a_{n-2,3}& a_{n-3,3}&\cdots &a_{n-i+3,3} &a_{n-i+2,3}&\cdots&A_{4,3} &(n-5)a+b& 0\\
 a_{n,1} & 0 & a_{n,3}& a_{n-1,3}& a_{n-2,3}&\cdots &a_{n-i+4,3}  &a_{n-i+3,3}&\cdots&a_{5,3} &A_{4,3}& (n-4)a+b
\end{smallmatrix}\right].$$
 \begin{itemize}
\item The transformation $e^{\prime}_k=e_k,(1\leq k\leq n),e^{\prime}_{n+1}=e_{n+1}-A_{4,3}e_1$
removes $A_{4,3}$ in $\r_{e_{n+1}}$ and $-A_{4,3}$ in $\L_{e_{n+1}}$ from the $(i,i-1)^{st}$ positions, where $(4\leq i\leq n),$ 
 but it affects other entries as well,
such as
the entries in the $(2,1)^{st}$ position in $\r_{e_{n+1}}$ and $\L_{e_{n+1}},$
which we change to $a_{2,1}-A_{4,3}$ and $b_{2,1}-A_{4,3},$ respectively.
It also changes the entry in the $(2,3)^{rd}$ position in $\L_{e_{n+1}}$ to 
$\frac{a\cdot a_{2,3}}{b-a}.$
At the same time, it affects the coefficient in front of $e_2$ in the bracket $[e_{n+1},e_{n+1}],$ which we change back to $a_{2,n+1}$.
\item Then we apply the transformation $e^{\prime}_i=e_i,(1\leq i\leq n),e^{\prime}_{n+1}=e_{n+1}+A_{4,1}e_3+\sum_{k=4}^{n-1}a_{k+1,1}e_{k}$
to remove $A_{4,1},a_{k+1,1}$ in $\r_{e_{n+1}}$ and $-A_{4,1},-a_{k+1,1}$ in $\L_{e_{n+1}}$ from the entries in the $(4,1)^{st}$ and $(k+1,1)^{st}$
positions, where $(4\leq k\leq n-1).$ It changes the entry in the $(2,1)^{st}$ position in
$\r_{e_{n+1}}$ to $a_{2,1}+A_{4,1}-A_{4,3}$
and the coefficient in front of $e_2$ in $[e_{n+1},e_{n+1}],$ which
we rename back by $a_{2,n+1}.$
\end{itemize}
\allowdisplaybreaks
\item[(2)] Suppose $b:=(4-n)a,a\neq0,(n\geq5).$ The right multiplication
operator (not a derivation) restricted to the nilradical is given below:
$$\r_{e_{n+1}}=\left[\begin{smallmatrix}
 a & 0 & 0 & 0&0&&\cdots &0& \cdots&0 & 0&0 \\
 a_{2,1} & 0 & a_{2,3}& (n-3)a &0& & \cdots &0&\cdots  & 0& 0&0\\
  (2-n)a & 0 & (3-n)a & 0 & 0& &\cdots &0&\cdots &0& 0&0 \\
  A_{4,1} & 0 &  A_{4,3} & (4-n)a &0 & &\cdots&0&\cdots &0 & 0&0\\
  a_{5,1} & 0 & a_{5,3} & A_{4,3} & (5-n)a  &&\cdots &0 &\cdots&0 & 0&0 \\
 \boldsymbol{\cdot} & \boldsymbol{\cdot} & \boldsymbol{\cdot} & a_{5,3} & A_{4,3}  &\ddots& &\vdots &&\vdots & \vdots&\vdots \\
  \vdots & \vdots & \vdots &\vdots &\vdots  &\ddots& \ddots&\vdots &&\vdots & \vdots&\vdots \\
  a_{i,1} & 0 & a_{i,3} & a_{i-1,3} & a_{i-2,3}  &\cdots&A_{4,3}& (i-n)a&\cdots&0 & 0&0\\
   \vdots  & \vdots  & \vdots &\vdots &\vdots&&\vdots &\vdots &&\vdots  & \vdots&\vdots \\
 a_{n-1,1} & 0 & a_{n-1,3}& a_{n-2,3}& a_{n-3,3}&\cdots &a_{n-i+3,3} &a_{n-i+2,3}&\cdots&A_{4,3} &-a& 0\\
 a_{n,1} & 0 & a_{n,3}& a_{n-1,3}& a_{n-2,3}&\cdots &a_{n-i+4,3}  &a_{n-i+3,3}&\cdots&a_{5,3} &A_{4,3}& 0
\end{smallmatrix}\right].$$

 \begin{itemize}
\item The transformation $e^{\prime}_k=e_k,(1\leq k\leq n),e^{\prime}_{n+1}=e_{n+1}-A_{4,3}e_1$
removes $A_{4,3}$ and $-A_{4,3}$ in $\r_{e_{n+1}}$ and $\L_{e_{n+1}},$ respectively, from the $(i,i-1)^{st}$ positions, where $(4\leq i\leq n),$ 
 but other entries are affected as well,
such as
the entries in the $(2,1)^{st}$ position in $\r_{e_{n+1}}$ and $\L_{e_{n+1}},$
which we change to $a_{2,1}-A_{4,3}$ and $b_{2,1}-A_{4,3},$ respectively.
This transformation also changes the entry in the $(2,3)^{rd}$ position in $\L_{e_{n+1}}$ to 
$\frac{a_{2,3}}{3-n}.$
At the same time, it affects the coefficient in front of $e_2$ in the bracket $[e_{n+1},e_{n+1}],$ which we change back to $a_{2,n+1}$.
\item Then we apply the transformation $e^{\prime}_i=e_i,(1\leq i\leq n),e^{\prime}_{n+1}=e_{n+1}+A_{4,1}e_3+\sum_{k=4}^{n-1}a_{k+1,1}e_{k}$
to remove $A_{4,1},a_{k+1,1}$ in $\r_{e_{n+1}}$ and $-A_{4,1},-a_{k+1,1}$ in $\L_{e_{n+1}}$ from the entries in the $(4,1)^{st}$ and $(k+1,1)^{st}$
positions, where $(4\leq k\leq n-1).$ The transformation changes the entry in the $(2,1)^{st}$ position in
$\r_{e_{n+1}}$ to $a_{2,1}+A_{4,1}-A_{4,3}$
and the coefficient in front of $e_2$ in $[e_{n+1},e_{n+1}],$ which
we rename back by $a_{2,n+1}.$
\item Applying the transformation $e^{\prime}_j=e_j,(1\leq j\leq n),e^{\prime}_{n+1}=e_{n+1}-\frac{a_{2,n+1}}{(n-4)a}e_2,$ we
remove the coefficient $a_{2,n+1}$ in front of $e_2$ in $[e_{n+1},e_{n+1}]$ and we prove the result.
\end{itemize}
\item[(3)] Suppose $a=0$ and $b\neq0,(n\geq4).$ The right multiplication operator (not a derivation) restricted to the nilradical is given below: 
$$\r_{e_{n+1}}=\left[\begin{smallmatrix}
 0 & 0 & 0 & 0&0&&\cdots &0& \cdots&0 & 0&0 \\
 a_{2,1} & 0 & a_{2,3}& -b &0& & \cdots &0&\cdots  & 0& 0&0\\
  b & 0 & b & 0 & 0& &\cdots &0&\cdots &0& 0&0 \\
  A_{4,1} & 0 &  a_{4,3} & b &0 & &\cdots&0&\cdots &0 & 0&0\\
  a_{5,1} & 0 & a_{5,3} & a_{4,3} & b  &&\cdots &0 &\cdots&0 & 0&0 \\
 \boldsymbol{\cdot} & \boldsymbol{\cdot} & \boldsymbol{\cdot} & a_{5,3} & a_{4,3}  &\ddots& &\vdots &&\vdots & \vdots&\vdots \\
  \vdots & \vdots & \vdots &\vdots &\vdots  &\ddots& \ddots&\vdots &&\vdots & \vdots&\vdots \\
  a_{i,1} & 0 & a_{i,3} & a_{i-1,3} & a_{i-2,3}  &\cdots&a_{4,3}&b&\cdots&0 & 0&0\\
   \vdots  & \vdots  & \vdots &\vdots &\vdots&&\vdots &\vdots &&\vdots  & \vdots&\vdots \\
 a_{n-1,1} & 0 & a_{n-1,3}& a_{n-2,3}& a_{n-3,3}&\cdots &a_{n-i+3,3} &a_{n-i+2,3}&\cdots&a_{4,3} &b& 0\\
 a_{n,1} & 0 & a_{n,3}& a_{n-1,3}& a_{n-2,3}&\cdots &a_{n-i+4,3}  &a_{n-i+3,3}&\cdots&a_{5,3} &a_{4,3}& b
\end{smallmatrix}\right].$$
 \begin{itemize}
\item The transformation $e^{\prime}_k=e_k,(1\leq k\leq n),e^{\prime}_{n+1}=e_{n+1}-a_{4,3}e_1$ first removes 
$a_{4,3}$ from the $(2,3)^{rd}$ position in $\L_{e_{n+1}}$, then it
removes $a_{4,3}$ and $-a_{4,3}$ in $\r_{e_{n+1}}$ and $\L_{e_{n+1}},$ respectively, from the $(i,i-1)^{st}$ positions, where $(4\leq i\leq n).$ 
 The entries affected by the transformation are
in the $(2,1)^{st}$ positions in $\r_{e_{n+1}}$ and $\L_{e_{n+1}},$
which we change to $a_{2,1}-a_{4,3}$ and $b_{2,1}-a_{4,3},$ respectively. We assign $b_{2,1}-a_{4,3}:=b_{2,1}$.
At the same time, the transformation affects the coefficient in front of $e_2$ in the bracket $[e_{n+1},e_{n+1}],$ which we change back to $a_{2,n+1}$.
\item Applying the transformation $e^{\prime}_i=e_i,(1\leq i\leq n),e^{\prime}_{n+1}=e_{n+1}+A_{4,1}e_3+\sum_{k=4}^{n-1}a_{k+1,1}e_{k},$
we remove $A_{4,1},a_{k+1,1}$ in $\r_{e_{n+1}}$ and $-A_{4,1},-a_{k+1,1}$ in $\L_{e_{n+1}}$ from the entries in the $(4,1)^{st}$ and $(k+1,1)^{st}$
positions, where $(4\leq k\leq n-1).$ This transformation changes the entry in the $(2,1)^{st}$ position in
$\r_{e_{n+1}}$ to $a_{2,3}$
and the coefficient in front of $e_2$ in $[e_{n+1},e_{n+1}],$ which we rename back by $a_{2,n+1}.$
\item The transformation $e^{\prime}_j=e_j,(1\leq j\leq n),e^{\prime}_{n+1}=e_{n+1}+\frac{a_{2,n+1}}{b}e_2$ 
removes the coefficient $a_{2,n+1}$ in front of $e_2$ in $[e_{n+1},e_{n+1}]$ and we prove the result.
\end{itemize}
\item[(4)] Suppose $b:=a,a\neq0,(n\geq4).$ The right multiplication
operator (not a derivation) restricted to the nilradical is given below:
$$\r_{e_{n+1}}=\left[\begin{smallmatrix}
 a & 0 & 0 & 0&0&&\cdots &0& \cdots&0 & 0&0 \\
 a_{2,1} & 0 & 0& 0 &0& & \cdots &0&\cdots  & 0& 0&0\\
  -a & 0 & 0 & 0 & 0& &\cdots &0&\cdots &0& 0&0 \\
  a_{4,1} & 0 &  a_{4,3} & a &0 & &\cdots&0&\cdots &0 & 0&0\\
  a_{5,1} & 0 & a_{5,3} & a_{4,3} & 2a  &&\cdots &0 &\cdots&0 & 0&0 \\
 \boldsymbol{\cdot} & \boldsymbol{\cdot} & \boldsymbol{\cdot} & a_{5,3} & a_{4,3}  &\ddots& &\vdots &&\vdots & \vdots&\vdots \\
  \vdots & \vdots & \vdots &\vdots &\vdots  &\ddots& \ddots&\vdots &&\vdots & \vdots&\vdots \\
  a_{i,1} & 0 & a_{i,3} & a_{i-1,3} & a_{i-2,3}  &\cdots&a_{4,3}& (i-3)a&\cdots&0 & 0&0\\
   \vdots  & \vdots  & \vdots &\vdots &\vdots&&\vdots &\vdots &&\vdots  & \vdots&\vdots \\
 a_{n-1,1} & 0 & a_{n-1,3}& a_{n-2,3}& a_{n-3,3}&\cdots &a_{n-i+3,3} &a_{n-i+2,3}&\cdots&a_{4,3} &(n-4)a& 0\\
 a_{n,1} & 0 & a_{n,3}& a_{n-1,3}& a_{n-2,3}&\cdots &a_{n-i+4,3}  &a_{n-i+3,3}&\cdots&a_{5,3} &a_{4,3}& (n-3)a
\end{smallmatrix}\right].$$
\begin{itemize}
\item The transformation $e^{\prime}_k=e_k,(1\leq k\leq n),e^{\prime}_{n+1}=e_{n+1}-a_{4,3}e_1$
removes $a_{4,3}$ and $-a_{4,3}$ from the $(i,i-1)^{st},(4\leq i\leq n)$ positions in $\r_{e_{n+1}}$ and $\L_{e_{n+1}},$ respectively,
affecting other entries as well,
such as
the entries in the $(2,1)^{st}$ positions in $\r_{e_{n+1}}$ and $\L_{e_{n+1}}$ and in the $(2,3)^{rd}$ position in $\L_{e_{n+1}},$
which we change to $a_{2,1}-a_{4,3}$ in $\r_{e_{n+1}}$ and to $b_{2,3}-a_{4,3}$ in $\L_{e_{n+1}}$.
At the same time, it affects the coefficient in front of $e_2$ in the bracket $[e_{n+1},e_{n+1}],$ which we change back to $a_{2,n+1}$.
\item Then we apply the transformation $e^{\prime}_i=e_i,(1\leq i\leq n),e^{\prime}_{n+1}=e_{n+1}+\sum_{k=3}^{n-1}a_{k+1,1}e_{k}$
to remove $a_{k+1,1}$ in $\r_{e_{n+1}}$ and $-a_{k+1,1}$ in $\L_{e_{n+1}}$ from the entries in the $(k+1,1)^{st}$
positions, where $(3\leq k\leq n-1).$ It changes the entry in the $(2,1)^{st}$ position in
$\r_{e_{n+1}}$ to $a_{2,1}+a_{4,1}-a_{4,3}$
and the coefficient in front of $e_2$ in $[e_{n+1},e_{n+1}],$ which
we rename back by $a_{2,n+1}.$
\item We assign $a_{2,1}+a_{4,1}-a_{4,3}:=a_{2,1}$ and $b_{2,3}-a_{4,3}:=b_{2,3}.$
Applying the transformation $e^{\prime}_i=e_i,(1\leq i\leq n),e^{\prime}_{n+1}=e_{n+1}+\frac{a_{2,n+1}}{a}e_2,$ we
remove the coefficient $a_{2,n+1}$ in front of $e_2$ in $[e_{n+1},e_{n+1}].$
\end{itemize}
\end{enumerate}
\end{proof}
\allowdisplaybreaks
 \begin{theorem}\label{TheoremL(L2)Basis}There are four solvable
indecomposable left Leibniz algebras up to isomorphism with a codimension one nilradical
$\mathcal{L}^2,(n\geq4),$ which are given below:
\begin{equation}
\begin{array}{l}
\displaystyle  \nonumber (i)\,l_{n+1,1}: [e_1,e_{n+1}]=e_1+(a-2)e_3,
[e_3,e_{n+1}]=(a-1)e_3, [e_4,e_{n+1}]=(1-a)e_2+ae_4,\\
\displaystyle[e_{j},e_{n+1}]=\left(a+j-4\right)e_{j},(5\leq j\leq n),[e_{n+1},e_1]=-e_1+(2-a)e_3,[e_{n+1},e_2]=-ae_2,\\
\displaystyle[e_{n+1},e_i]=\left(4-i-a\right)e_i,(3\leq i\leq n);DS=[n+1,n,n-2,0],LS=[n+1,n,n,...],\\
\displaystyle (ii)\,l_{n+1,2}: [e_1,e_{n+1}]=e_1-2e_3+\delta e_5,
[e_3,e_{n+1}]=-e_3, [e_4,e_{n+1}]=e_2,[e_{j},e_{n+1}]=\left(j-4\right)e_{j},\\
\displaystyle [e_{n+1},e_1]=-e_1+2e_3-\delta e_5, [e_{n+1},e_i]=\left(4-i\right)e_i,(\delta=\pm1,3\leq i\leq n,5\leq j\leq n,n\geq5),\\
\displaystyle  DS=[n+1,n,n-2,0],LS=[n+1,n,n,...],\\
\displaystyle (iii)\,l_{n+1,3}:[e_1,e_{n+1}]=e_1+(2-n)e_3, [e_3,e_{n+1}]=(3-n)e_3,[e_4,e_{n+1}]=(n-3)e_2+
(4-n)e_4,\\
\displaystyle [e_{j},e_{n+1}]=\left(j-n\right)e_{j},(5\leq j\leq n-1),[e_{n+1},e_{n+1}]=\delta e_n,[e_{n+1},e_1]=-e_1+(n-2)e_3,\\
\displaystyle [e_{n+1},e_2]=(n-4)e_2,[e_{n+1},e_i]=\left(n-i\right)e_i,(\delta=\pm1,3\leq i\leq n-1,n\geq5),\\
\displaystyle  DS=[n+1,n,n-2,0],LS=[n+1,n,n,...],\\
\displaystyle \nonumber (iv)\,l_{n+1,4}: [e_1,e_{n+1}]=e_3, [e_3,e_{n+1}]=
e_3+\epsilon e_5+\sum_{k=6}^n{b_{k-5}e_k},\\
\displaystyle [e_4,e_{n+1}]=-e_2+e_4+\epsilon e_6+ \sum_{k=7}^n{b_{k-6}e_k},[e_{j},e_{n+1}]=e_{j}+\epsilon e_{j+2}+\sum_{k=j+3}^n{b_{k-j-2}e_k},
\end{array} 
\end{equation}  
\begin{equation}
\begin{array}{l}
\displaystyle \nonumber [e_{n+1},e_1]=-e_3,[e_{n+1},e_{2}]=-e_2,[e_{n+1},e_i]=-e_i-\epsilon e_{i+2}-\sum_{k=i+3}^n{b_{k-i-2}e_k},\\
\displaystyle (\epsilon=0,\pm1,3\leq i\leq n,5\leq j\leq n);DS=[n+1,n-1,0],LS=[n+1,n-1,n-1,...].
\end{array} 
\end{equation}   
\begin{remark}
We remember that the right hand-side of each bracket is the span of the basis vectors in the nilradical.
\end{remark}
\end{theorem}\begin{proof}
\begin{enumerate}
\allowdisplaybreaks
\item[(1)] Suppose $a\neq0,b\neq a,(n=4)$ and $b\neq(4-n)a,a\neq0,b\neq a,(n\geq5).$ The right (not a derivation) and the left (a derivation) multiplication operators restricted to the nilradical are given below:
$$\r_{e_{n+1}}=\left[\begin{smallmatrix}
 a & 0 & 0 & 0&0&0&\cdots && 0&0 & 0\\
 a_{2,1}+A_{4,1}-A_{4,3} & 0 & a_{2,3}& a-b &0&0 & \cdots &&0  & 0& 0\\
  -2a+b & 0 & -a+b & 0 & 0&0 &\cdots &&0 &0& 0\\
  0 & 0 &  0 & b &0 &0 &\cdots&&0 &0 & 0\\
 0 & 0 & a_{5,3} & 0 & a+b  &0&\cdots & &0&0 & 0\\
  0 & 0 &\boldsymbol{\cdot} & a_{5,3} & 0 &2a+b&\cdots & &0&0 & 0\\
    0 & 0 &\boldsymbol{\cdot} & \boldsymbol{\cdot} & \ddots &0&\ddots &&\vdots&\vdots &\vdots\\
  \vdots & \vdots & \vdots &\vdots &  &\ddots&\ddots &\ddots &\vdots&\vdots & \vdots\\
 0 & 0 & a_{n-2,3}& a_{n-3,3}& \cdots&\cdots &a_{5,3}&0&(n-6)a+b &0& 0\\
0 & 0 & a_{n-1,3}& a_{n-2,3}& \cdots&\cdots &\boldsymbol{\cdot}&a_{5,3}&0 &(n-5)a+b& 0\\
 0 & 0 & a_{n,3}& a_{n-1,3}& \cdots&\cdots &\boldsymbol{\cdot}&\boldsymbol{\cdot}&a_{5,3} &0& (n-4)a+b
\end{smallmatrix}\right],$$
$$\L_{e_{n+1}}=\left[\begin{smallmatrix}
 -a & 0 & 0 & 0&0&0&\cdots && 0&0 & 0\\
  b_{2,1}-A_{4,3} & -b & \frac{a\cdot a_{2,3}}{b-a}& 0 &0&0 & \cdots &&0  & 0& 0\\
  2a-b & 0 & a-b & 0 & 0&0 &\cdots &&0 &0& 0\\
  0 & 0 &  0 & -b &0 &0 &\cdots&&0 &0 & 0\\
 0 & 0 & -a_{5,3} & 0 & -a-b  &0&\cdots & &0&0 & 0\\
  0 & 0 &\boldsymbol{\cdot} & -a_{5,3} & 0 &-2a-b&\cdots & &0&0 & 0\\
    0 & 0 &\boldsymbol{\cdot} & \boldsymbol{\cdot} & \ddots &0&\ddots &&\vdots&\vdots &\vdots\\
  \vdots & \vdots & \vdots &\vdots &  &\ddots&\ddots &\ddots &\vdots&\vdots & \vdots\\
 0 & 0 & -a_{n-2,3}& -a_{n-3,3}& \cdots&\cdots &-a_{5,3}&0&(6-n)a-b &0& 0\\
0 & 0 & -a_{n-1,3}& -a_{n-2,3}& \cdots&\cdots &\boldsymbol{\cdot}&-a_{5,3}&0 &(5-n)a-b& 0\\
 0 & 0 & -a_{n,3}& -a_{n-1,3}& \cdots&\cdots &\boldsymbol{\cdot}&\boldsymbol{\cdot}&-a_{5,3} &0& (4-n)a-b
\end{smallmatrix}\right].$$

\begin{itemize}[noitemsep, topsep=0pt]
\allowdisplaybreaks 
\item We apply the transformation $e^{\prime}_1=e_1,e^{\prime}_2=e_2,e^{\prime}_i=e_i-\frac{a_{k-i+3,3}}{(k-i)a}e_k,(3\leq i\leq n-2,
i+2\leq k\leq n,n\geq5),e^{\prime}_{j}=e_{j},(n-1\leq j\leq n+1),$ where $k$ is fixed, renaming all the affected entries back.
This transformation removes $a_{5,3},a_{6,3},...,a_{n,3}$ in $\r_{e_{n+1}}$ and $-a_{5,3},-a_{6,3},...,-a_{n,3}$ in $\L_{e_{n+1}}.$
Besides it introduces the entries in the $(5,1)^{st},(6,1)^{st},...,(n,1)^{st}$ positions in $\r_{e_{n+1}}$ and $\L_{e_{n+1}},$ 
which we call by $a_{5,1},a_{6,1},...,a_{n,1}$ and $-a_{5,1},-a_{6,1},...,-a_{n,1},$ respectively.
\item Applying the transformation $e^{\prime}_1=e_1,e^{\prime}_2=e_2,
e^{\prime}_3=e_3+\frac{a_{2,3}}{b-a}e_2,e^{\prime}_{i}=e_{i},
(4\leq i\leq n+1,n\geq4),$
we remove $a_{2,3}$ and $\frac{a\cdot a_{2,3}}{b-a}$ from the $(2,3)^{rd}$ positions in $\r_{e_{n+1}}$ and $\L_{e_{n+1}},$
respectively. This transformation changes the entries in the $(2,1)^{st}$ positions in $\r_{e_{n+1}}$ and $\L_{e_{n+1}}$
to 
$\frac{a(b_{2,1}-b_{2,3}+a_{2,3})}{b-a}$ and $b_{2,1}-b_{2,3}+a_{2,3},$ respectively.

\item The transformation $e^{\prime}_1=e_1+\frac{b_{2,1}-b_{2,3}+a_{2,3}}{b-a}e_2,e^{\prime}_{i}=e_{i},
(2\leq i\leq n,n\geq4),e^{\prime}_{n+1}=e_{n+1}+\frac{a_{2,n+1}}{b-a}e_4+\sum_{k=5}^{n-1}a_{k+1,1}e_{k}$
removes $\frac{a(b_{2,1}-b_{2,3}+a_{2,3})}{b-a}$ and $b_{2,1}-b_{2,3}+a_{2,3}$ from the entries in the $(2,1)^{st}$ positions in $\r_{e_{n+1}}$ and $\L_{e_{n+1}},$
respectively. It also removes 
$a_{k+1,1}$ and $-a_{k+1,1}$ in $\r_{e_{n+1}}$ and $\L_{e_{n+1}},$ respectively, from the entries in the $(k+1,1)^{st}$
positions, where $(5\leq k\leq n-1)$ and the coefficient $a_{2,n+1}$ in front of $e_2$ in $[e_{n+1},e_{n+1}]$. This transformation affects the entries in the $(5,1)^{st}$ positions in $\r_{e_{n+1}}$ and $\L_{e_{n+1}},$
which we change back.
\item Then we scale $a$ to unity applying the transformation $e^{\prime}_i=e_i,(1\leq i\leq n,n\geq4),e^{\prime}_{n+1}=\frac{e_{n+1}}{a}.$
Renaming $\frac{b}{a}$ by b and $\frac{a_{5,1}}{a}$ by $c,$
 we obtain a continuous family of Leibniz algebras given below, where $c=0,$ when $n=4$:
\begin{equation}
\left\{
\begin{array}{l}
\displaystyle  \nonumber [e_1,e_{n+1}]=e_1+(b-2)e_3+ce_5,
[e_3,e_{n+1}]=(b-1)e_3, [e_4,e_{n+1}]=(1-b)e_2+be_4,\\
\displaystyle[e_{j},e_{n+1}]=\left(b+j-4\right)e_{j},(5\leq j\leq n),[e_{n+1},e_1]=-e_1+(2-b)e_3-ce_5,\\
\displaystyle[e_{n+1},e_2]=-be_2,[e_{n+1},e_i]=\left(4-i-b\right)e_i,(3\leq i\leq n,n\geq4),\\
\displaystyle (b\neq1,n=4\,\,or\,\,b\neq1,b\neq4-n,n\geq5).
\end{array} 
\right.
\end{equation} 
Then we have the following two cases:

(I) If $b\neq0,(n\geq5),$ then we apply the transformation
$e^{\prime}_1=e_1-\frac{c}{b}e_5,e^{\prime}_i=e_i,(2\leq i\leq n+1)$
to remove $c,-c$
from the $(5,1)^{st}$ positions in $\r_{e_{n+1}}$ and $\L_{e_{n+1}},$ respectively. We have the following continuous family of Leibniz
algebras:
\begin{equation}
\begin{array}{l}
\displaystyle  [e_1,e_{n+1}]=e_1+(b-2)e_3,
[e_3,e_{n+1}]=(b-1)e_3, [e_4,e_{n+1}]=(1-b)e_2+be_4,\\
\displaystyle[e_{j},e_{n+1}]=\left(b+j-4\right)e_{j},(5\leq j\leq n),[e_{n+1},e_1]=-e_1+(2-b)e_3,\\
\displaystyle[e_{n+1},e_2]=-be_2,[e_{n+1},e_i]=\left(4-i-b\right)e_i,(3\leq i\leq n,n\geq4),\\
\displaystyle (b\neq0,b\neq1,n=4\,\,or\,\,b\neq0,b\neq1,b\neq4-n,n\geq5).
\label{L(5.2.3)}
\end{array} 
\end{equation} 
(II) Suppose $b=0.$ If $c=0,$ then we have a limiting case of $(\ref{L(5.2.3)})$ with $b=0,(n\geq4).$
If $c\neq0$ in particular greater than zero or less than zero, respectively, then
we apply the transformation $e^{\prime}_1=\sqrt{\pm c}e_1,e^{\prime}_2=\pm ce_2,e^{\prime}_i=\left(\pm c\right)^{\frac{i-2}{2}}e_i,
(3\leq i\leq n,n\geq5),e^{\prime}_{n+1}=e_{n+1}$ to scale $c$ to $\pm1$ and we obtain the algebra $l_{n+1,2}$ given below:
\begin{equation}
\begin{array}{l}
\displaystyle  \nonumber [e_1,e_{n+1}]=e_1-2e_3+\delta e_5,
[e_3,e_{n+1}]=-e_3, [e_4,e_{n+1}]=e_2,[e_{j},e_{n+1}]=\left(j-4\right)e_{j},\\
\displaystyle[e_{n+1},e_1]=-e_1+2e_3-\delta e_5, [e_{n+1},e_i]=\left(4-i\right)e_i, (\delta=\pm1,3\leq i\leq n,5\leq j\leq n,n\geq5).
\end{array} 
\end{equation} 
\end{itemize}
\item[(2)] Suppose $b:=(4-n)a,a\neq0,(n\geq5).$ We have that the right (not a derivation) and the left (a derivation) 
multiplication operators restricted to the nilradical are below:
$$\r_{e_{n+1}}=\left[\begin{smallmatrix}
 a & 0 & 0 & 0&0&0&\cdots && 0&0 & 0\\
 a_{2,1}+A_{4,1}-A_{4,3} & 0 & a_{2,3}& (n-3)a &0&0 & \cdots &&0  & 0& 0\\
  (2-n)a & 0 & (3-n)a & 0 & 0&0 &\cdots &&0 &0& 0\\
  0 & 0 &  0 & (4-n)a &0 &0 &\cdots&&0 &0 & 0\\
 0 & 0 & a_{5,3} & 0 & (5-n)a  &0&\cdots & &0&0 & 0\\
  0 & 0 &\boldsymbol{\cdot} & a_{5,3} & 0 &(6-n)a&\cdots & &0&0 & 0\\
    0 & 0 &\boldsymbol{\cdot} & \boldsymbol{\cdot} & \ddots &0&\ddots &&\vdots&\vdots &\vdots\\
  \vdots & \vdots & \vdots &\vdots &  &\ddots&\ddots &\ddots &\vdots&\vdots & \vdots\\
 0 & 0 & a_{n-2,3}& a_{n-3,3}& \cdots&\cdots &a_{5,3}&0&-2a &0& 0\\
0 & 0 & a_{n-1,3}& a_{n-2,3}& \cdots&\cdots &\boldsymbol{\cdot}&a_{5,3}&0 &-a& 0\\
 0 & 0 & a_{n,3}& a_{n-1,3}& \cdots&\cdots &\boldsymbol{\cdot}&\boldsymbol{\cdot}&a_{5,3} &0& 0
\end{smallmatrix}\right],$$
$$\L_{e_{n+1}}=\left[\begin{smallmatrix}
 -a & 0 & 0 & 0&0&0&\cdots && 0&0 & 0\\
  b_{2,1}-A_{4,3} & (n-4)a & \frac{a_{2,3}}{3-n}& 0 &0&0 & \cdots &&0  & 0& 0\\
  (n-2)a & 0 & (n-3)a & 0 & 0&0 &\cdots &&0 &0& 0\\
  0 & 0 &  0 & (n-4)a &0 &0 &\cdots&&0 &0 & 0\\
 0 & 0 & -a_{5,3} & 0 & (n-5)a  &0&\cdots & &0&0 & 0\\
  0 & 0 &\boldsymbol{\cdot} & -a_{5,3} & 0 &(n-6)a&\cdots & &0&0 & 0\\
    0 & 0 &\boldsymbol{\cdot} & \boldsymbol{\cdot} & \ddots &0&\ddots &&\vdots&\vdots &\vdots\\
  \vdots & \vdots & \vdots &\vdots &  &\ddots&\ddots &\ddots &\vdots&\vdots & \vdots\\
 0 & 0 & -a_{n-2,3}& -a_{n-3,3}& \cdots&\cdots &-a_{5,3}&0&2a &0& 0\\
0 & 0 & -a_{n-1,3}& -a_{n-2,3}& \cdots&\cdots &\boldsymbol{\cdot}&-a_{5,3}&0 &a& 0\\
 0 & 0 & -a_{n,3}& -a_{n-1,3}& \cdots&\cdots &\boldsymbol{\cdot}&\boldsymbol{\cdot}&-a_{5,3} &0&0
\end{smallmatrix}\right].$$
\begin{itemize}[noitemsep, topsep=0pt]
\allowdisplaybreaks 
\item We apply the transformation $e^{\prime}_1=e_1,e^{\prime}_2=e_2,e^{\prime}_i=e_i-\frac{a_{k-i+3,3}}{(k-i)a}e_k,(3\leq i\leq n-2,
i+2\leq k\leq n),e^{\prime}_{j}=e_{j},(n-1\leq j\leq n+1),$ where $k$ is fixed, renaming all the affected entries the way they were.
This transformation removes $a_{5,3},a_{6,3},...,a_{n,3}$ in $\r_{e_{n+1}}$ and $-a_{5,3},-a_{6,3},...,-a_{n,3}$ in $\L_{e_{n+1}}.$
Besides it introduces the entries in the $(5,1)^{st},(6,1)^{st},...,(n,1)^{st}$ positions in $\r_{e_{n+1}}$ and $\L_{e_{n+1}},$ 
which we name by $a_{5,1},a_{6,1},...,a_{n,1}$ and $-a_{5,1},-a_{6,1},...,-a_{n,1},$ respectively.
\item Applying the transformation $e^{\prime}_1=e_1,e^{\prime}_2=e_2,
e^{\prime}_3=e_3+\frac{a_{2,3}}{(3-n)a}e_2,e^{\prime}_{i}=e_{i},
(4\leq i\leq n+1),$
we remove $a_{2,3}$ and $\frac{a_{2,3}}{3-n}$ from the $(2,3)^{rd}$ positions in $\r_{e_{n+1}}$ and $\L_{e_{n+1}},$
respectively. This transformation changes the entries in the $(2,1)^{st}$ positions in $\r_{e_{n+1}}$ and $\L_{e_{n+1}},$
respectively, to 
$\frac{b_{2,1}-b_{2,3}+a_{2,3}}{3-n}$ and $b_{2,1}-b_{2,3}+a_{2,3}.$
\item The transformation $e^{\prime}_1=e_1+\frac{b_{2,1}-b_{2,3}+a_{2,3}}{(3-n)a}e_2,e^{\prime}_{i}=e_{i},
(2\leq i\leq n),e^{\prime}_{n+1}=e_{n+1}+\frac{(3-n)a_{5,1}}{n-4}e_2+\sum_{k=4}^{n-1}a_{k+1,1}e_{k}$
removes $\frac{b_{2,1}-b_{2,3}+a_{2,3}}{3-n}$ and $b_{2,1}-b_{2,3}+a_{2,3}$ from the entries in the $(2,1)^{st}$ positions in $\r_{e_{n+1}}$ and $\L_{e_{n+1}},$
respectively. It also removes 
$a_{k+1,1}$ and $-a_{k+1,1}$ in $\r_{e_{n+1}}$ and $\L_{e_{n+1}},$ respectively, from the $(k+1,1)^{st}$
positions, where $(4\leq k\leq n-1)$.
\item Then we scale $a$ to unity applying the transformation $e^{\prime}_i=e_i,(1\leq i\leq n),e^{\prime}_{n+1}=\frac{e_{n+1}}{a}$.
We rename $\frac{a_{n,n+1}}{a^2}$ back by $a_{n,n+1}$ and we have a continuous family of Leibniz algebras:
\begin{equation}
\begin{array}{l}
\displaystyle  \nonumber [e_1,e_{n+1}]=e_1+(2-n)e_3, [e_3,e_{n+1}]=(3-n)e_3,[e_4,e_{n+1}]=(n-3)e_2+
(4-n)e_4,\\
\displaystyle [e_{j},e_{n+1}]=\left(j-n\right)e_{j},(5\leq j\leq n-1),[e_{n+1},e_{n+1}]=a_{n,n+1}e_n,[e_{n+1},e_1]=-e_1+(n-2)e_3,\\
\displaystyle [e_{n+1},e_2]=(n-4)e_2,[e_{n+1},e_i]=\left(n-i\right)e_i,(3\leq i\leq n-1,n\geq5).
\end{array} 
\end{equation}
If $a_{n,n+1}=0,$ then we have a limiting case of $(\ref{L(5.2.3)})$ with $b=4-n,(n\geq5).$ If $a_{n,n+1}\neq0$ in particular greater than zero or less than zero, respectively,
 then we apply the transformation
$e^{\prime}_i=\left(\pm a_{n,n+1}\right)^{\frac{i}{n-2}}e_i,(1\leq i\leq 2),e^{\prime}_k=\left(\pm a_{n,n+1}\right)^{\frac{k-2}{n-2}}e_k,(3\leq k\leq n),
e^{\prime}_{n+1}=e_{n+1}$
to scale $a_{n,n+1}$ to $\pm1$. We have the algebra $l_{n+1,3}$ given below:
\begin{equation}
\begin{array}{l}
\displaystyle  \nonumber [e_1,e_{n+1}]=e_1+(2-n)e_3, [e_3,e_{n+1}]=(3-n)e_3,[e_4,e_{n+1}]=(n-3)e_2+
(4-n)e_4,\\
\displaystyle [e_{j},e_{n+1}]=\left(j-n\right)e_{j},(5\leq j\leq n-1),[e_{n+1},e_{n+1}]=\delta e_n,[e_{n+1},e_1]=-e_1+(n-2)e_3,\\
\displaystyle [e_{n+1},e_2]=(n-4)e_2,[e_{n+1},e_i]=\left(n-i\right)e_i,(\delta=\pm1,3\leq i\leq n-1,n\geq5).
\end{array} 
\end{equation}
\end{itemize}
\item[(3)] Suppose $a=0$ and $b\neq0,(n\geq4).$ The right (not a derivation)
and the left (a derivation) multiplication operators restricted to the nilradical are given below:
$$\r_{e_{n+1}}=\left[\begin{smallmatrix}
 0 & 0 & 0 & 0&0&0&\cdots && 0&0 & 0\\
 a_{2,3} & 0 & a_{2,3}& -b &0&0 & \cdots &&0  & 0& 0\\
  b & 0 & b & 0 & 0&0 &\cdots &&0 &0& 0\\
  0 & 0 &  0 & b &0 &0 &\cdots&&0 &0 & 0\\
 0 & 0 & a_{5,3} & 0 & b  &0&\cdots & &0&0 & 0\\
  0 & 0 &\boldsymbol{\cdot} & a_{5,3} & 0 &b&\cdots & &0&0 & 0\\
    0 & 0 &\boldsymbol{\cdot} & \boldsymbol{\cdot} & \ddots &0&\ddots &&\vdots&\vdots &\vdots\\
  \vdots & \vdots & \vdots &\vdots &  &\ddots&\ddots &\ddots &\vdots&\vdots & \vdots\\
 0 & 0 & a_{n-2,3}& a_{n-3,3}& \cdots&\cdots &a_{5,3}&0&b &0& 0\\
0 & 0 & a_{n-1,3}& a_{n-2,3}& \cdots&\cdots &\boldsymbol{\cdot}&a_{5,3}&0 &b& 0\\
 0 & 0 & a_{n,3}& a_{n-1,3}& \cdots&\cdots &\boldsymbol{\cdot}&\boldsymbol{\cdot}&a_{5,3} &0&b
\end{smallmatrix}\right],$$
$$\L_{e_{n+1}}=\left[\begin{smallmatrix}
 0 & 0 & 0 & 0&0&0&\cdots && 0&0 & 0\\
  b_{2,1} & -b & 0& 0 &0&0 & \cdots &&0  & 0& 0\\
 - b & 0 & -b & 0 & 0&0 &\cdots &&0 &0& 0\\
  0 & 0 &  0 & -b &0 &0 &\cdots&&0 &0 & 0\\
 0 & 0 & -a_{5,3} & 0 & -b  &0&\cdots & &0&0 & 0\\
  0 & 0 &\boldsymbol{\cdot} & -a_{5,3} & 0 &-b&\cdots & &0&0 & 0\\
    0 & 0 &\boldsymbol{\cdot} & \boldsymbol{\cdot} & \ddots &0&\ddots &&\vdots&\vdots &\vdots\\
  \vdots & \vdots & \vdots &\vdots &  &\ddots&\ddots &\ddots &\vdots&\vdots & \vdots\\
 0 & 0 & -a_{n-2,3}& -a_{n-3,3}& \cdots&\cdots &-a_{5,3}&0&-b &0& 0\\
0 & 0 & -a_{n-1,3}& -a_{n-2,3}& \cdots&\cdots &\boldsymbol{\cdot}&-a_{5,3}&0 &-b& 0\\
 0 & 0 & -a_{n,3}& -a_{n-1,3}& \cdots&\cdots &\boldsymbol{\cdot}&\boldsymbol{\cdot}&-a_{5,3} &0&-b
\end{smallmatrix}\right].$$
\begin{itemize}[noitemsep, topsep=0pt]
\allowdisplaybreaks 
\item We apply the transformation $e^{\prime}_1=e_1,e^{\prime}_2=e_2,
e^{\prime}_3=e_3+\frac{a_{2,3}}{b}e_2,e^{\prime}_{i}=e_{i},
(4\leq i\leq n+1,n\geq4)$
to remove $a_{2,3}$ from the $(2,1)^{st}$ and the $(2,3)^{rd}$ positions in $\r_{e_{n+1}}$.
This transformation changes the entry in the $(2,1)^{st}$ position in $\L_{e_{n+1}}$
to $b_{2,1}+a_{2,3}.$
\item The transformation $e^{\prime}_1=e_1+\frac{b_{2,1}+a_{2,3}}{b}e_2,e^{\prime}_{i}=e_{i},
(2\leq i\leq n+1,n\geq4)$
removes $b_{2,1}+a_{2,3}$ from the entry in the $(2,1)^{st}$ position in $\L_{e_{n+1}}.$
\item To scale $b$ to unity, we apply the transformation $e^{\prime}_i=e_i,(1\leq i\leq n,n\geq4),e^{\prime}_{n+1}=\frac{e_{n+1}}{b}.$ 
Then we rename $\frac{a_{5,3}}{b},\frac{a_{6,3}}{b},...,\frac{a_{n,3}}{b}$ by $a_{5,3},a_{6,3},...,a_{n,3},$ respectively.
We obtain a Leibniz algebra
\begin{equation}
\begin{array}{l}
\displaystyle  \nonumber [e_1,e_{n+1}]=e_3, [e_3,e_{n+1}]=
e_3+\sum_{k=5}^n{a_{k,3}e_k},[e_4,e_{n+1}]=-e_2+e_4+ \sum_{k=6}^n{a_{k-1,3}e_k},\\
\displaystyle
[e_{j},e_{n+1}]=e_{j}+\sum_{k=j+2}^n{a_{k-j+3,3}e_k},(5\leq j\leq n),[e_{n+1},e_1]=-e_3,[e_{n+1},e_{2}]=-e_2,\\
\displaystyle [e_{n+1},e_i]=-e_i-\sum_{k=i+2}^n{a_{k-i+3,3}e_k},(3\leq i\leq n,n\geq4).
\end{array} 
\end{equation}
 If $a_{5,3}\neq0,(n\geq5),$ precisely, greater than zero and less than zero, respectively, then applying
the transformation $e^{\prime}_i=\left(\pm a_{5,3}\right)^{\frac{i}{2}}e_i,(1\leq i\leq 2),e^{\prime}_k=\left(\pm a_{5,3}\right)^{\frac{k-2}{2}}e_k,
(3\leq k\leq n),e^{\prime}_{n+1}=e_{n+1},$ we scale it to $\pm1.$ We also rename all the affected entries the way they were
and then we rename $a_{6,3},...,a_{n,3}$ by $b_1,...,b_{n-5},$ respectively.
We combine with the case when $a_{5,3}=0$ and obtain a Leibniz algebra $l_{n+1,4}$ given below:
\begin{equation}
\begin{array}{l}
\displaystyle  \nonumber [e_1,e_{n+1}]=e_3, [e_3,e_{n+1}]=
e_3+\epsilon e_5+\sum_{k=6}^n{b_{k-5}e_k},[e_4,e_{n+1}]=-e_2+e_4+\epsilon e_6+ \sum_{k=7}^n{b_{k-6}e_k},\\
\displaystyle
[e_{j},e_{n+1}]=e_{j}+\epsilon e_{j+2}+\sum_{k=j+3}^n{b_{k-j-2}e_k},(5\leq j\leq n),[e_{n+1},e_1]=-e_3,[e_{n+1},e_{2}]=-e_2,\\
\displaystyle [e_{n+1},e_i]=-e_i-\epsilon e_{i+2}-\sum_{k=i+3}^n{b_{k-i-2}e_k},(\epsilon=0,\pm1,3\leq i\leq n,n\geq4).
\end{array} 
\end{equation}
\begin{remark}
We notice that $\epsilon=0,$ when $n=4.$
\end{remark}
\end{itemize}
\item[(4)] Suppose $b:=a,a\neq0,(n\geq4).$ The right (not a derivation) and the left (a derivation) multiplication
operators restricted to the nilradical are below:
$$\r_{e_{n+1}}=\left[\begin{smallmatrix}
 a & 0 & 0 & 0&0&0&\cdots && 0&0 & 0\\
 a_{2,1} & 0 & 0& 0 &0&0 & \cdots &&0  & 0& 0\\
  -a & 0 & 0 & 0 & 0&0 &\cdots &&0 &0& 0\\
  0 & 0 &  0 & a &0 &0 &\cdots&&0 &0 & 0\\
 0 & 0 & a_{5,3} & 0 & 2a  &0&\cdots & &0&0 & 0\\
  0 & 0 &\boldsymbol{\cdot} & a_{5,3} & 0 &3a&\cdots & &0&0 & 0\\
    0 & 0 &\boldsymbol{\cdot} & \boldsymbol{\cdot} & \ddots &0&\ddots &&\vdots&\vdots &\vdots\\
  \vdots & \vdots & \vdots &\vdots &  &\ddots&\ddots &\ddots &\vdots&\vdots & \vdots\\
 0 & 0 & a_{n-2,3}& a_{n-3,3}& \cdots&\cdots &a_{5,3}&0&(n-5)a &0& 0\\
0 & 0 & a_{n-1,3}& a_{n-2,3}& \cdots&\cdots &\boldsymbol{\cdot}&a_{5,3}&0 &(n-4)a& 0\\
 0 & 0 & a_{n,3}& a_{n-1,3}& \cdots&\cdots &\boldsymbol{\cdot}&\boldsymbol{\cdot}&a_{5,3} &0& (n-3)a
\end{smallmatrix}\right],$$
$$\L_{e_{n+1}}=\left[\begin{smallmatrix}
 -a & 0 & 0 & 0&0&0&\cdots && 0&0 & 0\\
  b_{2,3} & -a & b_{2,3}& 0 &0&0 & \cdots &&0  & 0& 0\\
  a& 0 & 0 & 0 & 0&0 &\cdots &&0 &0& 0\\
  0 & 0 &  0 & -a &0 &0 &\cdots&&0 &0 & 0\\
 0 & 0 & -a_{5,3} & 0 & -2a  &0&\cdots & &0&0 & 0\\
  0 & 0 &\boldsymbol{\cdot} & -a_{5,3} & 0 &-3a&\cdots & &0&0 & 0\\
    0 & 0 &\boldsymbol{\cdot} & \boldsymbol{\cdot} & \ddots &0&\ddots &&\vdots&\vdots &\vdots\\
  \vdots & \vdots & \vdots &\vdots &  &\ddots&\ddots &\ddots &\vdots&\vdots & \vdots\\
 0 & 0 & -a_{n-2,3}& -a_{n-3,3}& \cdots&\cdots &-a_{5,3}&0&(5-n)a &0& 0\\
0 & 0 & -a_{n-1,3}& -a_{n-2,3}& \cdots&\cdots &\boldsymbol{\cdot}&-a_{5,3}&0 &(4-n)a& 0\\
 0 & 0 & -a_{n,3}& -a_{n-1,3}& \cdots&\cdots &\boldsymbol{\cdot}&\boldsymbol{\cdot}&-a_{5,3} &0& (3-n)a
\end{smallmatrix}\right].$$
\begin{itemize}[noitemsep, topsep=0pt]
\allowdisplaybreaks 
\item We apply the transformation $e^{\prime}_1=e_1,e^{\prime}_2=e_2,e^{\prime}_i=e_i-\frac{a_{k-i+3,3}}{(k-i)a}e_k,(3\leq i\leq n-2,
i+2\leq k\leq n,n\geq5),e^{\prime}_{j}=e_{j},(n-1\leq j\leq n+1),$ where $k$ is fixed renaming all the affected entries the way they were.
This transformation removes $a_{5,3},a_{6,3},...,a_{n,3}$ and $-a_{5,3},-a_{6,3},...,-a_{n,3}$ in $\r_{e_{n+1}}$ and $\L_{e_{n+1}},$
respectively.
Besides it introduces the entries in the $(5,1)^{st},(6,1)^{st},...,(n,1)^{st}$ positions in $\r_{e_{n+1}}$ and $\L_{e_{n+1}},$ 
which we call by $a_{5,1},a_{6,1},...,a_{n,1}$ and $-a_{5,1},-a_{6,1},...,-a_{n,1},$ respectively.
\item Applying the transformation $e^{\prime}_1=e_1,e^{\prime}_2=e_2,
e^{\prime}_3=e_3+\frac{b_{2,3}}{a}e_2,e^{\prime}_{i}=e_{i},
(4\leq i\leq n+1,n\geq4),$
we remove $b_{2,3}$ from the $(2,1)^{st}$ and the $(2,3)^{rd}$ positions in $\L_{e_{n+1}}.$
This transformation changes the entry in the $(2,1)^{st}$ position in $\r_{e_{n+1}}$
to $a_{2,1}+b_{2,3}.$
\item The transformation $e^{\prime}_1=e_1+\frac{a_{2,1}+b_{2,3}}{a}e_2,e^{\prime}_{i}=e_{i},
(2\leq i\leq n,n\geq4),e^{\prime}_{n+1}=e_{n+1}+\sum_{k=4}^{n-1}a_{k+1,1}e_{k}$
removes $a_{2,1}+b_{2,3}$ from the entry in the $(2,1)^{st}$ position in $\r_{e_{n+1}}$.
 It also removes 
$a_{k+1,1}$ and $-a_{k+1,1}$ in $\r_{e_{n+1}}$ and $\L_{e_{n+1}},$ respectively, from the entries in the $(k+1,1)^{st}$
positions, where $(4\leq k\leq n-1)$.
\item Then we scale $a$ to unity applying the transformation $e^{\prime}_i=e_i,(1\leq i\leq n,n\geq4),e^{\prime}_{n+1}=\frac{e_{n+1}}{a}$
and obtain a Leibniz algebra given below:
\begin{equation}
\begin{array}{l}
\displaystyle  \nonumber [e_1,e_{n+1}]=e_1-e_3,[e_{i},e_{n+1}]=\left(i-3\right)e_{i},[e_{n+1},e_1]=-e_1+e_3,[e_{n+1},e_{2}]=-e_2,\\
\displaystyle [e_{n+1},e_i]=\left(3-i\right)e_i,(3\leq i\leq n,n\geq4),
\end{array} 
\end{equation} 
which is a limiting case of $(\ref{L(5.2.3)})$ with $b=1,(n\geq4).$ Altogether $(\ref{L(5.2.3)})$
and all its limiting cases after replacing $b$ with $a$ give us a Leibniz algebra $l_{n+1,1}:$
\begin{equation}
\begin{array}{l}
\displaystyle \nonumber [e_1,e_{n+1}]=e_1+(a-2)e_3,
[e_3,e_{n+1}]=(a-1)e_3, [e_4,e_{n+1}]=(1-a)e_2+ae_4,\\
\displaystyle[e_{j},e_{n+1}]=\left(a+j-4\right)e_{j},(5\leq j\leq n),[e_{n+1},e_1]=-e_1+(2-a)e_3,\\
\displaystyle[e_{n+1},e_2]=-ae_2,[e_{n+1},e_i]=\left(4-i-a\right)e_i,(3\leq i\leq n,n\geq4).
\end{array} 
\end{equation}
\end{itemize}
\end{enumerate}
\end{proof}

\subsubsection{Two dimensional left solvable extensions of $\mathcal{L}^2$}
The non-zero inner derivations of $\mathcal{L}^2,(n\geq4)$ are
given by
 \[
\L_{e_1}=\left[\begin{smallmatrix}
0&0 & 0 & 0 & \cdots & 0 & 0  & 0 \\
1&0 & 1 & 0 & \cdots & 0 & 0  & 0 \\
0&0 & 0 & 0 & \cdots & 0 & 0 & 0 \\
 0&0 & -1 & 0 & \cdots & 0 & 0 & 0 \\
0& 0 & 0 & -1 & \cdots & 0 & 0 & 0 \\
\vdots& \vdots  & \vdots  & \vdots  & \ddots & \vdots & \vdots & \vdots\\
 0& 0 & 0 & 0&\cdots & -1 & 0 &0\\
  0& 0 & 0 & 0&\cdots & 0 & -1 &0
\end{smallmatrix}\right],\L_{e_3}=E_{4,1},
\L_{e_i}=E_{i+1,1}=\left[\begin{smallmatrix} 0 & 0&0&\cdots &  0 \\
 0& 0&0&\cdots &  0 \\
  0 & 0&0&\cdots &  0 \\
    0 & 0&0&\cdots &  0 \\
 \vdots &\vdots &\vdots& & \vdots\\
 1 &0&0& \cdots & 0\\
  \vdots &\vdots &\vdots& & \vdots\\
  \boldsymbol{\cdot} & 0&0&\cdots &  0
 \end{smallmatrix}\right]\,(4\leq i\leq n-1),\] where $E_{4,1}$ and $E_{i+1,1}$ are $n\times n$ matrices that has $1$ in the 
 $(4,1)^{st}$ and $(i+1,1)^{st}$
positions, respectively, and all other entries are zero.

 We refer to Section \ref{Section5.1.2} for more explanation.
We set $\left(
\begin{array}{c}
  a \\
 b
\end{array}\right)=\left(
\begin{array}{c}
  1\\
 0
\end{array}\right)$ and $\left(
\begin{array}{c}
 \alpha \\
 \beta
\end{array}\right)=\left(
\begin{array}{c}
  1\\
 2
\end{array}\right)$ in $\L_{e_{n+1}}$ and $\L_{e_{n+2}},$ respectively.
Therefore the vector space of outer derivations as the $n \times n$ matrices is as follows:
{ $$\L_{e_{n+1}}=\left[\begin{smallmatrix}
 -1 & 0 & 0 & 0&0&0&\cdots && 0&0 & 0\\
  b_{2,1}-b_{2,3}-a_{2,3} & 0 & -a_{2,3}& 0 &0&0 & \cdots &&0  & 0& 0\\
  2 & 0 & 1 & 0 & 0&0 &\cdots &&0 &0& 0\\
  0 & 0 &  0 & 0 &0 &0 &\cdots&&0 &0 & 0\\
 0 & 0 & -a_{5,3} & 0 & -1  &0&\cdots & &0&0 & 0\\
  0 & 0 &\boldsymbol{\cdot} & -a_{5,3} & 0 &-2&\cdots & &0&0 & 0\\
    0 & 0 &\boldsymbol{\cdot} & \boldsymbol{\cdot} & \ddots &0&\ddots &&\vdots&\vdots &\vdots\\
  \vdots & \vdots & \vdots &\vdots &  &\ddots&\ddots &\ddots &\vdots&\vdots & \vdots\\
 0 & 0 & -a_{n-2,3}& -a_{n-3,3}& \cdots&\cdots &-a_{5,3}&0&6-n &0& 0\\
0 & 0 & -a_{n-1,3}& -a_{n-2,3}& \cdots&\cdots &\boldsymbol{\cdot}&-a_{5,3}&0 &5-n& 0\\
 0 & 0 & -a_{n,3}& -a_{n-1,3}& \cdots&\cdots &\boldsymbol{\cdot}&\boldsymbol{\cdot}&-a_{5,3} &0&4-n
\end{smallmatrix}\right],$$}
{ $$\L_{e_{n+2}}=\left[\begin{smallmatrix}
 -1 & 0 & 0 & 0&0&0&\cdots && 0&0 & 0\\
  \beta_{2,1}-\beta_{2,3}+\alpha_{2,3} & -2 & \alpha_{2,3}& 0 &0&0 & \cdots &&0  & 0& 0\\
  0 & 0 & -1 & 0 & 0&0 &\cdots &&0 &0& 0\\
  0 & 0 &  0 & -2 &0 &0 &\cdots&&0 &0 & 0\\
 0 & 0 & -\alpha_{5,3} & 0 & -3  &0&\cdots & &0&0 & 0\\
  0 & 0 &\boldsymbol{\cdot} & -\alpha_{5,3} & 0 &-4&\cdots & &0&0 & 0\\
    0 & 0 &\boldsymbol{\cdot} & \boldsymbol{\cdot} & \ddots &0&\ddots &&\vdots&\vdots &\vdots\\
  \vdots & \vdots & \vdots &\vdots &  &\ddots&\ddots &\ddots &\vdots&\vdots & \vdots\\
 0 & 0 & -\alpha_{n-2,3}& -\alpha_{n-3,3}& \cdots&\cdots &-\alpha_{5,3}&0&4-n &0& 0\\
0 & 0 & -\alpha_{n-1,3}& -\alpha_{n-2,3}& \cdots&\cdots &\boldsymbol{\cdot}&-\alpha_{5,3}&0 &3-n& 0\\
 0 & 0 & -\alpha_{n,3}& -\alpha_{n-1,3}& \cdots&\cdots &\boldsymbol{\cdot}&\boldsymbol{\cdot}&-\alpha_{5,3} &0& 2-n
\end{smallmatrix}\right].$$}
We follow the steps of the ``General approach'' in Section \ref{Section5.1.2}:

\noindent $(i)$ We consider $\L_{[e_{n+1},e_{n+2}]},$ which is the same as $[\L_{e_{n+1}},\L_{e_{n+2}}]$ and deduce that $\alpha_{2,3}:=-a_{2,3},
\beta_{2,1}:=\beta_{2,3}+b_{2,3}-b_{2,1},
\alpha_{i,3}:=a_{i,3},(5\leq i\leq n).$ As a result, we make changes in the outer derivation $\L_{e_{n+2}}.$ It becomes as follows:
$$\L_{e_{n+2}}=\left[\begin{smallmatrix}
 -1 & 0 & 0 & 0&0&0&\cdots && 0&0 & 0\\
  b_{2,3}-b_{2,1}-a_{2,3} & -2 & -a_{2,3}& 0 &0&0 & \cdots &&0  & 0& 0\\
  0 & 0 & -1 & 0 & 0&0 &\cdots &&0 &0& 0\\
  0 & 0 &  0 & -2 &0 &0 &\cdots&&0 &0 & 0\\
 0 & 0 & -a_{5,3} & 0 & -3  &0&\cdots & &0&0 & 0\\
  0 & 0 &\boldsymbol{\cdot} & -a_{5,3} & 0 &-4&\cdots & &0&0 & 0\\
    0 & 0 &\boldsymbol{\cdot} & \boldsymbol{\cdot} & \ddots &0&\ddots &&\vdots&\vdots &\vdots\\
  \vdots & \vdots & \vdots &\vdots &  &\ddots&\ddots &\ddots &\vdots&\vdots & \vdots\\
 0 & 0 & -a_{n-2,3}& -a_{n-3,3}& \cdots&\cdots &-a_{5,3}&0&4-n &0& 0\\
0 & 0 & -a_{n-1,3}& -a_{n-2,3}& \cdots&\cdots &\boldsymbol{\cdot}&-a_{5,3}&0 &3-n& 0\\
 0 & 0 & -a_{n,3}& -a_{n-1,3}& \cdots&\cdots &\boldsymbol{\cdot}&\boldsymbol{\cdot}&-a_{5,3} &0& 2-n
\end{smallmatrix}\right]$$
We find the following commutators:
\allowdisplaybreaks
\begin{equation}
\left\{
\begin{array}{l}
\displaystyle  \nonumber \L_{[e_{1},e_{n+1}]}=\L_{e_1}-2\L_{e_3},\L_{[e_{2},e_{n+1}]}=0,\L_{[e_{i},e_{n+1}]}=(i-4)\L_{e_i}+\sum_{k=i+2}^{n-1}{a_{k-i+3,3}\L_{e_k}},\\
\displaystyle \L_{[e_{j},e_{n+1}]}=0,(n\leq j\leq n+1),\L_{[e_{n+2},e_{n+1}]}=-2\sum_{k=4}^{n-1}{a_{k+1,3}\L_{e_{k}}},\L_{[e_{1},e_{n+2}]}=\L_{e_1},\\
\displaystyle \L_{[e_{2},e_{n+2}]}=0,\L_{[e_{i},e_{n+2}]}=(i-2)\L_{e_i}+\sum_{k=i+2}^{n-1}{a_{k-i+3,3}\L_{e_k}},(3\leq i\leq n-1),\\
\displaystyle \L_{[e_{n},e_{n+2}]}=0,\L_{[e_{n+1},e_{n+2}]}=2\sum_{k=4}^{n-1}{a_{k+1,3}\L_{e_{k}}},\L_{[e_{n+2},e_{n+2}]}=0.
\end{array} 
\right.
\end{equation} 
\noindent $(ii)$ We include a linear combination of $e_2$ and $e_n:$
\begin{equation}
\left\{
\begin{array}{l}
\displaystyle  \nonumber [e_{1},e_{n+1}]=e_1+c_{2,1}e_2-2e_3+c_{n,1}e_n,[e_{2},e_{n+1}]=c_{2,2}e_2+c_{n,2}e_n,\\
\displaystyle[e_{i},e_{n+1}]=c_{2,i}e_2+(i-4)e_i+\sum_{k=i+2}^{n-1}{a_{k-i+3,3}e_k}+c_{n,i}e_n,[e_{j},e_{n+1}]=c_{2,j}e_2+c_{n,j}e_n,\\
\displaystyle [e_{n+2},e_{n+1}]=c_{2,n+2}e_2-2\sum_{k=4}^{n-1}{a_{k+1,3}e_{k}}+
c_{n,n+2}e_n,[e_{1},e_{n+2}]=e_1+d_{2,1}e_2+d_{n,1}e_n,\\
\displaystyle [e_{2},e_{n+2}]=d_{2,2}e_2+d_{n,2}e_n, [e_{i},e_{n+2}]=d_{2,i}e_2+(i-2)e_i+\sum_{k=i+2}^{n-1}{a_{k-i+3,3}e_k}+d_{n,i}e_n,\\
\displaystyle [e_{n},e_{n+2}]=d_{2,n}e_2+d_{n,n}e_n,[e_{n+1},e_{n+2}]=d_{2,n+1}e_2+2\sum_{k=4}^{n-1}{a_{k+1,3}e_{k}}+d_{n,n+1}e_n,\\
\displaystyle [e_{n+2},e_{n+2}]=d_{2,n+2}e_2+d_{n,n+2}e_n,(3\leq i\leq n-1,n\leq j\leq n+1).
\end{array} 
\right.
\end{equation} 
Besides we have the brackets from $\mathcal{L}^2$ and outer derivations $\L_{e_{n+1}}$ and $\L_{e_{n+2}}$.

\noindent $(iii)$ Case $(n\geq5)$ is shown in Table \ref{LeftCodimTwo(L2)}.
To prove case $(n=4),$ we apply as many identities as possible
of the case $(n\geq5)$. So we start with the identities $1.-3.,6.,$ where the identity $3.$ gives that $c_{2,4}:=1$ as well, but $c_{4,4}=0$
and the identity $6.$ gives that $d_{2,4}:=-1$ and $d_{4,4}:=2,$
where the first condition is the same. Further we apply the following identities one after another:
$\L_{e_5}[e_3,e_5]=[\L_{e_5}(e_3),e_5]+[e_3,\L_{e_5}(e_5)], \L_{e_5}[e_5,e_5]=[\L_{e_5}(e_5),e_5]+[e_5,\L_{e_5}(e_5)]$
(gives $[e_5,e_5]=c_{2,5}e_2$), $11.,12.$ (gives $d_{2,3}:=-a_{2,3},d_{4,3}=0$).  After that we continue with the identities:
$\L_{e_6}[e_6,e_5]=[\L_{e_6}(e_6),e_5]+[e_6,\L_{e_6}(e_5)]$ (gives $[e_6,e_6]=d_{2,6}e_2$),
$\L_{e_6}[e_5,e_5]=[\L_{e_6}(e_5),e_5]+[e_5,\L_{e_6}(e_5)]$(gives $c_{4,6}:=-2c_{2,5}$), $\L_{e_6}[e_5,e_6]=[\L_{e_6}(e_5),e_6]+[e_5,\L_{e_6}(e_6)]$
(gives $d_{2,5}:=-c_{2,5},d_{4,5}:=2c_{2,5}$). Finally we apply the identity $14.$ to obtain that $d_{2,1}:=b_{2,3}-b_{2,1}-a_{2,3}$
and $d_{4,1}=0.$

\begin{table}[htb]
\caption{Left Leibniz identities in the codimension two nilradical $\mathcal{L}^2,(n\geq5)$.}
\label{LeftCodimTwo(L2)}
\begin{tabular}{|l|p{2.7cm}|p{12cm}|}
\hline
\scriptsize Steps &\scriptsize Ordered triple &\scriptsize
Result\\ \hline
\scriptsize $1.$ &\scriptsize $\L_{e_1}\left([e_{1},e_{n+1}]\right)$ &\scriptsize
$[e_{2},e_{n+1}]=0$
$\implies$ $c_{2,2}=c_{n,2}=0.$\\ \hline
\scriptsize $2.$ &\scriptsize $\L_{e_1}\left([e_{1},e_{n+2}]\right)$ &\scriptsize
$[e_2,e_{n+2}]=0$
$\implies$ $d_{2,2}=d_{n,2}=0.$\\ \hline
\scriptsize $3.$ &\scriptsize $\L_{e_1}\left([e_{3},e_{n+1}]\right)$ &\scriptsize
$c_{2,4}:=1,c_{n,4}:=a_{n-1,3},$ where $a_{4,3}=0$ 
$\implies$  $[e_4,e_{n+1}]=e_2+\sum_{k=6}^{n}{a_{k-1,3}e_k}.$ \\ \hline
\scriptsize $4.$ &\scriptsize $\L_{e_1}\left([e_{i},e_{n+1}]\right)$ &\scriptsize
$c_{2,i+1}=0,c_{n,i+1}:=a_{n-i+2,3},(4\leq i\leq n-2),$ where $a_{4,3}=0$ 
$\implies$  $[e_{j},e_{n+1}]=\left(j-4\right)e_j+\sum_{k=j+2}^{n}{a_{k-j+3,3}e_k},(5\leq j\leq n-1).$ \\ \hline
\scriptsize $5.$ &\scriptsize $\L_{e_{1}}\left([e_{n-1},e_{n+1}]\right)$ &\scriptsize
$c_{2,n}=0,$ $c_{n,n}:=n-4$
$\implies$  $[e_{n},e_{n+1}]=\left(n-4\right)e_{n}.$ Altogether with $4.,$
  $[e_{i},e_{n+1}]=\left(i-4\right)e_i+\sum_{k=i+2}^{n}{a_{k-i+3,3}e_k},(5\leq i\leq n).$  \\ \hline
   \scriptsize $6.$ &\scriptsize $\L_{e_1}\left([e_{3},e_{n+2}]\right)$ &\scriptsize
$d_{2,4}:=-1,d_{n,4}:=a_{n-1,3},$ where $a_{4,3}=0$ 
$\implies$  $[e_{4},e_{n+2}]=-e_2+2e_4+\sum_{k=6}^{n}{a_{k-1,3}e_k}.$   \\ \hline  
   \scriptsize $7.$ &\scriptsize $\L_{e_1}\left([e_{i},e_{n+2}]\right)$ &\scriptsize
$d_{2,i+1}=0,d_{n,i+1}:=a_{n-i+2,3},(4\leq i\leq n-2),$ where $a_{4,3}=0$ 
$\implies$  $[e_{j},e_{n+2}]=\left(j-2\right)e_j+\sum_{k=j+2}^{n}{a_{k-j+3,3}e_k},(5\leq j\leq n-1).$   \\ \hline
\scriptsize $8.$ &\scriptsize $\L_{e_{1}}\left([e_{n-1},e_{n+2}]\right)$ &\scriptsize
 $d_{2,n}=0,$ $d_{n,n}:=n-2$ 
$\implies$  $[e_{n},e_{n+2}]=\left(n-2\right)e_{n}.$ Altogether with $7.,$
  $[e_{i},e_{n+2}]=\left(i-2\right)e_i+\sum_{k=i+2}^{n}{a_{k-i+3,3}e_k},(5\leq i\leq n).$   \\ \hline
  \scriptsize $9.$ &\scriptsize $\L_{e_{3}}\left([e_{n+1},e_{n+1}]\right)$ &\scriptsize
$c_{2,3}:=a_{2,3},$ $c_{n,3}:=a_{n,3}$ 
$\implies$  $[e_{3},e_{n+1}]=a_{2,3}e_2-e_3+\sum_{k=5}^n{a_{k,3}e_k}.$   \\ \hline
\scriptsize $10.$ &\scriptsize $\L_{e_{n+2}}\left([e_{n+1},e_{n+1}]\right)$ &\scriptsize
$c_{2,n+1}:=a_{5,3},c_{n,n+1}=0$
$\implies$ $[e_{n+1},e_{n+1}]=a_{5,3}e_2.$ \\ \hline
\scriptsize $11.$ &\scriptsize $\L_{e_1}\left([e_{n+1},e_{n+1}]\right)$ &\scriptsize
 $c_{2,1}:=b_{2,3}-b_{2,1}+a_{2,3}, c_{n,1}=0$
$\implies$  $[e_{1},e_{n+1}]=e_1+\left(b_{2,3}-b_{2,1}+a_{2,3}\right)e_2-2e_3.$\\ \hline
\scriptsize $12.$ &\scriptsize $\L_{e_{3}}\left([e_{n+2},e_{n+2}]\right)$ &\scriptsize
$d_{2,3}:=-a_{2,3},$ $d_{n,3}:=a_{n,3}$ 
$\implies$  $[e_{3},e_{n+2}]=-a_{2,3}e_2+e_3+\sum_{k=5}^n{a_{k,3}e_k}.$   \\ \hline
\scriptsize $13.$ &\scriptsize $\L_{e_{n+1}}\left([e_{n+2},e_{n+2}]\right)$ &\scriptsize
$d_{2,n+1}:=-a_{5,3},d_{n,n+2}=0$
$\implies$  $[e_{n+1},e_{n+2}]=-a_{5,3}e_2+2\sum_{k=4}^{n-1}{a_{k+1,3}e_k}+d_{n,n+1}e_n,[e_{n+2},e_{n+2}]=d_{2,n+2}e_2$. \\ \hline
\scriptsize $14.$ &\scriptsize $\L_{e_{1}}\left([e_{n+2},e_{n+2}]\right)$ &\scriptsize
$d_{2,1}:=b_{2,3}-b_{2,1}-a_{2,3},d_{n,1}=0$
$\implies$ $[e_{1},e_{n+2}]=e_1+\left(b_{2,3}-b_{2,1}-a_{2,3}\right)e_2.$\\ \hline
\scriptsize $15.$ &\scriptsize $\L_{e_{n+2}}\left([e_{n+1},e_{n+2}]\right)$ &\scriptsize
$d_{n,n+1}:=-c_{n,n+2}$
$\implies$ $[e_{n+1},e_{n+2}]=-a_{5,3}e_2+2\sum_{k=4}^{n-1}{a_{k+1,3}e_{k}}-
c_{n,n+2}e_n.$\\ \hline
\end{tabular}
\end{table}
In both cases $\L_{e_{n+1}}$ and $\L_{e_{n+2}}$ remain the same, but we assign $b_{2,1}-b_{2,3}-a_{2,3}:=b_{2,1},$
make corresponding changes in $\L_{e_{n+1}}$ and $\L_{e_{n+2}}$ and in the the remaining brackets below:
   \allowdisplaybreaks
\begin{equation}
\begin{array}{l}
\displaystyle  \nonumber (1)\,[e_{1},e_{5}]=e_1-b_{2,1}e_2-2e_3,
[e_{3},e_5]=a_{2,3}e_2-e_3, [e_{4},e_5]=e_2,[e_{5},e_{5}]=c_{2,5}e_2,\\
\displaystyle[e_{6},e_{5}]=c_{2,6}e_2-2c_{2,5}e_4,[e_{1},e_6]=e_1-\left(b_{2,1}+2a_{2,3}\right)e_2,
[e_3,e_6]=-a_{2,3}e_2+e_3,\\
\displaystyle [e_{4},e_6]=-e_2+2e_4,[e_{5},e_{6}]=-c_{2,5}\left(e_2-2e_{4}\right),[e_{6},e_{6}]=d_{2,6}e_2,(n=4).\\
\displaystyle  \nonumber (2)\,[e_{1},e_{n+1}]=e_1-b_{2,1}e_2-2e_3,
[e_{3},e_{n+1}]=a_{2,3}e_2-e_3+\sum_{k=5}^n{a_{k,3}e_k},\\
\displaystyle [e_{4},e_{n+1}]=e_2+\sum_{k=6}^n{a_{k-1,3}e_k},[e_{i},e_{n+1}]=(i-4)e_i+\sum_{k=i+2}^{n}{a_{k-i+3,3}e_k},[e_{n+1},e_{n+1}]=a_{5,3}e_2,\\
\displaystyle [e_{n+2},e_{n+1}]=c_{2,n+2}e_2-2\sum_{k=4}^{n-1}{a_{k+1,3}e_{k}}+
c_{n,n+2}e_n,[e_{1},e_{n+2}]=e_1-\left(b_{2,1}+2a_{2,3}\right)e_2,\\
\displaystyle [e_{3},e_{n+2}]=-a_{2,3}e_2+e_3+\sum_{k=5}^n{a_{k,3}e_k},[e_{4},e_{n+2}]=-e_2+2e_4+\sum_{k=6}^n{a_{k-1,3}e_k},\\
\displaystyle [e_{i},e_{n+2}]=(i-2)e_i+\sum_{k=i+2}^{n}{a_{k-i+3,3}e_k},[e_{n+1},e_{n+2}]=-a_{5,3}e_2+2\sum_{k=4}^{n-1}{a_{k+1,3}e_{k}}-c_{n,n+2}e_n,\\
\displaystyle [e_{n+2},e_{n+2}]=d_{2,n+2}e_2,(5\leq i\leq n,n\geq5).\end{array} 
\end{equation}
Altogether the nilradical $\mathcal{L}^2$ $(\ref{L2}),$ the outer derivations $\L_{e_{n+1}}$ and $\L_{e_{n+2}}$
and the remaining brackets give a continuous family of two solvable left Leibniz algebras
depending on the parameters.

\noindent$(iv)$ Then we continue with the ``absorption''.
\begin{itemize}[noitemsep, topsep=0pt]
\item First we apply the transformation $e^{\prime}_i=e_i,(1\leq i\leq n,n\geq4),e^{\prime}_{n+1}=e_{n+1}+\frac{c_{2,n+2}}{2}e_2,
e^{\prime}_{n+2}=e_{n+2}+\frac{d_{2,n+2}}{2}e_2.$
This transformation removes the coefficients $c_{2,n+2}$ and $d_{2,n+2}$ in front of $e_2$ in
$[e_{n+2},e_{n+1}]$ and $[e_{n+2},e_{n+2}],$ respectively, without affecting other entries.

\item Then we apply the transformation
$e^{\prime}_i=e_i,(1\leq i\leq n,n\geq4),e^{\prime}_{n+1}=e_{n+1}-a_{5,3}e_4,
e^{\prime}_{n+2}=e_{n+2}.$ It removes the coefficients $a_{5,3}$ and $-a_{5,3}$
in front of $e_2$ in $[e_{n+1},e_{n+1}]$ and $[e_{n+1},e_{n+2}],$ respectively. Besides it removes
$2a_{5,3}$ and $-2a_{5,3}$ in front of $e_4$ in $[e_{n+1},e_{n+2}]$ and  $[e_{n+2},e_{n+1}],$ respectively. 
This transformation affects the coefficients in front of $e_{k},(6\leq k\leq n-1)$
in $[e_{n+1},e_{n+2}]$ and $[e_{n+2},e_{n+1}],$ which we rename by $2a_{k+1,3}-a_{5,3}a_{k-1,3}$ and 
$-2a_{k+1,3}+a_{5,3}a_{k-1,3},$ respectively. It also affects the coefficients in front $e_n,(n\geq6)$ in
$[e_{n+2},e_{n+1}]$ and $[e_{n+1},e_{n+2}],$ which we rename back by $c_{n,n+2}$ and $-c_{n,n+2},$
respectively. Finally this transformation introduces $a_{5,3}$ and $-a_{5,3}$ in the $(5,1)^{st}$ position in $\r_{e_{n+1}}$ and $\L_{e_{n+1}},$
respectively.
\begin{remark}
If $n=4,$ then we change $a_{5,3}$ to $c_{2,5}$. In this case the transformation does not affect any other entries.
\end{remark}
\item
Applying the transformation $e^{\prime}_j=e_j,(1\leq j\leq n+1,n\geq6),
e^{\prime}_{n+2}=e_{n+2}+2a_{6,3}e_5+\sum_{k=6}^{n-1}{\frac{A_{k+1,3}}{k-4}e_k},$
where $A_{k+1,3}:=2a_{k+1,3}-\sum_{i=5}^6{(i-4)a_{i,3}a_{k-i+4,3}}-\sum_{i=8}^k{\frac{A_{i-1,3}a_{k-i+5,3}}{i-6}},$
$(6\leq k\leq n-1,n\geq7)$ such that $a_{4,3}=0,$
we remove the coefficients $2a_{6,3}$ and $-2a_{6,3}$ in front of $e_5$ in $[e_{n+1},e_{n+2}]$
and $[e_{n+2},e_{n+1}],$ respectively. Besides it removes $2a_{k+1,3}-a_{5,3}a_{k-1,3}$ and $-2a_{k+1,3}+a_{5,3}a_{k-1,3}$
in front of $e_k,(6\leq k\leq n-1)$ in $[e_{n+1},e_{n+2}]$ and $[e_{n+2},e_{n+1}],$ respectively.
This transformation introduces $-2a_{6,3}$ and $2a_{6,3}$ in the $(6,1)^{st}$ position;
$\frac{A_{k+1,3}}{4-k}$ and $\frac{A_{k+1,3}}{k-4}$ in the $(k+1,1)^{st},(6\leq k\leq n-1)$ positions in $\r_{e_{n+2}}$
and $\L_{e_{n+2}},$ respectively. It also affects the coefficients in front $e_n,(n\geq7)$ in
$[e_{n+2},e_{n+1}]$ and $[e_{n+1},e_{n+2}],$ which we rename back by $c_{n,n+2}$ and $-c_{n,n+2},$
respectively.
\item Finally applying the transformation $e^{\prime}_i=e_i,(1\leq i\leq n+1,n\geq5),
e^{\prime}_{n+2}=e_{n+2}-\frac{c_{n,n+2}}{n-4}e_n,$
we remove $c_{n,n+2}$ and $-c_{n,n+2}$ in front of $e_n$
in $[e_{n+2},e_{n+1}]$ and $[e_{n+1},e_{n+2}],$ respectively,
without affecting other entries. We obtain that $\L_{e_{n+1}}$ and $\L_{e_{n+2}}$ are as follows:
 \end{itemize}
{$$\L_{e_{n+1}}=\left[\begin{smallmatrix}
 -1 & 0 & 0 & 0&0&0&\cdots && 0&0 & 0\\
  b_{2,1} & 0 & -a_{2,3}& 0 &0&0 & \cdots &&0  & 0& 0\\
  2 & 0 & 1 & 0 & 0&0 &\cdots &&0 &0& 0\\
  0 & 0 &  0 & 0 &0 &0 &\cdots&&0 &0 & 0\\
 -a_{5,3} & 0 & -a_{5,3} & 0 & -1  &0&\cdots & &0&0 & 0\\
  0 & 0 &\boldsymbol{\cdot} & -a_{5,3} & 0 &-2&\cdots & &0&0 & 0\\
    0 & 0 &\boldsymbol{\cdot} & \boldsymbol{\cdot} & \ddots &0&\ddots &&\vdots&\vdots &\vdots\\
  \vdots & \vdots & \vdots &\vdots &  &\ddots&\ddots &\ddots &\vdots&\vdots & \vdots\\
 0 & 0 & -a_{n-2,3}& -a_{n-3,3}& \cdots&\cdots &-a_{5,3}&0&6-n &0& 0\\
0 & 0 & -a_{n-1,3}& -a_{n-2,3}& \cdots&\cdots &\boldsymbol{\cdot}&-a_{5,3}&0 &5-n& 0\\
 0 & 0 & -a_{n,3}& -a_{n-1,3}& \cdots&\cdots &\boldsymbol{\cdot}&\boldsymbol{\cdot}&-a_{5,3} &0&4-n
\end{smallmatrix}\right],$$}
$$\L_{e_{n+2}}=\left[\begin{smallmatrix}
 -1 & 0 & 0 & 0&0&0&0&\cdots && 0&0 & 0\\
 -b_{2,1}-2a_{2,3} & -2 & -a_{2,3}& 0 &0&0 & 0&\cdots &&0  & 0& 0\\
  0 & 0 & -1 & 0 & 0&0 &0&\cdots &&0 &0& 0\\
  0 & 0 &  0 & -2 &0 &0 &0&\cdots&&0 &0 & 0\\
  0 & 0 & -a_{5,3} & 0 & -3  &0&0&\cdots & &0&0 & 0\\
 2a_{6,3} & 0 &\boldsymbol{\cdot} & -a_{5,3} & 0 &-4&0&\cdots & &0&0 & 0\\
  \frac{1}{2}A_{7,3} & 0 &\boldsymbol{\cdot} & \boldsymbol{\cdot} & -a_{5,3}&0 &-5& &&\vdots&\vdots &\vdots\\
      \frac{1}{3}A_{8,3} & 0 &\boldsymbol{\cdot} & \boldsymbol{\cdot} &\boldsymbol{\cdot} &-a_{5,3}&0 &\ddots&&\vdots&\vdots &\vdots\\
  \vdots & \vdots & \vdots &\vdots &  &&\ddots&\ddots &\ddots &\vdots&\vdots & \vdots\\
  \frac{1}{n-7}A_{n-2,3} & 0 & -a_{n-2,3}& -a_{n-3,3}& \cdots&\cdots&\cdots &-a_{5,3}&0&4-n &0& 0\\
    \frac{1}{n-6}A_{n-1,3} & 0 & -a_{n-1,3}& -a_{n-2,3}& \cdots&\cdots&\cdots &\boldsymbol{\cdot}&-a_{5,3}&0 &3-n& 0\\
 \frac{1}{n-5}A_{n,3} & 0 & -a_{n,3}&-a_{n-1,3}& \cdots&\cdots &\cdots&\boldsymbol{\cdot}&\boldsymbol{\cdot}&-a_{5,3} &0& 2-n
\end{smallmatrix}\right],(n\geq4).$$
The remaining brackets are given below:
\begin{equation}
\left\{
\begin{array}{l}
\displaystyle  \nonumber [e_{1},e_{n+1}]=e_1-b_{2,1}e_2-2e_3+a_{5,3}e_5,(a_{5,3}=0,\,when\,\,n=4),[e_{3},e_{n+1}]=a_{2,3}e_2-e_3+\\
\displaystyle \sum_{k=5}^n{a_{k,3}e_k},[e_{4},e_{n+1}]=e_2+\sum_{k=6}^n{a_{k-1,3}e_k},[e_{i},e_{n+1}]=(i-4)e_i+\sum_{k=i+2}^{n}{a_{k-i+3,3}e_k},\\
\displaystyle [e_{1},e_{n+2}]=e_1-(b_{2,1}+2a_{2,3})e_2-
\sum_{k=6}^n{\frac{A_{k,3}}{k-5}e_k};where\,A_{6,3}:=2a_{6,3},\\
\displaystyle [e_{3},e_{n+2}]=-a_{2,3}e_2+e_3+\sum_{k=5}^n{a_{k,3}e_k},[e_{4},e_{n+2}]=-e_2+2e_4+\sum_{k=6}^n{a_{k-1,3}e_k},\\
\displaystyle [e_{i},e_{n+2}]=(i-2)e_i+\sum_{k=i+2}^{n}{a_{k-i+3,3}e_k},(5\leq i\leq n,n\geq4).
\end{array} 
\right.
\end{equation} 
\noindent $(v)$ Finally we apply the change of basis transformation:
$e^{\prime}_1=e_1-\left(b_{2,1}+2a_{2,3}\right)e_2+
\frac{2a_{6,3}}{3}e_6+\sum_{k=7}^{n}{\frac{C_{1,k}}{k-5}e_k},e^{\prime}_2=e_2,
e^{\prime}_3=e_3-a_{2,3}e_2,
e^{\prime}_i=e_i-\sum_{k=i+2}^n{\frac{B_{k-i+3,3}}{k-i}e_k},$
$(3\leq i\leq n-2),e^{\prime}_{n-1}=e_{n-1},e^{\prime}_{n}=e_{n},e^{\prime}_{n+1}=e_{n+1},e^{\prime}_{n+2}=e_{n+2},$
where $B_{j,3}:=a_{j,3}-\sum_{k=7}^j{\frac{B_{k-2,3}a_{j-k+5,3}}{k-5}},(5\leq j\leq n)$ and
$C_{1,k}:=\frac{2B_{k,3}}{k-3}-\frac{2}{3}a_{6,3}a_{k-3,3}-\sum_{i=9}^k{\frac{C_{1,i-2}a_{k-i+5,3}}{i-7}},(7\leq k\leq n),$
where $a_{4,3}=0.$
\begin{remark} We have that $a_{6,3}=0$ in the transformation, when $n=4$ and $n=5$.
\end{remark}
 We summarize a result 
in the following theorem: 
 \begin{theorem}\label{(L)Codim2L2} There is one solvable
indecomposable left Leibniz algebra up to isomorphism with a codimension two nilradical
$\mathcal{L}^2,(n\geq4),$ which is given below:
\begin{equation}
\begin{array}{l}
\displaystyle  \nonumber l_{n+2,1}: [e_1,e_{n+1}]=e_1-2e_3,[e_3,e_{n+1}]=-e_3,[e_4,e_{n+1}]=e_2,
 [e_{i},e_{n+1}]=(i-4)e_i,(5\leq i\leq n),\\
\displaystyle [e_{n+1},e_{1}]=-e_1+2e_3,[e_{n+1},e_{j}]=(4-j)e_j,[e_{1},e_{n+2}]=e_1,[e_3,e_{n+2}]=e_3,[e_4,e_{n+2}]=-e_2+2e_4,\\
\displaystyle 
[e_i,e_{n+2}]=(i-2)e_i,[e_{n+2},e_1]=-e_1,[e_{n+2},e_2]=-2e_2,[e_{n+2},e_j]=(2-j)e_j,(3\leq j\leq n),\\
\displaystyle DS=[n+2,n,n-2,0],LS=[n+2,n,n,...].
\end{array} 
\end{equation} 
\end{theorem}

\end{document}